\documentclass[bibliography=totoc]{article}
\usepackage{array}
\usepackage{tabularx}
\usepackage{hyperref}
\usepackage[english,french]{babel}
\usepackage{todonotes}

\usepackage[utf8]{inputenc} 
\usepackage[T1]{fontenc}
\usepackage{amsmath}
\usepackage{overpic}
\usepackage{amssymb}
\usepackage{mathrsfs}
\usepackage{lmodern}
\usepackage{amsfonts}
\usepackage{dsfont}
\usepackage{amssymb}
\usepackage{amsthm}
\usepackage{stmaryrd}
\usepackage{bbm}
\usepackage{xparse}
\usepackage{tikz}
\usepackage{subfiles}
\usepackage{float}
\usepackage{geometry}
\usepackage{enumerate}
\usepackage{xcolor}
\usepackage{thmtools}
\usepackage{thm-restate}
\usepackage{pgf,tikz}
\usetikzlibrary{arrows}
\usepackage{subcaption}

\hypersetup{
	colorlinks   = true, 
	urlcolor     = blue, 
	linkcolor    = blue, 
	citecolor    = red 
}

\theoremstyle{plain}
\newtheorem{thm}{Theorem}[section]
\newtheorem{lem}[thm]{Lemma}
\newtheorem{prop}[thm]{Proposition}

\theoremstyle{definition}

\newtheorem{rem}[thm]{Remark}

\numberwithin{equation}{section}

\newcommand{\Var}{\text{\normalfont Var}}
\newcommand{\conv}{\text{\normalfont conv}}

\newcommand{\tconv}{\text{\normalfont conv}}

\newcommand{\la}{\lambda}
\renewcommand{\P}{\mathbb{P}}
\newcommand{\E}{\mathbb{E}}
\newcommand{\R}{\mathbb{R}}

\newcommand{\Z}{\mathbb{Z}}

\renewcommand{\S}{\mathbb{S}}

\newcommand{\cP}{\mathcal{P}}
\newcommand{\cC}{\mathcal{C}}
\newcommand{\cF}{\mathcal{F}}
\newcommand{\cS}{\mathcal{S}}
\newcommand{\cA}{\mathcal{A}}
\newcommand{\cM}{\mathcal{M}}
\newcommand{\cV}{\mathcal{V}}
\newcommand{\cG}{\mathcal{G}}
\newcommand{\cE}{\mathcal{E}}
\newcommand{\cL}{\mathcal{L}}
\newcommand{\cH}{\mathcal{H}}

\newcommand{\cN}{\mathcal{N}}
\newcommand{\cT}{\mathcal{T}}

\newcommand{\Vol}{\text{\normalfont Vol}}

\newcommand{\dd}{\mathrm{d}}

\DeclareOldFontCommand{\sc}{\normalfont\scshape}{\mathsc}

\makeindex		    
\definecolor{zzttqq}{rgb}{0.6,0.2,0.}
\definecolor{grey}{HTML}{808080}
\definecolor{lightseagreen}{HTML}{20B2AA}
\definecolor{darkblue}{HTML}{00008B}
\definecolor{green}{HTML}{008000}
\definecolor{pink}{HTML}{DDA0A9}
\usepackage{pgf,tikz,pgfplots}
\pgfplotsset{compat=1.15}
\usepackage{mathrsfs}
\usetikzlibrary{arrows}

\newcommand\blfootnote[1]{%
  \begingroup
  \renewcommand\thefootnote{}\footnote{#1}%
  \addtocounter{footnote}{-1}%
  \endgroup
}

\title{
Limit theory for the first layers of the random convex hull peeling in a simple polytope
}
\date{\today}
\author{Pierre Calka, Gauthier Quilan}
\begin{document}
\selectlanguage{english}
\maketitle
\begin{abstract}
    The convex hull peeling of a point set consists in taking the convex hull, then removing the extreme points and iterating that procedure until no point remains. The boundary of each hull is called a layer. Following on from \cite{CQ1}, we study the first layers generated by the peeling procedure when the point set is chosen as a homogeneous Poisson point process inside a polytope when the intensity goes to infinity. We focus on some specific functionals, namely  the number of $k$-dimensional faces and the outer defect volume. Since the early works of R\'enyi and Sulanke, it is well known that both the techniques and the rates are completely different for the convex hull when the underlying convex body has a smooth boundary or when it is itself a polytope. We expect such dichotomy to extend to the further layers of the peeling. More precisely we provide asymptotic limits for their expectation and variance as well as a central limit theorem. In particular, as in the unit ball, the growth rates do not depend on the layer. The method builds upon previous constructions for the convex hull contained in \cite{BR10b} and \cite{CY3} and requires the assumption that the underlying polytope is simple.
\end{abstract}
\blfootnote{\textit{American Mathematical Society 2020 subject classifications.} Primary 60D05, 60F05; Secondary 52A22, 52A23, 60G55}
\section{Introduction}
Random polytopes constructed as convex hulls of a random point set have attracted attention for sixty years since the seminal work due to R\'enyi and Sulanke \cite{RS63,RS64}. For a general overview, we refer to the classical surveys \cite{B07,S08,R10,H13}. The basic question is: given $n$ random points which are independent and uniformly distributed in some fixed convex body $K$ in $\R^d$, $d\ge 2$, what can be said about the convex hull of these points? Even easier to deal with is the Poissonized counterpart of that question, when the random input is a Poisson point process with intensity measure equal to $\la$ times the Lebesgue measure in $K$. With a few notable exceptions \cite{W62,E65} including a recent one \cite{K21b}, most of the existing work tackles asymptotic issues, including the behavior for large $n$ or large $\la$ of functionals of the random convex hull which are either combinatorial, namely the number of $k$-dimensional faces, or geometric, namely the volume and intrinsic volumes. R\'enyi and Sulanke's results in dimension two have revealed a fundamental dichotomy, depending on whether the convex body $K$ has a smooth boundary with a ${\mathcal C}^2$ regularity or is a polytope itself. In the first case, for any $d\ge 2$, the expectations of the number of extreme points and subsequently of any number of $k$-dimensional faces are proved to grow polynomially fast, with a rate proportional to $n^{\frac{d-1}{d+1}}$ or $\la^{\frac{d-1}{d+1}}$ \cite{B92,S94,R05a}. In the second case, the obtained growth rate is logarithmic \cite{BB93,R05a}. More precisely, for $0\le k\le (d-1)$, let us denote by $f_k(\cdot)$ the number of $k$-dimensional faces (or \textit{$k$-faces} in short) of a polytope and by $\conv(\cdot)$ the convex hull of a point set. When $K$ is a convex polytope and $\cP_\la$ is a Poisson point process with intensity measure $\la$ times the Lebesgue measure in $K$, a Poissonized version of a result due to Reitzner \cite[Theorem]{R05a} says that
\begin{equation}\label{eq:esperance 1ere couche polytope}
\lim_{\la\to\infty}\log^{-(d-1)}(\la)\E[f_k(\conv(\cP_\la))]=c_d T(K)
\end{equation}
where $c_d$ is an explicit positive constant depending only on dimension $d$ and $T(K)$ is the number of towers of $K$, a tower being an increasing chain $F_0\subset F_1\ldots\subset F_{d-1}$ of $k$-faces $F_k$ of $K$. A similar asymptotic estimate has been obtained for the variance of $f_k(\conv(\cP_\la))$ when the convex polytope $K$ is assumed to be simple, in which case the number $T(K)$ is proportional to the number of vertices of $K$ \cite[Theorem 1.3]{CY3}. In other words, when $K$ is a simple polytope, we get
\begin{equation}\label{eq:variance 1ere couche polytope}
\lim_{\la\to\infty}\log^{-(d-1)}(\la)\Var(f_k(\conv(\cP_\la)))=c_d' f_0(K)
\end{equation}
where $c_d'$ is a positive constant depending only on dimension $d$. 

For a general polytope $K$, B\'ar\'any and Reitzner have obtained lower and upper bounds for the variance which match with the growth rate $\log^{d-1}(\la)$ up to multiplicative constants as well as a central limit theorem satisfied by $f_k(\conv(\cP_\la))$, see \cite{BR10b,BR10}.

The defect volume $\Vol_d(K\setminus\conv(\cP_\la))$, where $\Vol_d$ denotes the $d$-dimensional Lebesgue measure, has been investigated in parallel. Asymptotics for its limiting expectation, a central limit theorem, bounds for the variance and the calculation of the limiting variance in the case of a simple polytope have been derived in \cite{BB93}, \cite{BR10b}, \cite{BR10} and \cite{CY3} respectively. 

In this work, we concentrate on the case when $K$ is a $d$-dimensional simple polytope and we investigate the so-called convex hull peeling of a Poisson input $\cP_\la$ in $K$.

The procedure of \textit{convex hull peeling} consists in constructing a decreasing sequence $\conv_n(\cP_\la)$, $n\ge 1$, of convex hulls as follows: we initialize the process by taking $\conv_1(\cP_\la):=\conv(\cP_\la)$. We then remove the set of extreme points $\partial\conv_1(\cP_\la)\cap \cP_\la$ and define $\conv_2(\cP_\la):=\conv(\cP_\la\setminus \partial\conv_1(\cP_\la))$. We extend recursively the process so that for every $n\ge 1$,
\begin{equation}\label{eq:def couchen polytope}
\conv_n(\cP_\la)=\conv(\cP_\la\setminus (\cup_{i=1}^{n-1}\partial\conv_i(\cP_\la))).    
\end{equation}
We call $n$-th layer the boundary $\partial\conv_n(\cP_\la)$ and for any $x\in K$, we denote by $\ell_{\cP_\la}(x)$ the label of the layer which contains $x$ in the peeling of the point set $\cP_\la\cup\{x\}$,
see Figure \ref{fig:peeling_carre} for an example of convex hull peeling in the square.
\begin{figure}
    \centering
    \includegraphics[width=0.8\linewidth,trim={0cm 1.5cm 0cm 1.5cm}]{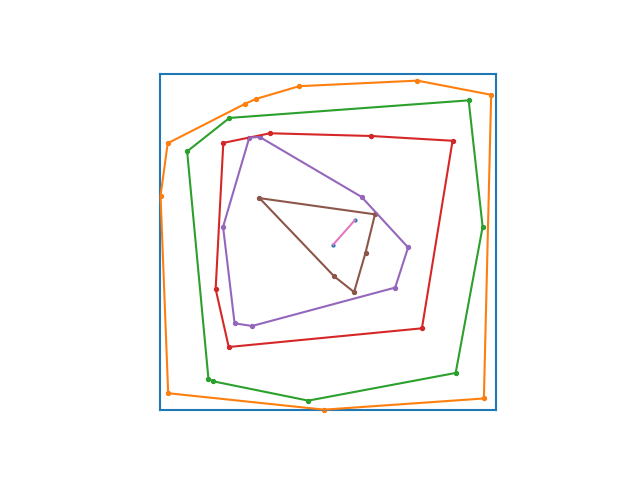}
    \caption{Convex hull peeling of a point set in a square, with a color for each layer. For example the third layer is in red.}
    \label{fig:peeling_carre}
\end{figure}

To the best of our knowledge, the literature on the theoretical study of random convex hull peeling remains sparse, as emphasized in \cite[Section 2.2.7]{R10}. Introduced by Barnett in \cite{B76} in 1976, the procedure of peeling finds its roots in spatial statistics, in particular it is used for defining a notion of depth inside multivariate data \cite{HA04, RS04}.  We draw attention to two particular references which investigate the asymptotics of the depth related to the peeling. In \cite{D04}, Dalal shows that the total number of layers, which differs from $\max_{x\in K}\ell_{\cP_\la}(x)$ by at most $1$, is upper and lower bounded  in expectation by $\la^{\frac{2}{d+1}}$ up to a multiplicative constant. Actually, his statement is given for a binomial input, i.e. with deterministic cardinality, but it can be Poissonized straightforwardly. In any case, the result is remarkable as it does not depend on the nature of the convex body $K$, which means that it includes both cases of a polytope and of a smooth convex body. In fact, Dalal does not even assume that the underlying mother body is convex.

More recently, in a breakthrough paper \cite{CS20}, Calder and Smart have strengthened Dalal's result by deriving a limit expectation and a law of large numbers for the total number of layers when the input is a Poisson point process in a convex body $K$, such that its intensity measure has a continuous and positive density $f$. Above all, they produce a functional version of the convergence by showing that the rescaled convex height function, equal to $\la^{-\frac{2}{d+1}}(\ell_{\cP_\la}(x)+1)$ according to their convention therein, converges uniformly and almost surely to a limit function which is proved to be the viscosity solution of an explicit non-linear partial differential equation, see \cite[Theorem 1.2]{CS20}. Naturally, their renormalization involves layers of the peeling which have a label proportional to $\la^{\frac{2}{d+1}}$. In particular, let us fix $n(\la,t)=\lfloor t\la^{\frac{2}{d+1}}\rfloor$ and denote by $N_{n(\la,t),0,\la}$ the number of Poisson points lying on the $n(\la,t)$-th layer of the peeling. Calder and Smart then conjecture \cite[display (1.18)]{CS20} with a short heuristic argument that when $\la\to\infty$, almost surely
\begin{equation}\label{eq:conjCS polytope}
\la^{-\frac{d-1}{d+1}}N_{n(\la,t),0,\la}\longrightarrow\frac1\alpha\int_{\{\alpha h=t\}}f^{\frac{d-1}{d+1}}    \kappa^{\frac1{d+1}}\mathrm{d}S    
\end{equation}
where $\alpha$ is a positive constant depending only on dimension $d$, $\alpha h$ is the limit function of $\la^{-\frac{2}{d+1}}\ell_{\cP_\la}(x)$, $\kappa$ is the Gaussian curvature of the level set $\{\alpha h=t\}$ and $\dd S$ is the Hausdorff measure of that level set.

Again, they believe that \eqref{eq:conjCS polytope} should not depend on the nature of $K$, i.e. whether $K$ is a polytope or is smooth. This would suggest in particular a polynomial growth for the number of vertices of the $n(\la, t)$-th layer, i.e. asymptotically equal to $\la^{\frac{d-1}{d+1}}$ up to a multiplicative constant. This is especially noticeable in view of \eqref{eq:esperance 1ere couche polytope} which claims that when $K$ is a polytope, the first layer for a uniform input contains in mean a logarithmic number of Poisson points. Incidentally, Dalal's result on the growth rate for the total number of layers confirms that the number $N_{n,0,\la}$ of Poisson points on the $n$-th layer could not be logarithmic for all $n$ as there are, in mean and up to a multiplicative constant, at most $\la^{\frac2{d+1}}$ layers and a total of $\la$ Poisson points in $K$.

Inspired by Calder and Smart's conjecture \eqref{eq:conjCS polytope} and by its patent discrepancy with the case of the first layer when $K$ is a polytope, we devote this paper to the study of the combinatorial functionals, i.e. the number of $k$-faces for $0\le k\le (d-1)$, as well as the defect volume of the consecutive layers of the convex hull peeling of a Poisson point process $\cP_\la$ in a simple polytope $K$. We focus on the first layers of the peeling, i.e. with a label independent of $\la$. The regime that we investigate is different from the regime which is covered by Calder and Smart's study and conjecture \eqref{eq:conjCS polytope} but in our opinion, this choice is meaningful for two reasons. First, when applying the peeling procedure to outlier detection in a point set, we expect the outliers to appear in the first layers rather than in the last ones. Secondly, we anticipate the first layers to be located near the boundary of $K$ so that the methods developed for deriving \eqref{eq:esperance 1ere couche polytope} and \eqref{eq:variance 1ere couche polytope} in the case of the first layer could extend to subsequent layers.

For a fixed $d$-dimensional simple polytope $K$ and a Poisson point process $\cP_\la$ in $K$ with intensity measure equal to $\la$ times the Lebesgue measure in $K$, we denote by  $N_{n,k,\la}$ the number of Poisson points on the $n$-th layer $\conv_n(\cP_\la)$, i.e. $N_{n,k,\la}=f_k(\conv_n(\cP_\la)).$ Our first main result below provides expectation and variance asymptotics for $N_{n,k,\la}$ when $\la\to\infty$, as well as a central limit theorem. For two real functions $f$ and $g$, the notation $f(\la)=O(g(\la))$ means that the ratio $f/g$ is bounded for $\la$ large enough.
\begin{thm}\label{thm:theoremeprincipalpolytope}
For any $n \geq 1$ and $k \in \lbrace 0, \ldots , d-1 \rbrace$ there exist constants $c_{n,k,d}, c'_{n,k,d} \in (0,\infty)$ which only depend on $n$, $k$ and $d$ such that
	$$\lim\limits_{\lambda \rightarrow \infty} \log^{-(d-1)}(\la){\mathbb{E}[N_{n,k,\lambda}]}
	= c_{n,k,d}f_0(K) \mbox{ and }
	\lim\limits_{\lambda \rightarrow \infty}  \log^{-(d-1)}(\la){\text{\normalfont Var}[N_{n,k,\lambda}]}
	= c'_{n,k,d}f_0(K).
	$$
    Furthermore, when $\la\to\infty$, we have
        \begin{equation*}
        \sup_t \Bigg|
        \P\left(\frac{N_{n,k,\la} - \E[N_{n,k,\la}]}{\sqrt{\Var[N_{n,k,\la}]}} \le t \right)
        - \P(\mathcal{N}(0,1) \le t )\Bigg| = O\left(\log^{-\frac{d-1}{2}} (\la) (\log \log (\la))^{9d^2 + 12(d-1)} \right)
    \end{equation*}
    where ${\mathcal N}(0,1)$ denotes a generic standard Gaussian random variable.
\end{thm}
Let us denote by $V_{n,\la}=\mbox{Vol}_d(K)-\mbox{Vol}_d(\mbox{conv}_n(\cP_\la))$ the defect volume of $\mbox{conv}_n(\cP_\la)$. Theorem \ref{thm:theoremeprincipalvolumepolytope} below extends the results from Theorem \ref{thm:theoremeprincipalpolytope} to the variable $V_{n,\la}$.
\begin{thm}\label{thm:theoremeprincipalvolumepolytope}
	For any $n \geq 1$, there exist $c_{V,n,d}, c'_{V,n,d} \in (0,\infty)$ such that
	$$\lim\limits_{\lambda \rightarrow\infty} \lambda \log^{-(d-1)} (\la)\mathbb{E}[V_{n,\lambda}]
	= c_{V,n,d}f_0(K) 
	\mbox{ and }
	\lim\limits_{\lambda \rightarrow\infty} \lambda^2\log^{-(d-1)} (\la)\text{\normalfont Var}[V_{n,\lambda}]
	= c'_{V,n,d}f_0(K).$$
Moreover, when $\la\to\infty$, we have
    \begin{equation*}
        \sup_t \Bigg|
        \P\left(\frac{V_{n,k,\la} - \E[V_{n,k,\la}]}{\sqrt{\Var[V_{n,k,\la}]}} \le t \right)
        - \P(\mathcal{N}(0,1) \le t )\Bigg| = O\left(\log^{-\frac{d-1}{2}} (\la) (\log \log (\la))^{9d^2 + 12(d-1)}\right).
    \end{equation*}
\end{thm}
The precise constants are given in Section \ref{sec:rescaling}, see Theorem \ref{thm:theoreme principal amelio}, as their expression requires introducing further notations.
Theorem \ref{thm:theoremeprincipalpolytope} (resp. Theorem \ref{thm:theoremeprincipalvolumepolytope}) shows that the growth rate for both the expectation and variance of the number of $k$-faces (resp. of the defect volume) of the $n$-th layer does not depend on $n$ up to a multiplicative constant, i.e. that for any $n\ge 1$, the $n$-th layer behaves like the first one. In view of \eqref{eq:conjCS polytope}, this suggests in particular that we should expect a phase transition when $n$ gets close to $\la^{\frac{2}{d+1}}.$ Only, the actual evolution of $N_{n,k,\la}$ (resp. $V_{n,k,\la}$) as a function of $n$ which would depend on $\la$ is unfortunately not covered by our methods. In any case, the transition from a logarithmic regime to a polynomial regime remains a pivotal question that would deserve further attention in the future.

Moreover, Theorems \ref{thm:theoremeprincipalpolytope} and \ref{thm:theoremeprincipalvolumepolytope} are intended as obvious companions to \cite[Theorems 1.1 \& 1.2]{CQ1}. Indeed, in \cite{CQ1}, we investigate the first layers of the convex hull peeling of a Poisson input in the $d$-dimensional unit ball and show that, in the same way, the expectation and variance of the number of $k$-faces (resp. defect volumes) of the $n$-th layer behave asymptotically like the first layer and have a growth rate proportional to $\la^{\frac{d-1}{d+1}}$ (resp. $\la^{-\frac2{d+1}}$ and $\la^{-\frac{d+3}{d+1}}$ in the case of the defect volume).

In fact, the connection with \cite{CQ1} is even more visible when we address the strategy of proof. Indeed, the main results in \cite{CQ1} rely on the use of a global scaling transformation in the unit ball which was introduced in the context of the first layer \cite{CSY}. This transformation sends the consecutive layers of the convex hull peeling to the layers of a different peeling with a so-called parabolic convexity. In particular, only the layers with label independent of $\la$ are visible in this new picture and this explains a posteriori the choice for this particular regime. In the case of Theorems \ref{thm:theoremeprincipalpolytope} and \ref{thm:theoremeprincipalvolumepolytope}, we take our inspiration from \cite{CY3} which contains in particular the proof of \eqref{eq:variance 1ere couche polytope} and we show that near each vertex of $K$, we can construct a scaling transformation which sends the consecutive layers of the initial convex hull peeling to layers of a new peeling. The role played by the paraboloids in \cite{CQ1} is now played by so-called cone-like grains. Again, the rescaling makes visible the first layers with label independent of $\la$ and only them.

The reason why the scaling transformation is needed is the following: in both \cite{CQ1} and the present paper, we can show properties of stabilization for the rescaled process. In other words, the status of one point, meaning the number of the layer containing that particular point, should only depend on the intersection of the Poisson point process with a vicinity of that point. This guarantees in turn mixing properties which imply, thanks to Mecke's formula for Poisson point processes, the calculation of the limit expectation and variance of the number of $k$-faces and of the defect volume in the neighborhood of a vertex of $K$.

Still, the analysis of the peeling in the neighborhood of each vertex of $K$ is not enough for deriving Theorems \ref{thm:theoremeprincipalpolytope} (resp. \ref{thm:theoremeprincipalvolumepolytope}), namely global asymptotics for the number of $k$-faces (resp. defect volume) of each layer. We need to show that on the one hand the contribution of the flat parts, i.e. the regions far from the vertices of $K$, is negligible and on the other hand that the contributions of all the neighborhoods of the vertices of $K$ can be added. In the case of the variance, this means that these contributions must decorrelate asymptotically. Both of these issues are treated with the same method as in \cite{CY3} which is in turn inspired by geometric techniques developed in \cite{BR10b}. One of the key ingredients in \cite{BR10b} is a so-called sandwiching result, which states that with high probability, the boundary of the convex hull of $\cP_\la$ lies between two so-called floating bodies related to $K$. 
The shape of such a sandwich is very thin along the facets of $K$ and it is wider in the neighborhood of the vertices. As such, it plays the role of a deterministic approximation for the location of the boundary of the random convex hull of $\cP_\la$. Subsequently, it is partitioned into smaller regions with the help of an economic cap covering and this paves the way for the construction of a dependency graph. This, in turn, induces both the decorrelation of the contributions of the neighborhoods of the vertices and the negligibility of the contribution of the flat parts. Additionally, this dependency graph allows us to derive a central limit theorem by classical methods. Our goal in this paper consists in adapting the technique to the first layers of the peeling, starting with the sandwiching property. 

In comparison with the previous works on the convex hull of a Poisson input in a polytope \cite{BR10b,CY3} and on the peeling in the unit ball \cite{CQ1}, the most challenging and novel parts in the proof of Theorems \ref{thm:theoremeprincipalpolytope}  and \ref{thm:theoremeprincipalvolumepolytope} are the following. 
\begin{itemize}
    \item The proof of the stabilization result after rescaling, see notably Proposition \ref{Stabilisation k faces}, differs significantly from its counterpart \cite[Proposition 3.6]{CQ1} in the case of the parabolic peeling for two reasons: first, the geometry of the cone-like grains which play the role held by the paraboloids in \cite{CQ1} prevents us from making explicit analytical calculations. This induces a specific study of the shape of a cone-like grain, see Lemma \ref{lem:incl cone}, and a careful use of a spherical cone included in a cone-like grain. Secondly, the scaling tranformation sends the initial Poisson point process $\cP_\la$ to a Poisson point process in the whole space $\R^{d-1}\times\R$ which has a limit density equal to $\sqrt{d}e^{dh}$. In comparison, the limit Poisson point process in \cite{CQ1} is homogeneous with intensity $1$ in the product space $\R^{d-1}\times \R_+$. This new exponential density induces several adaptations in the proofs of the main results in Section \ref{sec:stab polytopes}.
    \item The extension of the sandwiching result to the layers of the convex hull peeling requires new ideas: we study the peeling construction in each subset given by the economic cap covering of the sandwich, then use a general estimate for the probability of having points in convex position combined with the monotonicity of the layer numbers of the peeling.
    \item The proof in Section \ref{section:dec var} for the negligibility of the flat parts and decorrelation of the parts in the neighborhoods of the vertices of $K$  contains a specific twist with respect to its counterpart in \cite[Section  3]{CY3}. Indeed, the technique in \cite{CY3} relies on the knowledge of an actual upper bound for the variance of the number of $k$-faces of the first layer provided by \cite{BR10}. This falls down in the case of the subsequent layers of the convex hull peeling. Our proof then requires a careful study of each term appearing in our decomposition without any prior prediction on the growth rate of either the expectation or the variance.
\end{itemize}
The paper is structured as follows. In Section \ref{sec:floating bodies etc}, we introduce a few geometric objects,  including the floating bodies, and show the required sandwiching property for the first layers of the convex hull peeling. This represents our starting point for the proof of Theorems \ref{thm:theoremeprincipalpolytope} and \ref{thm:theoremeprincipalvolumepolytope}. Section \ref{sec:rescaling} is devoted to the introduction of a scaling transformation in the neighborhood of a vertex of $K$, which is the same as in \cite{CY3}, and to the description of its effect on the layers of the convex hull peeling. We show in Section \ref{sec:stab polytopes} the required  stabilization properties for the functionals of interest in the rescaled picture. This implies the asymptotics for the expectation and variance of the number of $k$-faces and defect volume in the neighborhood of a vertex of $K$ given in Proposition \ref{prop:asymptotics corner}. Section \ref{section:dec var} deals with the negligibility of the contribution of the remaining parts of the sandwich and the asymptotic decorrelation of the contributions of all the vertices of $K$. This relies upon the construction of a dependency graph, which is coincidentally the main ingredient to obtain our CLT. We then prove Proposition \ref{prop:asymptotics corner} and Theorems \ref{thm:theoremeprincipalpolytope} and \ref{thm:theoremeprincipalvolumepolytope} in Section \ref{sec:main results}. The final Section \ref{sec:concrempolytope} includes a list of selected comments and prospects suggested by our work.






Throughout the paper, unless noted otherwise, $c$, $c'$, $c_1$, $c_2\ldots$ denote positive constants which only depend on dimension $d$ and possibly $n$ and $k$ and whose value may change at each occurrence.
\section{Floating bodies, Macbeath regions and sandwiching}\label{sec:floating bodies etc}
Floating bodies and Macbeath regions are classical objects from convex geometry that have proved to play a significant role in the study of random polytopes \cite{BL88,BR10b}. This section aims at introducing them and using them to prove a probabilistic result of asymptotic nature, which goes as follows: with high probability, the first layers of the convex hull peeling are located  in a small vicinity of the boundary of $K$ which is precisely described as the region between two well-calibrated deterministic floating bodies. This extends a similar result due to B\'ar\'any and Reitzner in the context of the convex hull \cite[Section 5]{BR10b} and which is called by them {\it sandwiching property} in a welcome figurative way. 
That property is indeed crucial for the study of the second-order asymptotics of the convex hull inside a polytope \cite{BR10b,CY3}. Unsurprisingly, we show in the rest of the paper that this is also the case for the subsequent layers of the peeling.

\subsection{Floating bodies and Macbeath regions}
\label{sec:partie floating bodies polytope}
In this subsection, we recall the definitions of floating bodies and Macbeath regions and state the classical economic cap covering due to B\'ar\'any and Larman \cite{BL88} and B\'ar\'any \cite{B89}. 

Floating bodies are convex bodies included in $K$ which provide notably a deterministic approximation of the consecutive convex hulls of the random input in $K$. As in \cite{BR10b} and \cite{CY3}, we define $\textsl{v} : K \rightarrow \R$ as $$ \textsl{v}(z) := \min\{ \text{Vol}_d(K \cap H) : H \text{ is a half-space of } \R^d \text{ and } z \in H \},$$
where we recall that $\Vol_d$ is the $d$-dimensional Lebesgue measure.
For $t \in [0,\infty)$ we call floating body of $K$ with parameter $t$ the set $$K(\textsl{v}\geq t) := \{ z \in K : \textsl{v}(z) \geq t \}$$ and define similarly the sets $K(\textsl{v} \leq t)$ and $K(\textsl{v} = t)$. 

When $K$ is locally a cube around one of its vertices, it is possible to make explicit the equation of $K(\textsl{v}=t)$ and show that its shape is pseudo hyperbolic. We recall without proof the following result.
\begin{lem}[\cite{CY3}, Lemma 7.1]\label{lem:floating body eq}
    There exists $\Delta_d \in [1, \infty)$ depending only on $d$ such that when $K$ contains $[0, \Delta_d]$ and is contained in some multiple of that cube, then for any $t \in (0, \infty)$ :
    
    \begin{equation*}
        K(\textsl{v}=t)\cap [0, 1/2]^d = \lbrace (z_1, \ldots, z_d) \in [0, 1/2]^d : \prod_{i=1}^d z_i = \frac{d!}{d^d} t \rbrace
    \end{equation*}
\end{lem}

B\'ar\'any and Larman were the first to show a connection between floating bodies and random polytopes, namely that the mean defect volume of the random polytope generated as the convex hull of $n$ i.i.d. uniform points in any convex body is well approximated up to a multiplicative constant by the volume of the floating body with parameter $\frac1n$ \cite[Theorem 1]{BL88}. Such a property relies on the construction of a specific deterministic covering of the set $K(\textsl{v}\le t)$ which is called the \textit{economic cap covering} and which has been introduced in \cite{BL88} and \cite{B89}. Subsequently, the economic cap covering has been instrumental in B\'ar\'any and Reitzner's proof of the sandwiching property of the first layer of the peeling \cite{BR10b} and will remain so for proving the analogue for the next layers. Its statement, given in Theorem \ref{thm:economic cap} below, relies on the notion of  Macbeath regions defined as follows: the Macbeath region, or M-region for short, with center $z$ and factor $\la > 0$ is the set
\begin{equation}\label{def:macbeath}
M(z,\la) = M_K(z,\la) := z + \la[(K - z) \cap (z - K)].
\end{equation}

Let $s_0 := (2d)^{-2d}$. For any $s \in [0, s_0]$, we choose a maximal set of points $z_1(s), \ldots, z_{m(s)}(s)$ on $K(\textsl{v} = s)$ having pairwise disjoint M-regions $M(z_i, \frac{1}{2})$. We call such a system saturated. We assert that it exists for each $s$ although it is not necessarily unique.
For any $z$ in the interior of $K$, we denote by $C(z)$ the minimal cap of $z$, i.e. the intersection of $K$ with a half-space $H$ containing $z$ which has minimal volume equal to $\textsl{v}(z)$. The hyperplane bounding $H$ is at some distance $t>0$ from the support hyperplane in same direction $u\in\S^{d-1}$ and we denote by $C^{\gamma}(z)$ the intersection of $K$ with the translate of $H$ by $-(\gamma-1)tu$. We then write
$K'_i(s):= M(z_i, \frac{1}{2}) \cap C(z_i)$ and $K_i(s) := C^6(z_i)$ for $1\le i\le m(s)$ and state the economic cap covering result as follows.
\begin{thm}[\cite{BL88}, Theorem 6]\label{thm:economic cap}
    Assume that $\Vol_d(K)=1$. For all $s \in [0, s_0]$, we have :
    \begin{enumerate}
        \item $\cup_{i=1}^{m(s)} K'_i(s) \subset K(\textsl{v} \leq s) \subset \cup_{i=1}^{m(s)} K_i(s),$
        \item $s \leq \Vol_d(K_i(s)) \leq 6^d s$ for all $i = 1, \ldots, m(s)$,
        \item $(6d)^{-d}s \leq \Vol_d(K'_i(s)) \leq 2^{-d} s$ for all $i = 1, \ldots, m(s)$.
    \end{enumerate}
\end{thm}

\subsection{Sandwiching}
\label{sec:sandwiching}
In this section, we aim at proving Theorem \ref{thm:sandwiching} below, i.e. constructing two floating bodies with small parameters such that the first $n$ layers of the peeling are located between those floating bodies with high probability. This result extends B\'ar\'any and Reitzner's work \cite[Claims 5.1 and 5.2]{BR10b} on the convex hull, i.e. the first layer, and for that reason, we keep their original expression of \textit{sandwiching result}. 

Following their notation, we write 
\begin{equation}\label{eq:sTT*Sandwiching}
s:=\frac{1}{\la\log^{4d^2+d-1}(\la)},\quad T := \alpha \frac{\log \log (\la)}{\la}\mbox{ and } T^*=d 6^dT
\end{equation}
where $\alpha = 16\cdot2^{-d}\cdot(6d)^{2d}(4d^2 + d - 1)$. 
We make the observation here that the definitions of $s$, $T$ and $T^*$ are the same as in \cite[Section 5]{BR10b} but the constant $\alpha$ is larger. The reason why $\alpha$ has been recalibrated for our purpose is visible in the proof of Lemma \ref{lem:n couches macbeath}.

We introduce the sandwich set $$\cA(s,T^*,K) := K(\textsl{v} \geq s) \setminus K(\textsl{v} \geq T^*) $$
and the sandwiching event, for fixed $n\ge 1$, 
\begin{equation}\label{eq:def tilde A la}
    \tilde{A}_\la := \{\cup_{l=1}^n \partial \conv_l(\cP_\la) \subset  \cA(s,T^*,K)\}.
\end{equation}
Theorem \ref{thm:sandwiching} below shows that the event $\tilde{A}_\la$ above occurs with high probability.
\begin{thm}\label{thm:sandwiching}
Assume that $\Vol_d(K)=1$. There exists a positive constant $c>0$ such that for all $\la$ large enough,
        $$\P(\tilde{A}_\la) \ge 1 - c \log^{-4d^2}(\la).$$
\end{thm}
This result will prove to be of utmost importance further on. First, it is instrumental when studying the convex hull peeling near a vertex of $K$: indeed, in Section \ref{sec:rescaling}, the neighborhood of a vertex is transformed after rescaling into a product space in which the convex peeling becomes a peeling of another sort. Proposition \ref{lem:all cone-extreme} therein is proved with the help of Theorem \ref{thm:sandwiching} and shows that the faces of the first layers of the initial convex peeling are sent to the faces of the new peeling. 
Later on, Theorem \ref{thm:sandwiching} is again essential all the way through Section \ref{section:dec var} when proving Proposition \ref{prop:decompEVar} which says in particular that the contribution of the Poisson points far from a vertex of $K$ is negligible in the asymptotic estimate of both the expectation and the variance of the variables $N_{n,k,\la}$. Intuitively, this is due to the fact that the sandwich between the two floating bodies is asymptotically much thinner in the region far away from the vertices of $K$ and consequently contains less Poisson points in that particular region, see Figure \ref{fig:sandwiching}.
\begin{figure}
    \centering
    \includegraphics[width=0.7\linewidth, trim={0cm 1.5cm 0cm 1.5cm}]{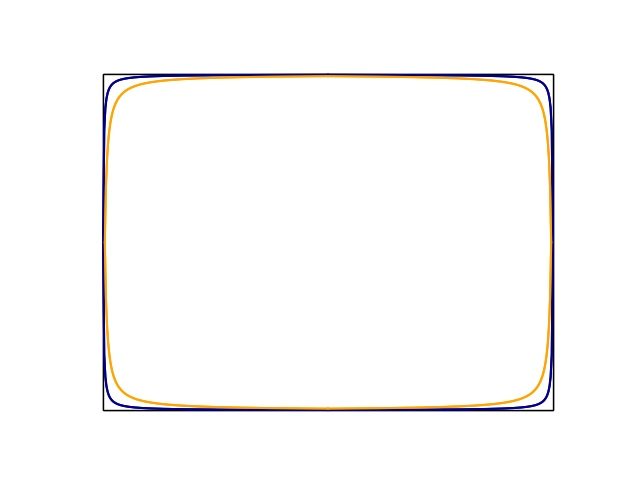}
    \caption{Two floating bodies in the square. Theorem \ref{thm:sandwiching} states that with high probability the first $n$-layers are located between two such well calibrated floating bodies.}
    \label{fig:sandwiching}
\end{figure}


The proof of Theorem \ref{thm:sandwiching} is postponed at the end of this section. It relies first on the use of Theorem \ref{thm:economic cap} which guarantees the existence of sets $K_i$ and $K_i'$, $1\le i\le m(T)$ for the coverage of $K(\textsl{v}\le T)$ and then on the essential fact that with high probability, the peeling of the Poisson points in each set $K_i'$ gives birth to at least $n$ layers, see Lemma \ref{lem:n couches macbeath} below.

A prerequisite to the proof of Lemma \ref{lem:n couches macbeath} is a general estimate on the probability denoted by $p(n,L)$ that $n$ i.i.d. points which are uniformly distributed in a fixed convex body $L$ in $\R^d$ are in convex position, i.e. belong to the boundary of their common convex hull. That particular question is a  classical extension to any finite point set and any dimension of the famous Sylvester's four-point problem which asks for the probability that $4$ \textit{random points} in the plane are the vertices of a convex quadrilateral. In dimension $2$, the precise asymptotic behavior of the probability $p(n,L)$ has been obtained by B\'ar\'any \cite{B99}, see also \cite{HCS08} for further discussion. In particular, he shows that for any planar convex body $L$, when $n\to+\infty$,
\begin{equation}\label{eq:dev Barany CP}
\log(p(n,L))=-2n\log n+n\log(e^2 A^3(L)/(4\Vol_2(L)))+o(n),
\end{equation}
where $A(L)$ means the maximal affine perimeter of any convex body included in $L$ and where the notation $f(n)=o(g(n))$ for two functions $f$ and $g$ means that the ratio $f/g$ goes to zero when $n\to\infty$.
The case of higher dimension is considered in \cite{B99} and for any $d\ge 3$, an expansion similar to \eqref{eq:dev Barany CP} is conjectured, see display (3.1) therein. To the best of our knowledge, the existence of such a result is still open. Fortunately, this is not required for our purpose. We state below a weaker estimate which has been obtained by B\'ar\'any in 2001.
\begin{lem}[\cite{B01}, Theorem]\label{lem:points convex position}
    There exist two positive constants $c_1<c_2$ and $n_0\ge 1$ such that for any convex body $L$ in $\R^d$ and $n \ge n_0$,
    \begin{equation}\label{eq:probaconvpos}
    c_1^n n^{-\frac{2}{d-1}n}\le p(n,L)
    \leq c_2^n n^{-\frac{2}{d-1}n}.
    \end{equation}
\end{lem}
As emphasized above, Lemma \ref{lem:points convex position} is the key ingredient for proving Lemma \ref{lem:n couches macbeath} below which investigates the peeling in each set $K_i'$ from the economic cap covering applied to the convex body $K$ and the parameter $T$. For any $1\le i\le m(T)$, we denote by $L_{i,\la}$ the number of layers in the peeling of  $\cP_\la \cap K'_i$, i.e. 
\begin{equation}\label{eq:defLila}
L_{i,\la} := \max_{x \in \cP_\la \cap K'_i} \ell_{\cP_\la\cap K_i'}(x).
\end{equation}
\begin{lem}\label{lem:n couches macbeath}
    Let $K_1',\ldots,K'_{m(T)}$ be the $m(T)$ sets provided by the economic cap covering applied to $K$ with the parameter $T$ given at \eqref{eq:sTT*Sandwiching}. For every $n\ge 1$, there exists $c>0$ such that
    $$\P(\exists i \in \{1, \ldots, m(T) \} : L_{i,\la} \leq n) \leq c \log^{-4d^2}(\la) .$$
\end{lem}
\begin{proof}
Let us introduce $m_{i,\la}:= \E[\text{card}(\cP_\la \cap K'_i)]$ for $1\le i\le m(T)$. Thanks to part 3 of Theorem \ref{thm:economic cap} and \eqref{eq:sTT*Sandwiching}, we obtain
\begin{equation}\label{eq:encadrementmila}
 (6d)^{-d}\alpha \log\log(\la) \le m_{i,\la}\le 2^{-d}\alpha \log\log( \la).
 \end{equation}
Consequently, we get for every $1\le i\le m(T)$
 \begin{align}\label{eq:somme Lj}
     \P(L_{i,\la} \leq n) \le & T_1+T_2
 \end{align}
 where
 \begin{equation*}\label{eq:defTermT1polyt}
T_1:=\P\left(|\text{card}(\cP_\la \cap K'_i)- m_{i,\la}| \geq \frac12(6d)^{-d}\alpha\log\log( \la)\right)     
 \end{equation*}
 and
 \begin{equation*}\label{eq:defTermT2polyt}
T_2:= \P\left(L_{i,\la} \leq n, c_1\log\log(\la)\le \text{card}(\cP_\la \cap K'_i) \le c_2\log\log(\la)\right), \end{equation*}
 with $c_1=\frac12(6d)^{-d}\alpha$ and $c_2=\left(\frac12(6d)^{-d}+2^{-d}\right)\alpha.$
 
Using the fact that $\text{card}(\cP_\la \cap K'_i)$ is Poisson distributed with mean $m_{i,\la}$, we observe that the term $T_1$ can be estimated through a classical concentration inequality for the Poisson distribution, see e.g. \cite[Section 2.2]{BLM13}. Indeed, using that any Poisson variable $\mbox{Pois}(\mu)$ with mean $\mu$ satisfies for every $x>0$
$$\P(|\mbox{Pois}(\mu)-\mu|\ge x)\le 2\exp\left(-\frac{x^2}{2(\mu+x)}\right),$$
combining it with \eqref{eq:encadrementmila} and recalling the value of $\alpha$,
we obtain that 
 \begin{equation}
 \label{eq:estimation1 Lj}
T_1\le  2\log^{-( 4d^2 + d - 1)}(\la).
 \end{equation}
We now estimate $T_2$. By the law of total probability, 
 \begin{equation}\label{eq:bayespourT2}
 T_2\leq \sum_{j= \lfloor c_1 \log\log( \la)\rfloor}^{\lceil c_2 \log\log(\la)\rceil}
 \P\big(L_{i,\la} \leq n \,\mid\, \text{card}(\cP_\la \cap K'_i) = j \big).
 \end{equation}
 Given that $\text{card}(\cP_\la \cap K'_i) = j$, $\cP_\la \cap K'_i$ consists of  $j$ i.i.d. variables which are uniformly distributed in $K'_i$.
Moreover, whenever $L_{i,\la} \leq n$, at least one layer with label smaller than $n$ must contain at least $j/n$ points. In particular at least $j/n$ points in $\cP_\la \cap K'_i$ need to be in convex position. This implies that 
\begin{align*} \P\big(L_{i,\la} \leq n \,\mid\, \text{card}(\cP_\la \cap K'_i) =j \big)
 \leq {{j}\choose{\lceil j/n \rceil}} p(\lceil j/n\rceil,K_i')
 \end{align*}
where we recall the notation $p(n,L)$ introduced before \eqref{eq:dev Barany CP}. Using Lemma \ref{lem:points convex position}, the fact that $j$ is proportional to $\log\log (\la)$ and the inequality ${{j}\choose{\lceil j/n \rceil}} \leq 2^j$, we deduce from the previous inequality that for $\la$ large enough, 
\begin{equation}\label{eq:majunifgraceaBarany}
 \P\big(L_{i,\la} \leq n \,\mid\, \text{card}(\cP_\la \cap K'_i) =j \big)\le c^{\log\log(\la)}(\log\log(\la))^{-c'\log\log(\la)}
\end{equation}
 where $c$ and $c'$ are two positive constants which depend only on dimension $d$. Summing \eqref{eq:majunifgraceaBarany} over $j$ between $\lfloor c_1\log\log(\la)\rfloor$ and $\lceil c_2\log\log(\la)\rceil$ and using \eqref{eq:bayespourT2}, we obtain the existence of $c>0$ such that  for $\la$ large enough
\begin{equation}\label{eq:estimT2bonpourdem}
T_2\le (\log\log(\la))^{-c\log\log(\la)}.
\end{equation}
Combining \eqref{eq:somme Lj}, \eqref{eq:estimation1 Lj} and \eqref{eq:estimT2bonpourdem}, we get for every $1\le i\le m(T)$
$$\P(L_{i,\la}\le n)\le c\log^{-(4d^2+d-1)}(\la).$$
It remains to use the estimate $m(T)\le c\log^{d-1}(\la)$ from \cite[Theorem 2.7]{BR10b} to complete the proof of Lemma \ref{lem:n couches macbeath}. 
    
\end{proof}
We are finally ready to prove Theorem \ref{thm:sandwiching}.
\begin{proof}[Proof of Theorem \ref{thm:sandwiching}]
 We start by recalling the sandwiching result for the convex hull of $\cP_\la$ \cite[Claims 5.1 and 5.2]{BR10b}, which implies in particular that
$$\P(\conv_1(\cP_\la)\subset K(\textsl{v}\ge s))\ge 1 - c \log^{-4d^2}(\la).$$ 
Since the consecutive convex hulls $\conv_n(\cP_\la)$, $n\ge 1$ produced by the peeling are decreasing, this also implies that for every $n\ge 1$,
$$\P(\conv_n(\cP_\la)\subset K(\textsl{v}\ge s))\ge 1 - c \log^{-4d^2}(\la).$$
 Consequently, only the bound 
 \begin{equation}\label{eq:sandwich tranche de pain du bas}
 \P(K(\textsl{v} \geq T^*) \subset  \conv_n(\cP_\la)) \ge 1 - c \log^{-4d^2}(\la)
 \end{equation}
requires an explanation.
 We consider again the sets $K_i=K_i(T)$ and $K_i'=K_i'(T)$, $1\le i\le m(T)$, provided by Theorem \ref{thm:economic cap} applied to the convex body $K$ and the parameter $T$. Thanks to Lemma \ref{lem:n couches macbeath}, each $K'_i$ contains at least $n$ layers for the peeling of $\cP_\la \cap K'_i$ with probability greater than $1 - c \log^{-4d^2}(\la)$. On this event, we pick a point $x_i \in \cP_\la \cap K'_i$ for each $1\le i\le m(T)$ such that $\ell_{\cP_\la \cap K'_i}(x_i) = n$. In particular, the monotonicity of the function $\ell$ with respect to the input set \cite[Lemma 3.1]{D04}
 implies that $\ell_{\cP_\la}(x_i) \geq  n$ for every $i$ so $\conv(x_1, \ldots, x_{m(T)}) \subset \conv_n(\cP_\la)$. Moreover, \cite[Claim 4.5]{BR10b} provides the inclusion $K(\textsl{v} \geq T^*) \subset \conv(x_1, \ldots, x_{m(T)})$. Consequently, we have proved that
 \begin{equation}\label{eq:inclusion A'la}
 A'_\la:= \{\forall\,i \in \{1, \ldots, m(T) \}, L_{i,\la} \geq n\}\subset \{K(\textsl{v} \geq T^*) \subset \conv_n(\cP_\la)\}
 \end{equation}
 where we recall the definition of $L_{i,\la}$ at \eqref{eq:defLila}.
Lemma \ref{lem:n couches macbeath} now justifies \eqref{eq:sandwich tranche de pain du bas} and completes the proof of Theorem \ref{thm:sandwiching}.
\end{proof}

\section{Scaling transform and scores}
\label{sec:rescaling}
In this section, we introduce a scaling transform that lets us study the layers of the peeling in a more efficient way. This transform was already used in \cite{CY3} in the case of the first layer and we expect it to bring the same benefits in our setting.
As in \cite{CQ1}, the scaling transform leads to a new type of peeling where the hyperplanes are replaced by a different geometric shape. In \cite{CQ1}, that shape was a paraboloid while it is here a shape close to a cone. An important difference with the convex hull peeling in the ball \cite{CQ1} is the following: in the case of the simple polytope, the rescaling procedure is tailored for the study of the layers in the neighborhood of a vertex of $K$ but we are unable to build a suitable global scaling transform, see Remark ii) after Theorem 4 in \cite{CY3} for an explanation of this fact. 
A large part of this paper, see Section \ref{section:dec var}, is then dedicated to proving that the contribution of the Poisson points far from the vertices is negligible and that the contributions of the Poisson points near different vertices of $K$ are independent. With that in mind, we devote Sections \ref{sec:rescaling} and \ref{sec:stab polytopes} to the estimation through the scaling transform of the \textit{local} numbers of $k$-faces of the layers in the vicinity of a fixed vertex of $K$.

The outline of this section is the following. We first define the scaling transform and describe its effect on the Poisson point process, see Lemma \ref{lem:effet rescaling PPP polytope}. We then study its effect on the layers. In order to do this, we observe that only the so-called cone-extreme faces find their counterparts in the rescaled picture but fortunately, on the sandwiching event defined in Section \ref{sec:sandwiching}, all regular faces are proved to be cone-extreme faces, see Lemma \ref{lem:all cone-extreme}. We are then led to study a new type of rescaled peeling procedure where so-called cone-like grains defined with a specific function $G$, see \eqref{eq:defFonctionGpolytope}, play the role of the half-spaces in the classical convex hull peeling and where the same basic properties as for the regular peeling are proved to occur, see Lemmas \ref{lem:be on layer n} and \ref{lem:dalalreecrit2}.  
Next, we introduce a sequence of random variables called scores and its equivalent in the rescaled picture. Incidentally, we give a more precise version of Theorem \ref{thm:theoremeprincipalpolytope} where the constants are rewritten explicitly through integral formulas involving the scores. Finally, we state a few analytical properties for the function $G$ that are required in Section \ref{sec:stab polytopes}, see notably Lemma \ref{lem:incl cone}.

\subsection{Rescaling}\label{sec:rescaling2}
First of all, we introduce some useful notation related to the simple polytope $K$.
We denote by $\mathcal{V}_K := \{\mathscr{V}_i\}_{i = 1, \ldots, f_0(K)}$ the set of vertices of $K$. 
Rescaling $K$ if necessary, for each vertex $\mathscr{V}_i \in \mathcal{V}_K
$, we introduce
an associated volume preserving affine transformation $a_i : \R^d \rightarrow \R^d$ , with $a_i (\mathscr{V}_i) = 0$,
and such that the facets of $a_i(K)$ containing $0$ also contain the facets of $K' := \left[0, \Delta_d \right]^d$ 
belonging to the coordinate hyperplanes, with $\Delta_d$ fixed by Lemma \ref{lem:floating body eq}.
This is possible because $K$ is a simple polytope.
For any $\delta \in (0, \Delta_d )$, we define the parallelepiped 
\begin{equation}\label{eq:parallelepiped} p_d(\mathscr{V}_i, \delta) := a_i^{-1}([0, \delta]^d).
\end{equation}

In practice, it means that the behavior of the convex hull peeling in any of the sets $p_d(\mathscr{V}_i, \delta)$ can be deduced from the behavior of the convex hull peeling in $[0, \delta]^d$ through the use of $a_i$. For this reason, we focus in the sequel on the convex hull peeling in a neighborhood of the vertex $0$. We begin with a definition of a  rescaling in the corner $(0, \infty)^d$ that is suitable for our purpose.

We extend the definition of the logarithm and exponential functions to $d$-dimensional vectors by writing for any $z=(z_1,\ldots,z_d)\in (0,\infty)^d$ and $v=(v_1,\ldots,v_d)\in\R^d$, $$\log(z) = (\log(z_1), \ldots, \log(z_d))\mbox{ and } e^v=(e^{v_1},\ldots,e^{v_d}).$$ We write $p_V : \R^d \rightarrow V$ for the orthogonal projection onto $V$ where $$V = \{ z \in \R^d : \sum_{i=1}^d z_i = 0 \}.$$
The $(d-1)$ dimensional vector space $V$ is frequently identified with $\R^{d-1}$ in the sequel.
For any $\la \in [1, \infty)$ we define the scaling transform
\begin{equation}T^{(\la)} : \left\{
	\begin{array}{ll}
		(0,\infty)^d  & \rightarrow V\times \R\\
		(z_1, \ldots, z_d) & \mapsto (p_V(\log(z)), \frac{1}{d}(\log(\la) + \sum_{i=1}^d \log(z_i)))
	\end{array}
\right. .
\end{equation}
The transform $T^{(\la)}$ comes from \cite[Equation (1.5)]{CY3} and \cite{B00}. It is a one-to-one map between $(0,\infty)^d$ and $V \times \R$ of inverse
\begin{equation*}[T^{(\la)}]^{-1} : \left\{
	\begin{array}{ll}
		  V\times \R & \rightarrow (0,\infty)^d \\
		(v,h) & \mapsto \la^{-1/d} e^h e^{l(v)}
	\end{array}
\right.
\end{equation*}
where for any $v\in V$, $l(v)=(l_1(v),\ldots,l_d(v))$ and $l_i(v)$ is the $i$-th coordinate of $v$ in the standard basis of $\R^d$.


Let us consider the cube $Q_0 := [0,\delta_0]^d$ where 
\begin{equation}\label{eq:def delta_0}
\delta_0 := \exp(-\log^{1/d}(\la)).
\end{equation}
Its image by $T^{(\la)}$ is
   \begin{equation}\label{def:Wla}
   W_\la := T^{(\la)}(Q_0) = \{ (v,h) \in V\times \R : h \leq -l_i(v) + \log(\la^{1/d}\delta_0) \}.
   \end{equation}
We also define the image in $W_\la$ of the initial Poisson point process, i.e.
    $$ \cP^{(\la)} := T^{(\la)}(\cP_\la \cap Q_0).  $$
We recall \cite[Lemma 4.2]{CY3} that gives the distribution and the limit in distribution of $\cP^{(\la)}$.
\begin{lem}[\cite{CY3}, Lemma 4.2]\label{lem:effet rescaling PPP polytope}
Let $\cP$ be the Poisson point process in $\R^{d-1} \times \R$ with intensity measure $\sqrt{d} e^{dh} \dd v \dd h$. 
    The Poisson point process $\cP^{(\la)}$ satisfies 
 $$\cP^{(\la)}\overset{(d)}{=}\cP \cap W_\la\;\;\mbox{ and }\;\;\cP^{(\la)} \xrightarrow[]{\cL} \cP \mbox{ when $\la \rightarrow \infty$}$$
 where $\overset{(d)}{=}$ means the equality in distribution and $\xrightarrow[]{\cL}$ means the convergence in distribution.
\end{lem}
For sake of simplicity, when $\la = \infty$ we identify $W_\la$ with  $V \times \R$ and $\cP^{(\la)}$ with $\cP$ .
\subsection{Effect on the layers}
We now focus on the effect of the scaling transform $T^{(\la)}$ on the layers of the peeling procedure. In agreement with \cite[Section 3]{CY3}, we anticipate that $T^{(\la)}$ makes visible only a fraction of the faces of a layer, namely the so-called \textit{cone-extreme} faces. 

Adapting \cite[Definition 3.1]{CY3} to the $n$-th layer, we call a face $F$ of $\conv_n(\cP_\la)$ a `cone-extreme' face if 
    the collection $C_F(\conv_n(\cP_\la) \cap Q_0)$ of outward normals to $F$ 
     belongs to the normal cone $C_0(K) := (-\infty, 0)^d$.

Proposition \ref{lem:all cone-extreme} below states that on the event $\tilde{A}_\la$ defined at \eqref{eq:def tilde A la}, every face of $\conv_n(\cP_\la)$ is indeed cone-extreme. It extends to the subsequent layers the analogous result \cite[Proposition 3.1]{CY3} proved in the case of the first layer.
\begin{prop}\label{lem:all cone-extreme}
    On the event $\tilde{A}_\la$ we have for $\la$ large enough $C_F(\conv_n(\cP_\la \cap Q_0)) \subset C_0(K)$ for any face $F$ of $\conv_n(\cP_\la \cap Q_0).$
\end{prop}
The proof of \cite[Proposition 3.1]{CY3} relies on two ingredients: the sandwiching estimate \cite[Claims 5.1 and 5.2]{BR10b} and a construction of explicit dyadic Macbeath regions in a region containing $Q_0$. Unsurprisingly, the proof of Proposition \ref{lem:all cone-extreme} goes along very similar lines, i.e. it builds upon Theorem \ref{thm:sandwiching} and the explicit Macbeath regions given in Section 
\ref{section:dec var}. Consequently, for sake of brevity, we omit its proof.
We say that a hyperplane (resp. a half-space) is cone-extreme, if it has a unit normal vector (resp. an outward unit normal vector) in $C_0(K)$. We describe in the next lines how to encode in a suitable way such a hyperplane.

For any $t>0$, we call pseudo-hyperboloid the surface
$$\cH_t := \{(z_1, \ldots, z_d) \in (0, \infty)^d : \prod_{i=1}^d z_i = t \}.$$
For every $z^{(0)}\in(0,\infty)^d$, we denote by $H(z^{(0)})$ the hyperplane tangent to the unique pseudo-hyperboloid $\cH_t$ containing $z^{(0)}.$ The reasoning before \cite[(4.2)]{CY3} shows that the equation of this hyperplane is given by
$$H(z^{(0)}) := \{ (z_1, \ldots, z_d) \in \R^d : \sum_{i=1}^d \frac{z_i}{z^{(0)}} = d \} .$$ In particular, any cone-extreme hyperplane can be written as $H(z^{(0)})$ for some $z^{(0)}\in (0,\infty)^d$. 

Let us define for any $v\in V$ the function
\begin{equation}\label{eq:defFonctionGpolytope}
G(v) := \log(\frac{1}{d} \sum_{i=1}^d e^{l_i(v)}) .    
\end{equation}
Next, we define the downward cone-like grain as
$$\Pi^\downarrow := \{(v,h) \in \R^{d-1} \times \R : h < -G(v) \} .$$
For any $w\in \R^d$, we then denote by $\Pi^\downarrow(w) := w + \Pi^\downarrow$ the translate of $\Pi^\downarrow$ by $w$, often called cone-like grain with apex at $w$. In a similar way, we introduce the upward cone-like grain
$$\Pi^\uparrow := \{(v,h) \in \R^{d-1} \times \R : h > -G(v) \} $$
and its translate $\Pi^\uparrow(w) := w + \Pi^\uparrow$. The duality between the upward and downward cone-like grains is given by the equivalence
$$ w \in \partial{\Pi^\uparrow(w')} \iff w' \in \partial{\Pi^\downarrow(w)} .$$
Lemma \ref{lem:image hyperplane} below is taken from \cite[Lemma 4.3]{CY3}. It shows that any cone-extreme half-space is mapped to a downward cone-like grain. In other words, the cone-like grains play the role in the rescaled model of the cone-extreme half-spaces in the original model. 
\begin{lem}[\cite{CY3}, Lemma 4.3]\label{lem:image hyperplane}
    \noindent (i) For every $c \in (0, \infty)$, we have
    $$T^{(\la)}(\cH_{c/\la}) = V \times \lbrace \frac{1}{d} \log(c) \rbrace .$$
    
    \noindent (ii) For every cone-extreme half-space $H^+(z^{(0)})$, $z^{(0)} \in (0, \infty)^d$ , we have
    $$T^{(\la)}(H^+(z^{(0)})) = \Pi^\downarrow(T^{(\la)}(z^{(0)})) .$$
\end{lem}
A consequence of Lemmas \ref{lem:all cone-extreme} and \ref{lem:image hyperplane} is that on the event $\tilde{A}_\la$, the convex hull peeling of $\cP_\la$ is mapped to a new peeling procedure with cone-like grains instead of half-spaces. We now describe properly that peeling.

For any $\la \in [1,\infty]$ and any locally finite point set $Y$ in $W_\la$, we introduce $[\Pi^\downarrow ]^{(\la)} = \Pi^\downarrow \cap W_\la$ and  the cone-like hull of $Y$ as
 \begin{equation}\label{eq:def cone-like hull}
 \Phi^{(\la)}(Y) := \bigcap_{\substack{w \in W_\la \\ Y \cap [\Pi^\downarrow ]^{(\la)}(w) = \varnothing}} [\Pi^\downarrow ]^{(\la)}(w)^c.    
 \end{equation}
We then define recursively the consecutive hulls $\Phi_n^{(\la)}(Y)$, $n\ge 2$, of the cone-like peeling with the formula
\begin{equation}\label{eq:def layer cone-like peeling}
\Phi^{(\la)}_n(Y) = \Phi^{(\la)}(Y \cap \text{int}(\Phi^{(\la)}_{n-1}(Y))).    
\end{equation}
When $\la = \infty$, we generally write $\Phi_n$ instead of $\Phi_n^{(\infty)}$.
As a result of Lemma \ref{lem:image hyperplane}, we obtain that on the event $\tilde{A}_\la$, the transformation $T^{(\la)}$ maps the layers of the convex hull peeling of $\cP_\la \cap Q_0$ to the layers of the cone-like hull peeling of its image, i.e. for any $n\ge 1$
 $$T^{(\la)}(\conv_n(\cP_\la \cap Q_0)) = \Phi_n^{(\la)}(\cP^{(\la)}).$$
See Figure \ref{fig:cone_like_peeling} for an example of cone-like peeling.
\begin{figure}
    \centering
    \includegraphics[width=\textwidth, trim = {3cm, 1cm, 2cm, 0}]{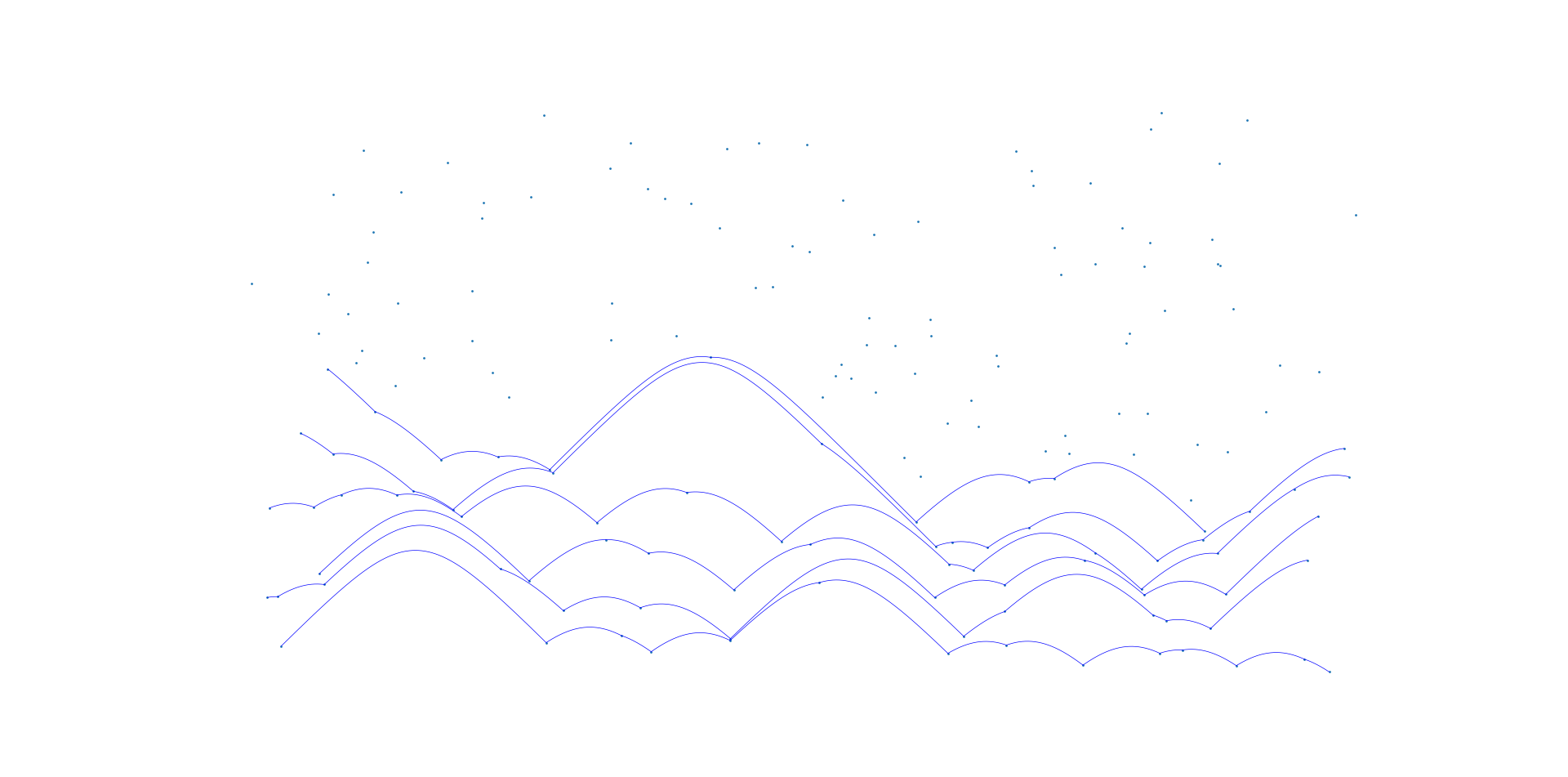}
    \caption{The first 6 layers of the cone-like peeling of a point set.}
    \label{fig:cone_like_peeling}
\end{figure}
\subsection{Properties of the rescaled layers}
    This section gathers two relevant results on the rescaled layers that will be used frequently in Section \ref{sec:stab polytopes}.
    We omit the proofs as they are direct transpositions to the context of the cone-like peeling of \cite[Lemma 2.5 and 2.6.]{CQ1} written for a parabolic peeling.
    
    For any $w\in \R^{d-1}\times \R$, we introduce its layer number as
    \begin{equation}\label{eq:defnumerocouche}
    \ell^{(\la)}(w,Y)=n \mbox{ such that $w\in \partial \Phi^{(\la)}_n(Y\cup\{w\})$}.
    \end{equation}
    The number $\ell^{(\la)}(w,Y)$ is the counterpart for the cone-like peeling of $Y$ of the layer number $\ell_X(x)$ of a point $x$ in the classical convex hull peeling of a point set $X$.
      According to the construction of the cone-like hull given at \eqref{eq:def cone-like hull}, we recall that a point $w$ in $\R^{d-1}\times\R$ is extreme if and only if there exists $(v_1, h_1) \in \partial [\Pi^\uparrow]^{(\la)}(w)$ such that
    $[\Pi^\downarrow]^{(\la)}(v_1,h_1) \cap Y = \varnothing$. The following result is a generalization of this criterion to the subsequent layers. 

    \begin{lem}\label{lem:be on layer n}
        Let $Y$ be a locally finite subset of $\R^{d-1}\times\R$, $w \in Y$, $\lambda\in [1,\infty]$ and $n\ge 1$. Then we have the two following equivalences.
        
        \noindent (i) 
        \emph{(}$\ell^{(\la)}(w,Y)\ge n$\emph{)} $\Longleftrightarrow$ \emph{(}$\forall\,(v_1, h_1) \in \partial[\Pi^\uparrow]^{(\la)}(w):
    Y\cap [\Pi^\downarrow]^{(\la)}(v_1,h_1) \not\subset  \cup_{i=1}^{n-2} \partial [\Phi_i]^{(\la)}(Y)$\emph{)}.\\~\\
        (ii) \emph{(}$\ell^{(\la)}(w,Y)\le n$\emph{)} $\Longleftrightarrow$ \emph{(}$\exists\;(v_1, h_1) \in \partial[\Pi^\uparrow]^{(\la)}(w): 
        Y\cap [\Pi^\downarrow]^{(\la)}(v_1,h_1) \subset \cup_{i=1}^{n-1} \partial
        [\Phi_i]^{(\la)}(Y)$\emph{)}.
    \end{lem}

Another important property of the peeling is that the layer number of any point increases with the point set. This is stated in Lemma \ref{lem:dalalreecrit2} below.
\begin{lem}\label{lem:dalalreecrit2}
For $\la \in [1,\infty]$, if $X\subset Y\subset W_\la$, we have for every $w\in W_\la$, $\ell^{(\la)}(w,X)\le \ell^{(\la)}(w,Y).$    
\end{lem}

\subsection{Properties of the function $G$}
We state here a few properties satisfied by the function $G$, introduced at \eqref{eq:defFonctionGpolytope}, which are needed in Section \ref{sec:stab polytopes}.

Lemma \ref{lem:sandwich cone} below comes from \cite[Lemma 4.4 and 4.5]{CY3}. It shows that the graph of the function $G$ is convex and sandwiched between two circular cones.
\begin{lem}[\cite{CY3}, Lemmas 4.4 and 4.5]\label{lem:sandwich cone}
   The function $G$ is convex and positive. Moreover, there exist two constants $\underline{c}$ and $\overline{c} \in (0, \infty)$ such that for every $v \in V$,
    $$\underline{c}\Vert v \Vert - \log(d) \leq G(v) \leq \overline{c} \Vert v \Vert  .$$
\end{lem}
Lemma \ref{lem:incl cone} provides a new property of the function $G$ that is heavily required to prove the stabilization of the scores in Section \ref{sec:stab polytopes}. In particular, it can be geometrically reinterpreted by stating that there exists a downward circular cone $\mathcal{C}$ such that for any point $w$ on the boundary of a cone-like grain $\Pi^{\downarrow}(v_1, h_1)$, we have $\mathcal{C} + w \subset \Pi^{\downarrow}(v_1, h_1)$, see Figure \ref{fig:inclusion cone}. 
\begin{figure}
    \centering
    \includegraphics[width = 0.7\textwidth]{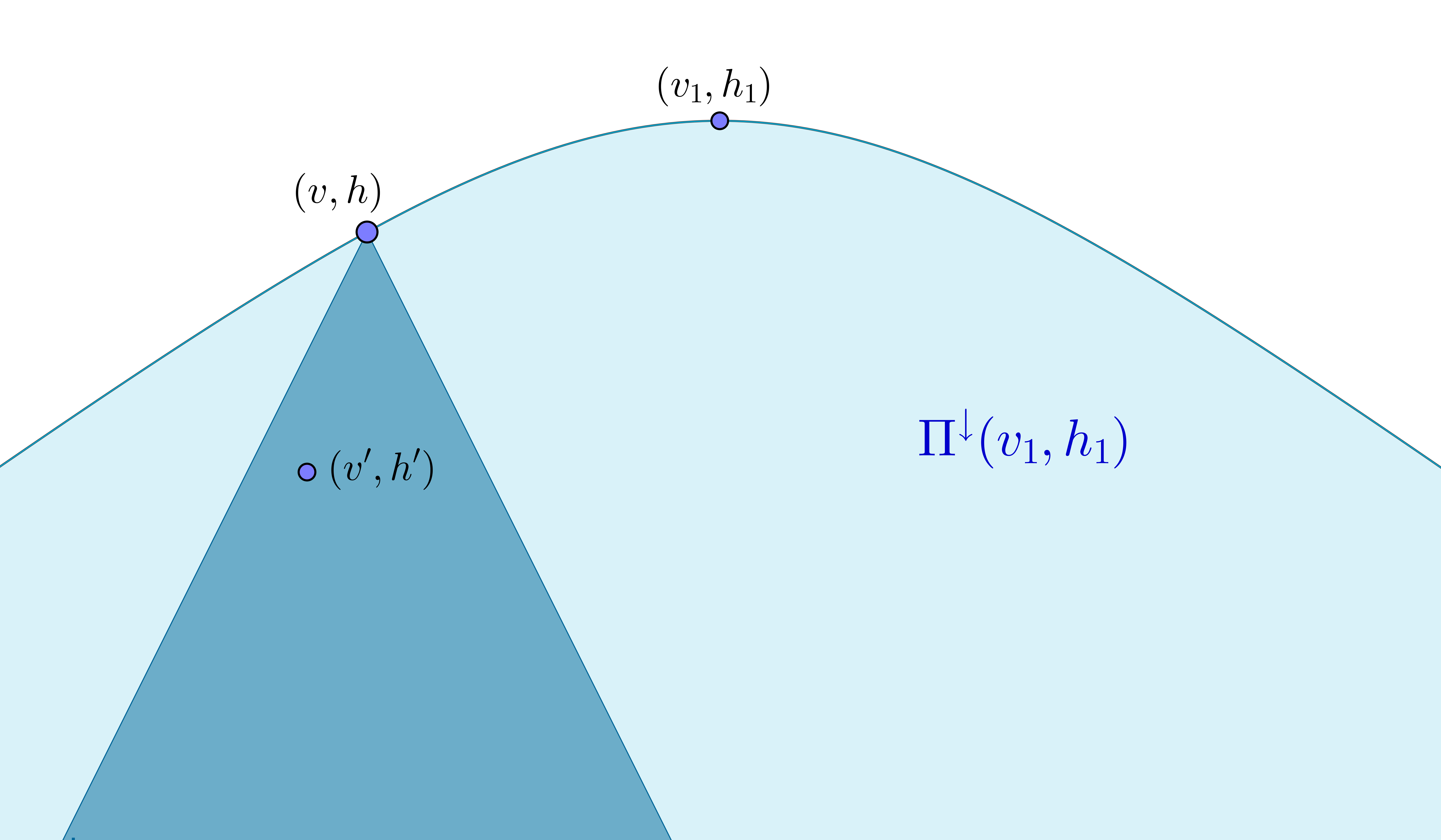}
    \caption{The cone inclusion of Lemma \ref{lem:sandwich cone}, part 2.}
    \label{fig:inclusion cone}
\end{figure}

\begin{lem}\label{lem:incl cone}
~\\
    \begin{enumerate}
        \item There exists a constant $c >0$ such that for any $v \in V$, $$\Vert\nabla G(v)\Vert \leq c \Vert v \Vert.$$
        \item For any $(v_1, h_1)$, $(v,h)$ and $(v',h')\in V \times \R $ such that $h = h_1 - G(v - v_1)$ and $h' \leq h - c \Vert v' - v \Vert$, we have 
        $$h' \leq h_1 - G(v' - v_1),$$
        with $c > 0$ the constant of part 1.
    \end{enumerate}
\end{lem}
\begin{proof}
    We fix a basis $\mathcal{B} = (\varepsilon_1, \ldots, \varepsilon_{d-1})$ of $V$ and a vector $v \in V$ which will be identified with its coordinates $(v_1, \ldots, v_{d-1})$ in $\mathcal{B}$.
    \\
    1. 
    For any $i = 1, \ldots, (d-1)$, we introduce $\alpha_{i,1}, \ldots, \alpha_{i,d-1} $ such that for any $v \in V$ we can write
    $l_i(v) = \alpha_{i,1} v_1 + \ldots + \alpha_{i,d-1} v_{d-1} $. It is enough to prove that for any $1\le j\le (d-1)$, there exists $c_j$ such that for any $v\in V$,  $\lvert\frac{\partial G}{\partial v_j}(v)\rvert \leq c_j$. Since for any $i,j$, we have $\frac{\partial l_i}{\partial v_j}(v) = \alpha_{i,j}$, we deduce from \eqref{eq:defFonctionGpolytope} that
    \begin{equation*} \frac{\partial G}{\partial v_j}(v) = \frac{\sum_{i=1}^{d} \frac{\partial l_i}{\partial v_j}(v) e^{l_i(v)}}{\sum_{i=1}^{d}e^{l_i(v)}} = \frac{\sum_{i=1}^{d} \alpha_{i,j} e^{l_i(v)}}{\sum_{i=1}^{d}e^{l_i(v)}}.
    \end{equation*}
    Thus we get the bound 
    $$\bigg\lvert \frac{\partial G}{\partial v_j}(v) \bigg\rvert \leq \max_{k,l} \lvert \alpha_{k,l}\rvert $$
    and taking $c_j = \max_{k,l} \lvert \alpha_{k,l} \rvert $ yields the result.
    \\~\\
    2. Using part 1 and the mean value theorem applied to the function $G$, we obtain that
    \begin{equation}\label{eq:meanvaluetheoremG}
    G(v'-v_1)\le G(v-v_1)+c\|v'-v\|.    
    \end{equation}
    Using consecutively the two assumptions on $(v_1,h_1)$, $(v,h)$ and $(v',h')$ then \eqref{eq:meanvaluetheoremG}, we then get
    \begin{align*}
    h'&\le h-c\|v'-v\|\\
    &\le h_1-G(v-v_1)-c\|v'-v\|\\
    &\le h_1-G(v'-v_1),
\end{align*}
which completes the proof of Lemma \ref{lem:incl cone}.
    \end{proof}

\subsection{Scores}
To each point in $K$, we associate a random variable that depends on the point and the point process and that we call \textit{score}. It represents the contribution of that point to the number of $k$-faces of the $n$-th layer of the convex hull hull peeling of the point process.
For any $x\in K$, $n\ge 1$, $k\in \{0,\ldots,d-1\}$ and any point set $X$ of $\R^d$, we introduce the function
$$\xi_{n,k}(x,X):=  \left\{\begin{array}{ll}\frac{1}{k+1} \text{card}(\cF_{n,k}(x,X)) &\mbox{ if $x\in \partial \tconv_{n}(X\cup\{x\})$}\\0 &\mbox{ otherwise}\end{array}\right.$$
where $\cF_{n,k}(x, X)$ is the set of all $k$-faces  of 
$\partial\tconv_n(X \cup \{x\})$ containing $x$. 
The factor $\frac{1}{k+1}$ is needed to take into account the fact that the faces are counted multiple times since a $k$-face contains a.s. $(k+1)$ points of $\cP_{\la}$. In particular, we get the identity
\begin{equation}\label{eq:decompavantrescaling}
N_{n,k,\la}=\sum_{x\in \cP_\la}\xi_{n,k}(x,\cP_\la).    
\end{equation}
We now extend this notion of score to the rescaled model. Let $\la \in [1,\infty]$ and $Y$ be a locally finite subset of $W_\la$.
For any $w \in W_\la$, we define the score of $w$ in the rescaled model as
	\begin{equation}\label{eq:egalie scores}
	\xi_{n,k}^{(\la)}(w, Y) = \xi_{n,k}\left([T^{(\la)}]^{-1}(w), [T^{(\la)}]^{-1}(Y)\right).
	\end{equation}

When $w \in \partial \Phi_n^{(\la)}(Y)$, we put
\begin{equation}\label{eq:def xichapeau}
    \hat{\xi}_{n,k}^{(\la)}(w, Y) := \left\{\begin{array}{ll}\frac{1}{k+1} [\text{number of } k-\text{faces of } \partial \Phi_n^{(\la)}(Y) \text{ containing } w]&\mbox{ if $w\in \partial \Phi_n^{(\la)}(Y)$}\\0&\mbox{ otherwise}\end{array}\right.
\end{equation}
where as in \cite[Section 4]{CY3}, we call $k$-face of $\partial\Phi_n^{(\la)}(Y)$ the image by $T^{(\la)}$ of a cone-extreme $k$-face of $[T^{(\la)}]^{-1}(\partial\Phi_n^{(\la)}(Y))$.
 
 Analogously, we can decompose the defect volume of the $n$-th layer into the sum over a sequence of scores, denoted by $\xi_{V,n}$, which are defined in the exact same way as the volume score of the convex hull, see \cite[Section 2]{CY3}. This does not allow to decompose exactly $V_{n,\la}$ into the sum of scores for the same reasons as in \cite[Section 2.1]{CY3}. Nevertheless, we get an asymptotic equality similar to \eqref{eq:decompavantrescaling} in expectation, variance and in probability, i.e.
 \begin{equation}\label{eq:decomp vol exp}
\E[\Vol_d(\conv_n(\cP_\la))]  =  \E[\frac{1}{\la} \sum_{x \in \cP_\la} \xi_{V,n}(x, \cP_\la)]+ o(\la^{-1}\log^{d-1}(\la)),     
 \end{equation}
\begin{equation}\label{eq:decomp vol var}
\Var[\Vol_d(\conv_n(\cP_\la))]
 = 
 \Var[\frac{1}{\la} \sum_{x \in \cP_\la} \xi_{V,n}(x, \cP_\la)]+ o(\la^{-2}\log^{d-1}(\la)).
\end{equation}
and
\begin{equation}\label{eq:decomp vol P}
\Vol_d(\conv_n(\cP_\la))-\frac{1}{\la} \sum_{x \in \cP_\la} \xi_{V,n}(x, \cP_\la)\overset{\P}{\to} 0    
\end{equation}
where $\overset{\P}{\to}$ denotes the convergence in probability.
The order of magnitude of the remainders in \eqref{eq:decomp vol exp} and \eqref{eq:decomp vol var} are calibrated so that they will be negligible in front of the expectation and variance of $\Vol_d(\conv_n(\cP_\la))$ eventually. We omit the proof of these equalities as it relies on arguments similar to \cite[end of proof of Lemma 3.7]{CY3} combined with Lemma \ref{lem:nb points sandwich A la}. 
 
 These scores have their counterparts in the rescaled picture, denoted by $\xi^{(\la)}_{V,n}$ and $\hat{\xi}^{(\la)}_{V,n}$, $\la\in [1,\infty]$.
 In particular, for $w_0  \in \R^{d-1} \times \R$ on the $n$-th layer of the peeling of $\cP$, the rescaled limit defect volume score is defined as
 \begin{equation}\label{eq:def xichapeau vol}
     \hat{\xi}_{V,n}^{(\la) }(w_0, \cP) := \frac{1}{d \sqrt{d}} \int_{v \in \text{Cyl}(w_0)} \exp(d \cdot \partial(\Phi_n^{(\la)})(v))\dd v
 \end{equation}
 where $\partial(\Phi_n^{(\la)})$ is seen as a function from $\R^{d-1}$ to $\R$ and 
 $\text{Cyl}(w_0)$ is the projection onto
 $\R^{d-1}$ of the hyperfaces of $\Phi_n^{(\la)}$ containing $w_0$. This score is $0$ when $w_0$ is not on the $n$-th layer.

For sake of simplicity, we introduce the following sets of scores : $\Xi$ denotes the set of functionals $\xi_{n,k}$ and $\xi_{V,n}$ for $n\in \mathbb{N}^*$ and $k \in \{0, \ldots ,d-1 \}$ while the counterpart in the rescaled model $\Xi^{(\la)}$ denotes the set of functionals $\xi^{(\la)}_{n,k}$ and $\xi^{(\la)}_{V,n}$ for $n\in \mathbb{N}^*$, $k \in \{0, \ldots ,d-1 \}$ and $\la \in [1, \infty]$.

A consequence of Proposition \ref{lem:all cone-extreme} and Lemma \ref{lem:image hyperplane} is that on the event $\tilde{A}_\la$, for any choice of $\xi\in \Xi$, the score $\xi$ of a point $ x \in \cP_\la \cap Q_0$ can be rewritten as $\hat{\xi}^{(\la)}(T^{(\la)}(x), \cP^{(\la)})$. More precisely, putting $w = T^{(\la)}(x)$, we have for any $\xi \in \Xi$ with rescaled counterpart $\xi^{(\la)}$, $\la \in [1, \infty)$ and for $x \in Q_0$
\begin{equation}\label{eq:equality score}
    \xi(x,\cP_\la)\mathds{1}_{\tilde{A}_\la} = \hat{\xi}^{(\la)}(w, \cP^{(\la)})\mathds{1}_{\tilde{A}_\la} =  \xi^{(\la)}(w, \cP^{(\la)})\mathds{1}_{\tilde{A}_\la}.
\end{equation}
    In fact, because of boundary effects, the scores of the points close to the boundary of $Q_0$ may be different when the peeling is restricted to $Q_0$, making the equation above possibly false for these points. A solution to that issue would consist in applying the scaling transform in a larger cube $Q_1$. Because of Lemma \ref{lem:distance no edge}, the scores of the points in $Q_0$ would only depend on the points in $Q_1$ for a suitable choice of $Q_1$.
    The rescaled model would then be constructed inside a window slightly larger than $W_\la$. This would not lead to any significant change and the identity \eqref{eq:equality score} would then be true for any point in $Q_0$. For sake of brevity, we ignore these considerations in the rest of the paper and assume that
    \eqref{eq:equality score} holds for any point $x\in Q_0$.

For any $\la \in [1,\infty]$, $\xi\in \Xi$ and $(v_0, h_0), (v_1, h_1) \in W_\la$, we introduce the two-point correlation function
\begin{align*}
    c^{(\la)}((v_0,h_0), (v_1, h_1)) :=
    \E[ \hat{\xi}^{(\la)}((v_0, h_0), \cP^{(\la)} \cup \{ (v_1, h_1) \}) \hat{\xi}^{(\la)}((v_1, h_1), \cP^{(\la)} \cup \{ (v_0, h_0) \}) ] \\- \E[\hat{\xi}^{(\la)}((v_0, h_0), \cP^{(\la)})]\E[\hat{\xi}^{(\la)}((v_1, h_1), \cP^{(\la)})].
\end{align*}
This function plays a crucial role in the proof of the convergence of the variance of both $N_{n,k,\la}$ and $V_{n,\la}$.

Let us denote by $S(d)$ the regular
$d$-dimensional simplex of edge length $\sqrt{2}d$ given by
\begin{equation}\label{eq:def S(d)}
S(d) := \{(x_1, \ldots, x_d) \in (-\infty, 1] : \sum_{i=1}^d x_i = 0 \} .
\end{equation}
Now that we have introduced every notation involved in the limiting constants of Theorem \ref{thm:theoremeprincipalpolytope} and \ref{thm:theoremeprincipalvolumepolytope}, we can state more precise results which clearly imply Theorems \ref{thm:theoremeprincipalpolytope} and \ref{thm:theoremeprincipalvolumepolytope}, thanks to \eqref{eq:decompavantrescaling} for the number of $k$-faces and thanks to \eqref{eq:decomp vol exp}, \eqref{eq:decomp vol var} and \eqref{eq:decomp vol P} for the defect volume.
\begin{thm}\label{thm:theoreme principal amelio}
For any $\xi\in \Xi$, we let $Z:=\sum_{x\in{\mathcal P}_\la}\xi(x,{\mathcal P}_\la)$ and obtain
	$$\lim\limits_{\lambda \rightarrow \infty} \log^{-(d-1)}(\la){\mathbb{E}[Z]}
	= f_0(K) d^{-d + \frac32} \Vol_d(S(d))
	\int_{-\infty}^\infty \E[\xi^{(\infty)}((0,h_0), \cP)] e^{d h_0} \dd h_0
	\in (0,\infty)$$
    and
	$$
	\lim\limits_{\lambda \rightarrow +\infty}  \log^{-(d-1)}(\la){\text{\normalfont Var}[Z]}
	= f_0(K) d^{-d + 1} \Vol_d(S(d))(I_1(\infty) + I_2(\infty) )\in (0,\infty)
	$$
	where
	$$I_1(\infty) := \sqrt{d} \int_{-\infty}^\infty \E[\xi^{(\infty)}((0,h_0), \cP)^2]e^{dh_0} \dd h_0$$
	and
	$$I_2(\infty) := d \int_{-\infty}^\infty \int_{\R^{d-1}} \int_{-\infty}^\infty c^{(\infty)}((0,h_0), (v_1, h_1)) e^{d(h_0 + h_1)} \dd h_1 \dd v_1 \dd h_0 .$$
     Furthermore when $\la\to\infty$, we have
        \begin{equation*}
        \sup_t \Bigg|
        \P\left(\frac{Z - \E[Z]}{\sqrt{\Var[Z]}} \le t \right)
        - \P(\mathcal{N}(0,1) \le t )\Bigg| = O\left(\log^{-\frac{d-1}{2}}(\la) (\log \log \la)^{15d^2} \right).
    \end{equation*}
 \end{thm}
The next two sections are designed to pave the way for the proof of Theorem \ref{thm:theoreme principal amelio} which is postponed to Section \ref{sec:main results}.

\section{Stabilization and limit theory near a vertex}
\label{sec:stab polytopes}
The variables of interest, namely $N_{n,k,\la}$ and $V_{n,\la}$, have been decomposed into sums of scores, see \eqref{eq:decompavantrescaling} and \eqref{eq:decomp vol exp}, \eqref{eq:decomp vol var}, \eqref{eq:decomp vol P} respectively. These scores behave quite differently, depending on whether the underlying Poisson point falls in the vicinity of a vertex of the mother body $K$ or far from the vertices. In this section, we concentrate on the sum of the scores near a vertex, i.e. in the set $p_d(\mathscr{V}_i, \delta)$ introduced at \eqref{eq:parallelepiped} for some particular choice of $i$ and $\delta$. More precisely, for any $\xi\in \Xi$, $i\in \{1,\ldots,f_0(K)\}$ and $\delta\in (0,\frac12)$, we introduce
 \begin{equation}\label{eq:defZetZi}
 Z:= \sum_{x \in \cP_\la} \xi(x, \cP_\la)\;\;\mbox{ and }\;\;
Z_i(\delta) :=  \sum_{x \in \cP_\la \cap p_d(\mathscr{V}_i, \delta)} \xi(x, \cP_\la).     
 \end{equation}
Both $Z$ and $Z_i(\delta)$ depend on $\la$ (and also on the choice of the functional $\xi$) but for sake of readability, we omit this dependency in the notation although we intend to prove asymptotic results when $\la\to\infty$.

The final goal of this section is the convergence of the expectation and the variance of the variables $Z_i(\delta_0)$ where $1\le i\le f_0(K)$ and
$\delta_0 = \exp(-\log^{1/d}(\la))$ as introduced before equation \eqref{def:Wla} in Section \ref{sec:rescaling}. Starting from \eqref{eq:defZetZi}, we apply Mecke's theorem then a rescaling to rewrite both the expectation and the variance as an integral of a functional which involves either an expected score or the correlation function introduced in Section \ref{sec:rescaling}. The proof of the convergence then relies on the use of the dominated convergence theorem, in the spirit of \cite{CY3}.
In order to get the required convergence and domination of the integrand, we follow a series of papers including \cite{CY3} and \cite{CQ1}, i.e. we appeal to a crucial ingredient, which consists in proving stabilization results for rescaled scores. In other words, we show that with probability exponentially
close to one, the score of a point $w$ only
depends on the points located in a 
neighborhood of $w$. An additional difficulty compared to the case of the first layer studied by \cite{CY3} comes from the lack of an easy criterion for determining the layer number of a point without the knowledge of all the preceding layers.
To tackle this problem, we prove a lemma on the localization in height of the $n$-th layer, as in the case of the unit ball \cite{CQ1}.

Beforehand, 
we need to introduce useful notation for several types of cylinders that are used in the rest of the paper. For any $v \in \R^{d-1}$ and $r > 0$, $C_v(r)$ denotes the vertical cylinder $B_{d-1}(v,r) \times \R$ with the convention $C(r)=C_0(r)$. 
We also define the truncated cylinders
    $C_v^{\geq t}(r) := C_v(r) \cap \{ (v',h') \in \R^d : h' \geq t \}$,
    $C_v^{\leq t}(r) := C_v(r) \cap \{ (v',h') \in \R^d : h' \leq t \}$ and  
    $C_v^I(t) := C_v(r) \cap \{ (v',h') \in \R^d : h' \in I \}$ for any $t>0$ and any interval $I\subset \R$.

In this first lemma, we show that the maximal height of the Poisson points on the $n$-th layer of the cone-like hull peeling inside a cylinder is bounded with a probability going to
1 exponentially fast with respect to the bound. This represents an essential ingredient of the proof of the stabilization result in height of Lemma \ref{lem:stabilisation hauteur}. 

\begin{lem}\label{lemme hauteur max}
	For all $n \geq 1$ and $\varepsilon > 0$, there exist $\lambda_0 \geq 1$ and $c > 0$ such that for all $t> 0$, $r \geq \varepsilon$, $v_0 \in \R^{d-1}$ and $\lambda \in [\lambda_0, \infty]$ we have 
		 \begin{align*}
		 \mathbb{P}\left( \exists (v, h) \in \partial \Phi_n^{(\lambda)}\left(\mathcal{P}^{(\lambda)} \cap \text{C}_{v_0}(r)\right) \cap {\mathcal P}^{(\lambda)}\cap \text{C}_{v_0}(r) \text{ with } h \geq t \right) 
		 \\
		 \leq c r^{(n-1)(d-1)} \exp\left( - e^{t/c} \right)
		 \end{align*}
		 and
		 \begin{align*} 
		 \mathbb{P}\left( \exists (v, h) \in \partial \Phi_n^{(\lambda)}\left(\mathcal{P}^{(\lambda)}\right) \cap {\mathcal P}^{(\la)}\cap\text{C}_{v_0}(r) \text{ with } h \geq t \right)
		 \\
		 \leq c r^{(n-1)(d-1)}\exp\left( - e^{t/c} \right).
		 \end{align*}

\end{lem}
\begin{proof}
    We only prove the first statement as the proof of the second one is identical.
    We start with the case $\la = \infty$ and explain at the end why it still works for $\la < \infty$.
    
	We show the result by induction, we start with the induction step and assume that $n \geq 2$ and that the result is verified for any $p < n$.
	
	Our first step is to show that for any fixed $w = (v,h) \in C_{v_0}(r)$ with $h \geq t$ the event
	\begin{equation}\label{lemme hauteur max eq w}
		\lbrace w \in \partial\Phi_n\left((\mathcal{P} \cup \lbrace w \rbrace) \cap C_{v_0}(r))\right \rbrace
	\end{equation}
	 occurs with probability smaller than 
	$c\exp\left( - e^{h/c} \right)$.
	We consider a fixed $w$ that verifies these conditions. Then there exists $(v_1,h_1) \in \R^{d-1}\times \R$ such that $w \in \partial\Pi^\downarrow(v_1, h_1)$ and $\Pi^\downarrow(v_1, h_1)$ only contains points of layer number at most $(n-1)$ for the peeling in $C_{v_0}(r)$. This downward cone-like grain contains a fixed circular cone $\cC_w$ with apex $w$ as
	given by Lemma \ref{lem:incl cone}.
	
	From the preceding reasoning and
	because $\cC_w$ does not depend on $(v_1, h_1)$, we deduce that 
    \begin{align*}
			\lbrace w \in \partial\Phi_n\left((\mathcal{P} \cup \lbrace w \rbrace) \cap C_{v_0}(r)\right) \rbrace
		\subset  \{
	\cP \cap \tilde{\cC}_w \subset \cup_{i=1}^{n-1}\partial\Phi_{i}((\cP\cup\{ w\})\cap C_{v_0}(r)) \}
	\end{align*}
	where
	$\tilde{\cC}_w := \mathcal{C}_w \cap \{ (v',h') : h' \geq h/2 \} \cap C_{v_0}(r)$, see Figure \ref{fig:stab1}.
    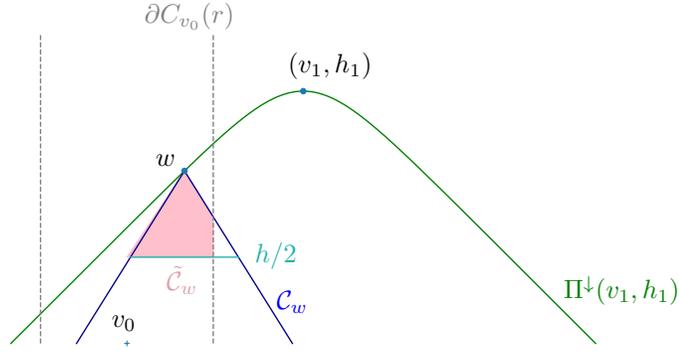
\begin{figure}
        \centering
        \begin{overpic}[width=0.97\linewidth, trim = {3cm, 6cm, 3cm, 6cm}]{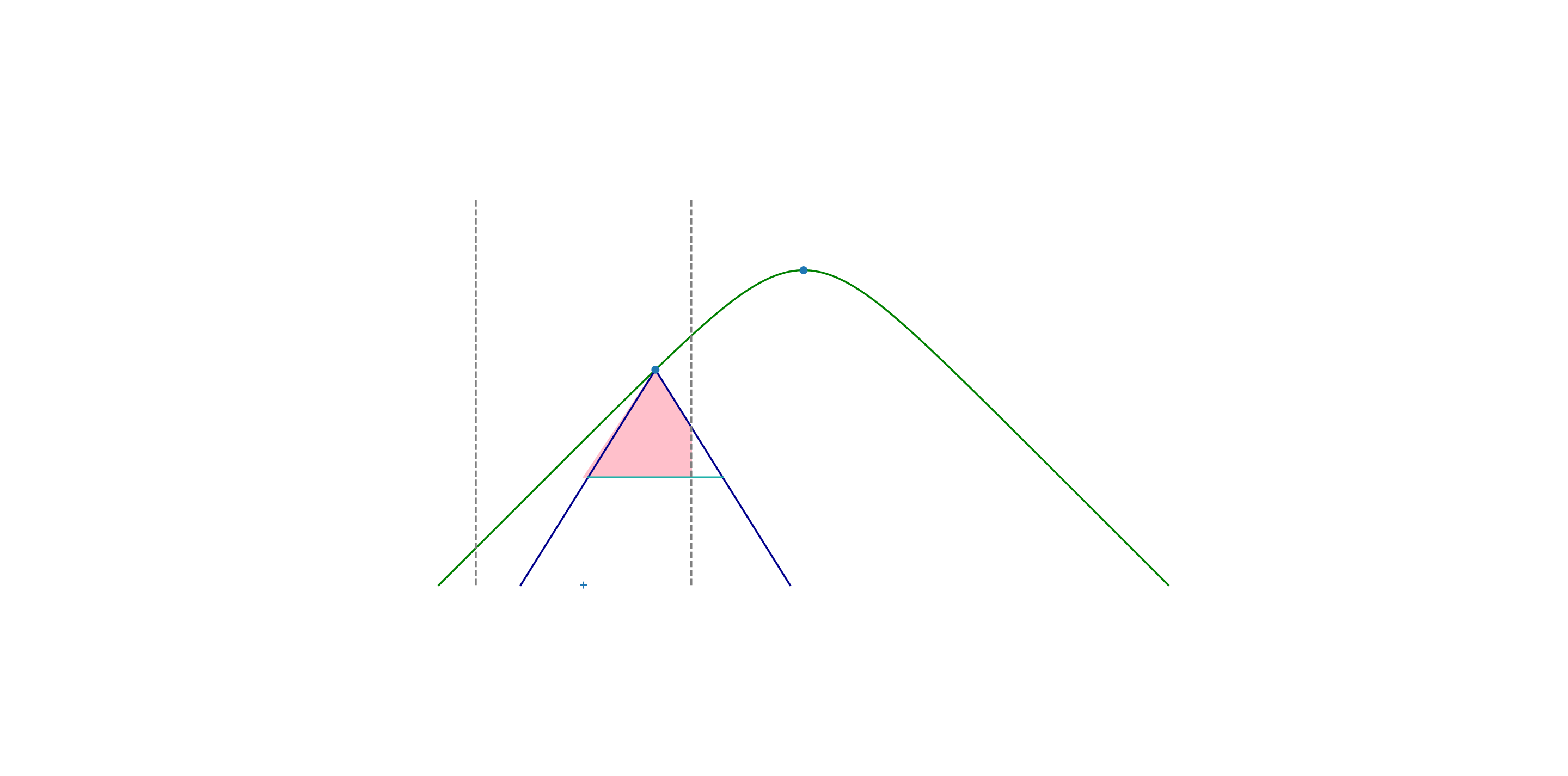}
            \put (50, 24.5) {$(v_1, h_1)$}
            \put (75, 4) {$ \color{green}\Pi^\downarrow(v_1,h_1) $}
            \put (34, 1.4) {$v_0$}
            \put (38, 16) {$w$}
            \put (37, 29) {$\color{grey} \partial C_{v_0}(r)$}
            \put (39, 4.8){$ \color{pink}\tilde{\cC}_w $}
            \put (49, 3){$ \color{blue}\cC_w $}
            \put (47, 7.3) {\color{lightseagreen}$h/2$}
        \end{overpic}
        \caption{Illustration of the set $\tilde{\cC}_w$ of Lemma \ref{lemme hauteur max} in dimension 2.}
        \label{fig:stab1}
    \end{figure}
	The set $\cP \cap \tilde{\cC}_w$ is empty with probability smaller than $c\exp(-e^{h/c})$.
	If $\cP \cap \tilde{\cC}_w$ is not empty, it contains a point at height larger than $h/2$ on a layer at most $(n-1)$, which happens with probability smaller than $c r^{(n-2)(d-1)} \exp(-e^{h/c})$ thanks to the induction hypothesis. Consequently we have 
	\begin{equation}\label{lemme hauteur max eq w fixe 2}
	\P(\lbrace w \in \partial\Phi_n\left((\mathcal{P} \cup \lbrace w \rbrace) \cap C_{v_0}(r)\right) ) \le c r^{(n-2)(d-1)} \exp(-e^{h/c}) . 
	\end{equation}
	We can write
	\begin{align*}
	\mathbb{P}\left( \exists (v, h)  \in \partial\Phi_n(\mathcal{P} \cap C_{v_0}(r)) \cap \mathcal{P}\cap C_{v_0}(r)
	 \text{ with } h \geq t \right)& \\
	\leq \mathbb{E}\left[ \sum_{w \in \mathcal{P} \cap C_{v_0}^{\geq t}(r) }
	\mathds{1}_{w \in \partial\Phi_n(\mathcal{P} \cap C_{v_0}(r))}  \right]&.
	\end{align*}
	We combine this with the Mecke formula and \eqref{lemme hauteur max eq w fixe 2} to get 
		\begin{align*}
	\mathbb{P}&\left( \exists (v, h)  \in \partial\Phi_n(\mathcal{P} \cap C_{v_0}(r)) \cap C_{v_0}(r)
	 \text{ with } h \geq t \right) \\
	&\leq \int_{\Vert v - v_0 \Vert \leq r} \int_{h \in \left] t, \infty\right[} 
	\mathbb{P}\left((v,h) \in \partial \Phi_n
	((\mathcal{P} \cup \{ (v, h) \} ) \cap C_{v_0}(r)) \right) \sqrt{d} e^{dh} \dd h \dd v\\
	&\leq \int_{\Vert v - v_0 \Vert \leq r} \int_{h \in \left] t, \infty\right[} c r^{(n-2)(d-1)} \exp(-e^{h/c}) \sqrt{d} e^{dh} \dd h \dd v\\
	&\leq c r^{(n-1)(d-1)}\exp(-e^{t/c}).
	\end{align*}
	This proves the induction step.
	
	The base case is easier as the same reasoning holds except that we only have to consider the case where the circular cone $\cC_w$ is empty.
	\\~\\
	\noindent \textit{Case $\la <\infty$.}
	We claim that as long as the circular cone $\cC_w$ is included in $W_\la$ the proof remains valid because the intensity of the point process $\cP^{(\la)}$ is the same as the intensity of $\cP$ except that the process is restricted to $W_\la$. 
	Let us describe how we can ensure that $\cC_w$ is included in $W_\la$.
	Let $C_{m} = C_m(\la)$ be the largest cone contained in $W_\la$ with the same apex as $W_\la$. As for $\la \geq \la'$, $W_\la$ is only a vertical translation of $W_{\la'}$, the aperture of $C_m(\la)$ stays the same. Thus, for $\la\geq \la_0$ and any $(v',h')\in W_\la$ the cone $C_m(\la_0)$ is contained in $W_\la$. If we take for $\cC_w$ the smallest cone between the translation of $C_m(\la_0)$ with apex $(v,h)$ and the choice of $\cC_w$ in the case $\la = \infty$ we obtain a circular cone that is contained in $W_\la$ and verifies the volume estimates in the proof of the case $\la=\infty$. 
	This makes the proof valid in the case $\la < \infty$.
	\end{proof}
	
	When $w \in \partial \Phi_{n}(\cP^{(\la)})$
we define $H_n^{(\la)}(w, \cP^{(\la)})$ as
the maximal height of an apex of a down cone-like grain belonging to the $n$-th layer and containing $w$.
Otherwise, we put $H_n^{(\la)}(w, \cP^{(\la)}) = 0$.
The next result has the flavor of a stabilization in height. It provides an upper bound on the distribution tail of the height of the rescaled cone-like facets containing a fixed point from the $n$-th layer.
\begin{lem}\label{lem:stabilisation hauteur}
    There exist a constant $c > 0$ and $\la_0 \geq 1$ such that for all $\la \in [\la_0, \infty]$, $(v_0, h_0) \in W_\la$ and $t \geq h_0 \vee 0$, 
    $$\P(H_{n}^{(\la)}((v_0,h_0), \cP^{(\la)}) \geq t) \leq c\exp(-e^{t/c}) .$$
    Furthermore, for any $\varepsilon>0$, there exist a constant $c > 0$ and $\la_0 \geq 1$ such that for all $\la \in [\la_0, \infty]$, $(v_0, h_0) \in W_\la$, $t \geq h_0 \vee 0$ and $r>\varepsilon$, 
    $$\P(\exists s \geq r :  H_{n}^{(\la)}((v_0,h_0), \cP^{(\la)} \cap C(s)) \geq t) \leq c\exp(-e^{t/c}) .$$
\end{lem}
\begin{proof}
    Let us concentrate on the first statement.
    We can assume that $n \geq 2$ as the case $n=1$ is proved in \cite[Lemma 5.1]{CY3}.
    Let us assume that $H_{n}^{(\la)}((v_0,h_0), \cP^{(\la)}) \geq t$. It implies that there exists a downward cone-like grain with apex $(v_1, h_1) \in \partial [\Pi^\uparrow]^{(\la)}(v_0, h_0) 
    $ and $h_1 \geq t$  that only contains points on layer at most $(n-1)$.
    
    Let $c_{min}$ be the minimum between the aperture of the circular cone contained in $\Pi^{\downarrow}(v_1,h_1)$ given by Lemma \ref{lem:sandwich cone} and the aperture of the largest circular cone contained in $W_\la$ with the same apex as $W_\la$. We can fit a circular cone $\cC_{(v_1, h_1)}$ with apex $(v_1, h_1)$ and aperture $c_{min}$ in $\Pi^{\downarrow}(v_1, h_1)$ that is entirely contained in $W_\la$.
    
    Let $\tilde{\cC} := \cC_{(v_1, h_1)} \cap \{ (v',h') : h_1/4 \leq h' \leq h_1/2 \} \cap C_{v_1}(t)$, represented on Figure \ref{fig:stab2}.
    \begin{figure}
        \centering
        \begin{overpic}[width=0.97\linewidth, trim={2cm 6cm 2cm 6cm},clip]{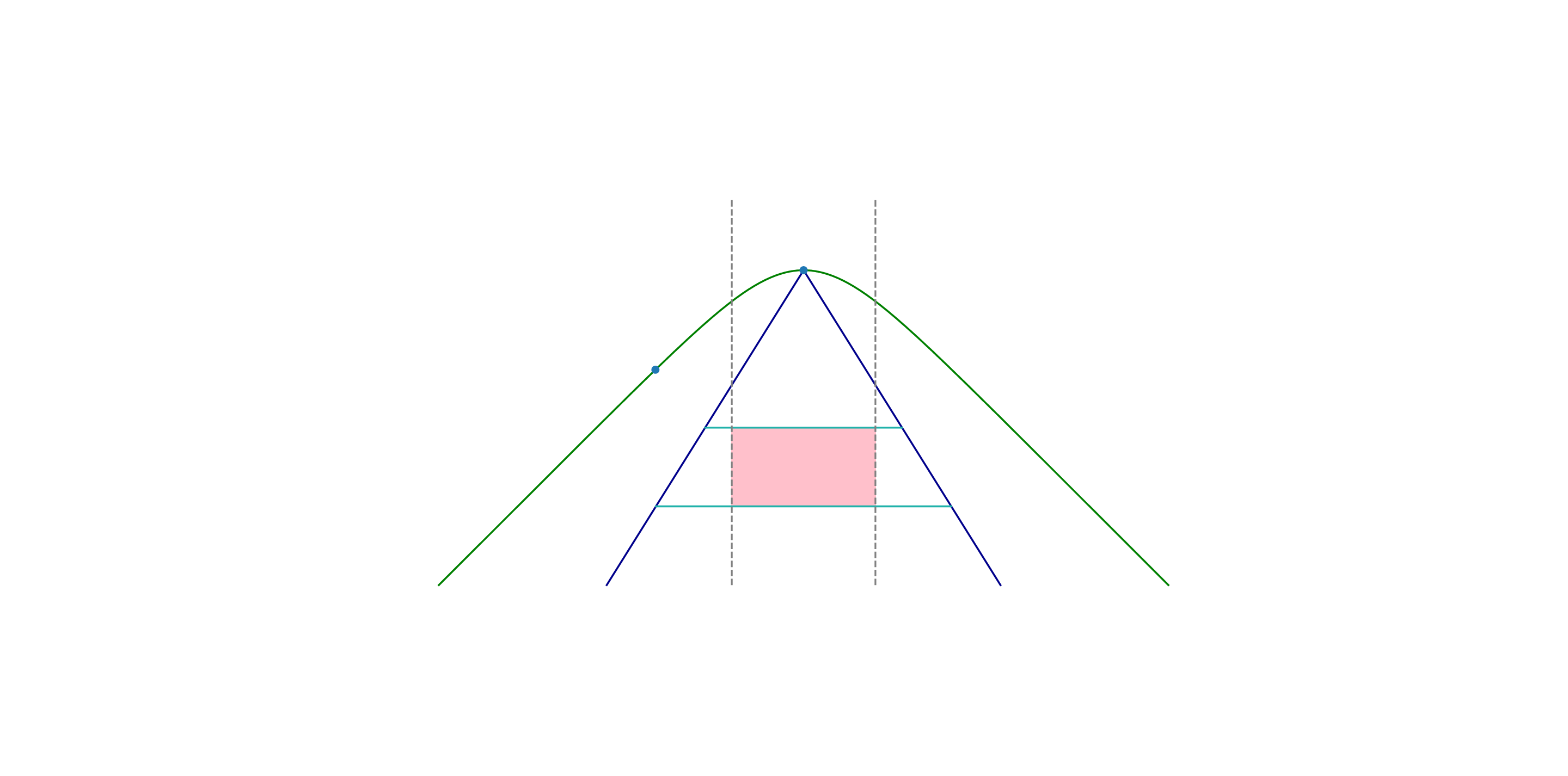}
        \put (76, 2){$ \color{green}\Pi^\downarrow(v_1,h_1) $}
        \put (50, 2.5){$ \color{pink}\tilde{\cC} $}
        \put (35, 15){$(v, h)$}
        \put (48, 23.6){$(v_1, h_1)$}
        \put (42, 28){$\color{grey} \partial C_{v_1}(t)$}
        \put (59, 10.5) {$\color{lightseagreen} h_1/2$}
        \put (62.5, 5) {$\color{lightseagreen} h_1/4$}
        \end{overpic}
        \caption{Illustration of the set $\tilde{\cC}$ of Lemma \ref{lem:stabilisation hauteur} in dimension 2.}
        \label{fig:stab2}
    \end{figure}
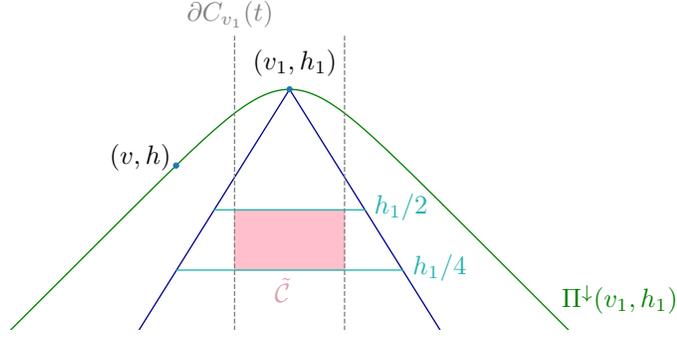
    
    Then either $\cP^{(\la)} \cap \tilde{\cC} = \varnothing$, which happens with probability smaller than $c \exp(-e^{h_1/c})$ or it contains at least one point which is thus on layer at most $(n-1)$. This last event occurs with probability smaller than $c\exp(-e^{h_1/c})$ because of Lemma \ref{lemme hauteur max}.
    
    Discretizing and integrating over $(v_1, h_1) \in [\Pi^\uparrow]^{(\la)}(v_0, h_0), h_1 \in [t, \infty)$ yields the desired result.
    
    The second statement follows along similar lines with a proper updating of the cylinder $C_{v_1}(t)$ in the intersection defining the set $\tilde{\cC}$. 
\end{proof}
	

We are now ready to prove in Proposition \ref{Stabilisation points poly} below a result of stabilization in width for the score $\hat{\xi}_{n,0}^{(\la)}$ corresponding to the number of vertices of the $n$-th layer. It is a necessary step towards the proof of the stabilization in width for the general score $\hat{\xi} \in \Xi$ which is stated in Proposition \ref{Stabilisation k faces} that follows directly.

For any $h_0$ we write
$$\tilde{h}_0 := (\frac{6}{\underline{c}} \log(d)) \vee (-\frac{6}{\underline{c}} h_0) \mathds{1}_{\{ h_0 < 0 \}}.$$
For any $\la \geq 1$, $(v_0, h_0) \in W_\la$ and $\xi\in \Xi$, we define the radius of stabilization as
$$ R^{(\la)}(v_0, h_0) := \inf\{ r > 0 : \hat{\xi}^{(\la)}((v_0, h_0), \cP^{(\la)}) = \hat{\xi}^{(\la)}((v_0, h_0), \cP^{(\la)} \cap C(s)) \text{ for all } s\geq r \}.$$
In Proposition \ref{Stabilisation points poly} below, we consider the particular case $\xi=\xi_{n,0}$, i.e. $\hat{\xi}(x,{\mathcal P}^{(\la)})$ is the indicator function of the event that $x$ lies on the $n$-th layer of the cone-like peeling of ${\mathcal P}^{(\la)}\cup\{x\}.$
\begin{prop}\label{Stabilisation points poly}
	We take $\xi^{(\la)}=\hat{\xi}^{(\la)}_{n,0}$ with $n \geq 1$ and denote by $R_{n,0}^{(\la)}$ the corresponding radius of stabilization. There exist $\lambda_0 \geq 1$ and $c > 0$  such that for any 
 $(v_0, h_0) \in W_\la$, $r \geq 3^{n-1}\tilde{h}_0$
 and $\lambda \in \left[\lambda_0, +\infty\right]$ we have 
	$$ \mathbb{P}\left( R_{n,0}^{(\lambda)}(v_0,h_0) 
	\geq r \right) \leq c \exp\left( -r / c \right) .$$
\end{prop}
\begin{proof}
    For sake of simplicity, we only treat $v_0 = 0$, the proof for general $v_0$ is analogous.
    Let us 
    write $R_{n,0}$ for $R_{n,0}^{(\la)}(0,h_0)$. We prove Proposition \ref{Stabilisation points poly} by induction on $n$. The base case was already proved in \cite[Lemma 5.2]{CY3}. We assume that the result holds for any $p < n$ with $n \geq 2$.
    
    Let us suppose that $R_{n,0} \geq r$. We can further assume that $(0,h_0) \in \partial \Phi^{(\la)}_n(\cP^{(\la)})$ and
    $(0,h_0) \not\in \partial \Phi^{(\la)}_n(\cP^{(\la)} \cap C(r))$ as the other case, which is $(0,h_0) \not\in \partial \Phi^{(\la)}_n(\cP^{(\la)})$ and $(0,h_0) \in \partial \Phi^{(\la)}_n(\cP^{(\la)} \cap C(r))$, can be proved almost identically.
    Let us write  $c_1 := \frac{\underline{c}}{12}$ with $\underline{c}$ the constant from Lemma \ref{lem:sandwich cone}   and
    $c_2 = \inf(c_1, \frac{1}{6d c_{stab}})$ with $c_{stab}$ the minimum (that makes the exponential the smallest) of all the constants $c$ in the exponential given by the induction hypothesis for $p < n$.
    
    Let $l = \ell^{(\la)}((0,h_0), \cP^{(\la)} \cap C(r))$. In particular, due to Lemma \ref{lem:dalalreecrit2}, 
    $\ell^{(\la)}((0,h_0), \cP^{(\la)}) = n > l$.
    We choose a cone-like grain $[\Pi^\downarrow]^{(\la)}(v_1, h_1)$, with $(v_1, h_1) \in [\Pi^\uparrow]^{(\la)}(0,h_0)$, that only contains points of layer at most $(l-1) \leq (n-2)$ for the peeling of $\cP^{(\la)}\cap C(r)$  and at least one point of layer number $\geq (n-1)$ for the whole peeling.
    As Lemma \ref{lem:stabilisation hauteur} implies that $$\P(H_{l}^{(\la)}((0,h_0), \cP^{(\la)}\cap C(r)) \geq c_2 r) \leq c\exp(-e^{r/c}),$$
     we can assume that $h_1 \leq c_2 r$.
     We write $E$ for the corresponding event.
     We assert that
    \begin{equation}\label{eq:incl cyl}
    [\Pi^\downarrow]^{(\la)}(v_1, h_1) \cap \{(v',h') : h' \geq -c_1 r\} \subset C(2r/3).
    \end{equation}
    Indeed let us take $(v,h) \in [\Pi^\downarrow]^{(\la)}(v_1, h_1) \cap \{(v',h') : h' \geq -c_1 r\}$. The first step is to prove that $\Vert v - v_1 \Vert \leq r / 3$.
    The equation of the cone-like grains and Lemma \ref{lem:sandwich cone} imply that
    $$h \leq h_1 - G(v - v_1) \leq h_1 - \underline{c} \Vert v - v_1 \Vert + \log(d).$$
    In turn 
    $$ \Vert v - v_1 \Vert \leq \frac{1}{\underline{c}} [h_1 - h + \log(d)].$$
    As $h_1$ and $h$ are both bounded by $c_1 r = \frac{\underline{c}}{12} r$ and $\log(d) \leq \frac{\underline{c}}{6} r$ (see the hypothesis $r \geq \Tilde{h_0}$ we obtain
    \begin{equation}\label{eq:ineq stab}
        \Vert v - v_1 \Vert \leq r/3.
    \end{equation}
    Additionally, a similar reasoning yields
    $$h_1 = h_0 + G(-v_1) \geq h_0 + \underline{c} \Vert v_1 \Vert - \log(d) $$
    and then $\Vert v_1 \Vert \leq r / 3$. Combining this, \eqref{eq:ineq stab} and the triangle inequality gives \eqref{eq:incl cyl}.
    
    Now we write $E$ as the disjoint union $E=E_1\cup E_2$ where
     $$E_1 := E \cap \{ [\Pi^\downarrow]^{(\la)}(v_1,h_1) \cap \{ h \leq -c_1 r \} \cap \cP^{(\la)} \cap \partial\Phi^{(\la)}_{n-1}(\cP^{(\la)}) = \varnothing \}$$ and
    $$E_2 = E \cap \{ [\Pi^\downarrow]^{(\la)}(v_1,h_1) \cap \{ h \leq -c_1 r \} \cap \cP^{(\la)} \cap \partial\Phi^{(\la)}_{n-1}(\cP^{(\la)}) \neq \varnothing \} .$$
    \noindent (a) If $E_1$ occurs, we have a point $(v,h)$ of layer $(n-1)$ for the peeling of $\cP^{(\la)}$ with $h \geq -c_1 r$. It is in particular included in $C(2r/3)$, see \eqref{eq:incl cyl}. This point must have a stabilization radius greater than $r/3$ as its layer number can not be $(n-1)$ for the peeling of $\cP^{(\la)} \cap C(r)$.
	\begin{align*}
	\mathbb{P}\left( E_1 \right) &\leq \mathbb{P}\left(\bigcup_{m \leq n-1} \bigcup_{w \in \mathcal{P^{(\la)}} \cap C^{ [ -c_1 r , c_2 r  ]}(2r/3)} 
	\{ R_{m,0}^{(\la)}(w, \mathcal{P^{(\la)}}) \geq r/3 \} \right)\\
	&\leq \mathbb{E}\left[ \sum_{m \leq n-1} \sum_{w \in \mathcal{P^{(\la)}} \cap C^{[-c_1r, c_2 r]}(2r/3)}
	\mathds{1}_{ R_{m,0}^{(\la)}(w, \mathcal{P^{(\la)}}) \geq r/3} \right].
	\end{align*}
	Now we use the Mecke formula, the induction hypothesis and the calibration of $c_2$ to obtain
	\begin{align*}
	\mathbb{P}\left( E_1 \right) &\leq \sqrt{d} \sum_{m \leq n-1}\int_{C^{[-c_1r, c_2 r]}(2r/3)} \mathbb{P}\left( R_{m,0}^{(\la)}((v,h), \mathcal{P^{(\la)}}) \geq r/3 \right) e^{dh} \dd v \dd h \leq c\exp(-r / c).
	\end{align*}
	
	\noindent (b) If $E_2$ occurs, it means that $[\Pi^\downarrow]^{(\la)}(v_1,h_1) \cap \{ h \leq -c_1 r \} \cap \cP^{(\la)} \neq \varnothing$. This happens with probability smaller than $c \exp(-r/c)$, see \cite[bottom of p.32]{CY3}.
	
    Combining a) and b) we get $\P(E) \leq c \exp(-r/c).$

\end{proof}
Proposition \ref{Stabilisation points poly} shows that the particular score $\xi=\xi_{n,0}$, i.e the score correponding to the number of vertices of the $n$-th layer of the peeling, satisfies a stabilization property. In Proposition \ref{Stabilisation k faces} below, we extend that result to any score $\xi\in \Xi$.


\begin{prop}\label{Stabilisation k faces}
	For any $\xi\in \Xi$, there exist $\lambda_0 \geq 1$ and $c > 0$  such that for any $(v_0, h_0) \in W_\la$, 
 $ r \geq 3^n\tilde{h}_0$
 and $\lambda \in \left[\lambda_0, \infty\right]$ we have 
	$$ \mathbb{P}\left( R^{(\lambda)}(v_0,h_0) 
	\geq r \right) \leq c \exp\left( -r/c \right) .$$
\end{prop}
\begin{proof}
    Let us consider $(v_0,h_0)$ such that $R^{(\la)}(v_0,h_0)\ge r$. We also write $c_1 = \frac{\underline{c}}{12} $ as in Proposition \ref{Stabilisation points poly} and $c_2 = \inf(c_1, \frac{1}{6d c_{stab}})$ with $c_{stab}$ the constant from Proposition \ref{Stabilisation points poly}.
    
    Guided by the previous arguments used in Proposition \ref{Stabilisation points poly}, we can assume several facts without loss of generality.
    \begin{itemize}
        \item $v_0 = 0$.
    \item $R_{n,0}^{(\lambda)}(0,h_0) 
	< r$. Indeed, the complement event occurs with probability smaller than $c\exp(-r/c)$ thanks to Proposition \ref{Stabilisation points poly}. This implies that $(0,h_0)$ is on the $n$-th layer of the peeling of $\cP^{(\la)}$ if and only if it is on the $n$-th layer of the peeling of $\cP^{(\la)}\cap C(s)$ for any $s\ge r$. 
	\item $(0,h_0)$ is on the $n$-th layer for the peelings of both $\cP^{(\la)}$ and $\cP^{(\la)} \cap C(s)$ for any $s \geq r$. Otherwise, the score $\xi^{(\la)}(0,h_0)$ would be equal to $\xi_{n,0}^{(\la)}(0,h_0)=0$ and would share the same stabilization radius.
	\item $\cP^{(\la)}$ does not meet $C(r)\cap \{(v',h'):h'\le -c_1r\}$ as the complement set occurs with probability smaller than $c\exp(-r/c)$.
	\item $\min(H_n((0,h_0), \cP^{(\la)}), \sup_{s \geq r}H_n((0,h_0), \cP^{(\la)} \cap C(s))) \leq  c_2 r$. Indeed, Lemma \ref{lem:stabilisation hauteur} implies that the complement event occurs with probability smaller than  $c\exp(-e^{r/c})$.
    \end{itemize}
    We denote by $\tilde{E}$ the event corresponding to all of the assumptions above. We consider the set $$\mathcal{U} := \bigcup_{(v_1, h_1) \in \partial[\Pi^{\uparrow}]^{(\la)}(0,h_0), h_1 \leq c_2 r} [\Pi^{\downarrow}]^{(\la)}(v_1,h_1) \cap \{ (v',h'): h' \geq -c_1r \}.$$ 
	On the event $\tilde{E}$, we obtain thanks to \eqref{eq:incl cyl} that  $\mathcal{U}\subset C^{[-c_1r,c_2r]}(2r/3)$. 
	Moreover, because of the assumption on the variables $H_n((0,h_0),\cP^{(\la)}\cap C(s))$, $\mathcal{U}$ contains every point that shares a common hyperface of the $n$-layer with $(0,h_0)$ for the peeling of  $\cP^{(\la)} \cap C(s)$ for any $s \in [r, \infty]$. Thanks to \eqref{eq:def xichapeau} and \eqref{eq:def xichapeau vol}, this implies that for any $\xi\in \Xi$ including the volume score, the knowledge of ${\mathcal P}^{(\la)}\cap {\mathcal U}$ is enough to determine the score at $(0,h_0)$.
 
    We can now proceed as in the proof of  \cite[Lemma 3.6]{CQ1}. We assert that
	\begin{equation}\label{eq:inclusionRnk poly}
	\{R^{(\la)}((0,h_0), \cP^{(\la)})\ge r\} \cap \tilde{E} \subset\{\exists w \in \mathcal{P}^{(\la)}\cap\mathcal{U}  : R_{n,0}^{(\la)}(w, \cP^{(\la)}) \ge r/3\} \cap \tilde{E}.
	\end{equation}
		Indeed, if every point $w \in \mathcal{P}^{(\la)}\cap\mathcal{U} $ verifies $R_{n,0}^{(\la)}(w, \mathcal{P}^{(\la)}) \leq r/3$, then the status of these points with respect to the $n$-th layer is the same for both $\cP^{(\la)}$ and $\cP^{(\la)}\cap C(s),$ for any $s\ge r$. Consequently, the $k$-faces of the $n$-th layer containing $(0,h_0)$ are the same for the peeling of any $\cP^{(\la)}\cap C(s)$, $s\ge r$, which implies that $R^{(\la)}((0,h_0), \cP^{(\la)})\le r$.  
		 Using consecutively \eqref{eq:inclusionRnk poly} and Mecke's formula we have
	\begin{align*}
		\mathbb{P}(\{ R^{(\la)}(0,h_0))\ge r\}
		 \cap E)
			&\leq \mathbb{E}\left[ \sum_{(v,h) \in \mathcal{U}\cap \mathcal{P}^{(\la)}, h \geq -c_1 r} \mathds{1}_{\{R_{n,0}^{(\la)}(v,h) \ge r/3\} \cap E} \right]\\
			&\leq \sqrt{d} \int_{C^{[-c_1 r, c_2 r]}(2r/3)} \mathbb{P}(R_{n,0}^{(\la)}((v,h)) \ge r/3) e^{dh}  \mathrm{d}v \dd h\\
			&\leq c \exp(-r/c).
	\end{align*}
\end{proof}


As previously stated, our strategy to derive the limit of the sum of the scores in the vicinity of the origin relies on the dominated convergence theorem. This requires to dominate the expectation of the score or one of its powers. This is done in Lemma \ref{lem:borne Lp} below. Its proof which relies on Lemmas \ref{Stabilisation k faces} and \ref{lem:stabilisation hauteur} is omitted, as it is
almost identical to the proof given in the case of the first layer, see \cite[Lemma 5.3]{CY3} and also very close to the proof of \cite[Lemma 4.1]{CQ1}.
\begin{lem}\label{lem:borne Lp}
    For all $p \in [1, \infty)$ and $\xi \in \Xi$, there exists a constant $c > 0$ and $\la_0 \geq 1$ such that for all $(v_0, h_0) \in W_\la$ and $\la \geq \la_0$
    $$\E[\hat{\xi}^{(\la)}(v_0, h_0)^p] \leq c_1(|h_0| + 1)\exp\left(-c_2 e^{h_0 \vee 0} \right) .$$
\end{lem}

In the next lemma, we state the convergence of $\E[\hat{\xi}_{n,k}^{(\la)}(v_0, h_0)]$ to
$\E[\hat{\xi}_{n,k}^{(\infty)}(v_0, h_0)]$. Its proof which is identical to the proof of \cite[Lemma 5.4 (a)]{CY3}, is also omitted.
\begin{lem}\label{lem:lim xink}
    For all $(v_0, h_0) \in \R^{d-1}\times \R$ and $\xi \in \Xi$, we have
    $$\lim_{\la \rightarrow \infty} \E[\hat{\xi}^{(\la)}((v_0, h_0),\cP^{(\la)})] =  \E[\hat{\xi}^{(\infty)}((v_0, h_0), \cP)] .$$
\end{lem}


 The next two lemmas imply analogous results of domination and convergence for the two-point correlation function. They are used  in the proof of the convergence of the variance of the number of $k$-faces in a neighborhood of a vertex of $K$. Proofs of these results are again omitted, as they are almost identical to the proofs of \cite[Lemma 5.4 (b)]{CY3}, \cite[Lemma 5.5]{CY3} and \cite[Lemma 5.6]{CY3}. 
\begin{lem}\label{lem:bound cor} For all $n \geq 1$ and $ \xi\in \Xi$, 
there exists a constant $c > 0$ such that
for all $\la \geq 1$ and
$(v_0, h_0), (v_1, h_1) \in W_\la$ 
satisfying 
$$ \Vert v_1 - v_0 \Vert \geq 2 \cdot 3^{n} \max\left( \frac{6}{\underline{c}} \log(d), -\frac{6}{\underline{c}}\mathds{1}_{\{h_0 < 0\}}, -\frac{6}{\underline{c}}\mathds{1}_{\{h_1 < 0\}} \right)  $$
with $\underline{c}$ as in Lemma \ref{lem:sandwich cone},
we have
$$|c^{(\la)}((v_0,h_0), (v_1,h_1))| \leq c(\Vert h_0 \Vert + 1)^c (\Vert h_1 \Vert + 1)^c \exp\left(- \frac{1}{c} (\Vert v_1 - v_0\Vert + e^{h_0 \vee 0} + e^{h_1 \vee 0})\right) .$$
Additionally, 
there is an integrable function $g : \R \times \R^{d-1} \times \R \rightarrow \R+$ such that
for all $\la \in [1, \infty]$ we have
$$ |c^{(\la)}((0,h_0), (v_1,h_1))| e^{dh_0}e^{dh_1} \leq g_{n,k}(h_0, v_1, h_1). $$

\end{lem}

\begin{lem}\label{lem:lim cor}
    For all $h_0 \in \R$, $(v_1, h_1) \in \R^{d-1}\times \R$ and $\xi\in \Xi$, we have
    $$ \lim_{\la \rightarrow \infty} c^{(\la)}((0,h_0), (v_1, h_1)) = c^{(\infty)}((0, h_0), (v_1, h_1)) .$$
\end{lem}

The combination of the previous results allows us to obtain Proposition \ref{prop:asymptotics corner} which provides the asymptotics for the expectation and variance of the contribution of Poisson points in a neighborhood of a vertex of $K$. We recall the definition of $Z_i(\delta_0)$ given at \eqref{eq:defZetZi}. The event $A_\la$ therein is a slight modification of the event $\tilde{A}_\la$ which is introduced later on, at \eqref{eq:def A la}. 
Proposition \ref{prop:asymptotics corner} is used in Section \ref{section:dec var} where the event $A_\la$ is required for technical reasons. 
\begin{prop}\label{prop:asymptotics corner}
    For each $1 \leq i \leq f_0(K)$ and $\xi\in \Xi$, we have
    $$ \lim\limits_{\la \rightarrow \infty}\frac{\E[Z_i(\delta_0)\mathds{1}_{A_\la}]}{\log^{d-1}(\la)} = d^{-d+3/2}\text{\normalfont Vol}_d(S(d)) \int_\R \E[\xi^{(\infty)}((0,h_0), \cP) ]e^{dh_0}\dd h_0  $$
    and
    $$ \lim\limits_{\la \rightarrow \infty}\frac{\Var[Z_i(\delta_0)\mathds{1}_{A_\la}]}{\log^{d-1}(\la)} = I_1(\infty) + I_2(\infty)  $$
    where
    $$I_1(\infty) =  d^{-d+3/2}\text{\normalfont Vol}_d(S(d)) \int_\R \E[\xi^{(\infty)}((0,h_0),\cP)^2] e^{dh}\dd h_0 $$
    and
    $$I_2(\infty) = d^{-d+2}\text{\normalfont Vol}_d(S(d)) \int_{\R^{d-1}\times\R} \int_\R c^{(\infty)}((0,h_0), (v',h_1)) e^{d(h_0 + h_1)}\dd h_0 \dd (v',h_1)  .$$
\end{prop}
\begin{proof}
The method consists in switching to the rescaled picture by rewriting $Z_i(\delta_0)$, thanks to \eqref{eq:equality score}, as
    \begin{equation}\label{eq:rewritingZ_iinrescaledworld}
    Z_i(\delta_0){\mathds{1}}_{A_\la}=\sum_{w\in \cP^{(\la)}\cap W_\la}\xi^{(\la)}(w,\cP^{(\la)})\mathds{1}_{A_\la}.
    \end{equation}
        We then proceed with a direct adaptation of the proof of \cite[Theorem 2.1]{CY3}. As we have already proved the same kind of required results of stabilization and domination adapted to the $n$-th layer in Section \ref{sec:stab polytopes} and an analogous sandwiching result in Section \ref{sec:floating bodies etc}, one can follow the exact same reasoning line by line. No new problem arises and thus, we omit the details.
\end{proof}

\section{Decomposition of the expectation and variance} 
\label{section:dec var}

The aim of this section is to prove Proposition \ref{prop:decompEVar} which combines the two following facts.
\begin{itemize}
    \item The variance is additive over the vertices of $K$, i.e. for a suitable $\delta$ that depends on $\la$, the variables $Z_i(\delta)$, $1\le i \le f_0(K)$, decorrelate asymptotically where we recall that $Z_i(\delta)$ defined at \eqref{eq:defZetZi} is the sum of the scores of points in sets $p_d({\mathcal V}_i,\delta)$ and $p_d({\mathcal V}_i,\delta)$ is a neighbourhood of vertex ${\mathcal V}_i$ of $K$ up to a transformation.
    \item Asymptotically, the contribution of the Poisson points far from the vertices to the expectation and variance of $Z$ is negligible compared to the contribution of the points near the vertices, where we recall that $Z$ is the sum of all the scores. In other words, the sum of scores in the complement of $\cup_{i=1}^{f_0(K)}p_d({\mathcal V}_i,\delta)$ in the sandwich ${\mathcal A}(s,T^*,K)$ is negligible in the asymptotic calculation of both the expectation and the variance of $Z$.
\end{itemize}
In order to do so, we use an explicit construction of Macbeath  regions appearing near the vertices in the economic cap covering theorem, see Theorem \ref{thm:economic cap}. This construction allows us to partition the annulus $\cA(s,T^*,K)$ in a way that makes it easier to control the spatial dependence of the scores, conditional on the sandwiching event. This dependence is described through a dependency graph. It turns out that the construction of this dependency graph is also a key element of the proof of the central limit theorem from our main result.   The strategy described above is heavily inspired by \cite[Section 3]{CY3}. The current section has been designed as well as we could so that the description of the method stays self-contained and, at the same time, is not too redundant with respect to \cite{BR10b} and \cite{CY3}.

The results presented below justify a posteriori the work of Sections \ref{sec:rescaling} and \ref{sec:stab polytopes}. Indeed it shows that it is enough to study the asymptotics of the sums of the scores near a vertex, which is precisely the goal of Sections \ref{sec:rescaling} and \ref{sec:stab polytopes}.

In the whole section, unless stated otherwise, we consider $\delta \in (0, 1/2)$.

\subsection{Dyadic M-regions and supersets}
This part is dedicated to the construction of the so-called supersets. It is done via an explicit description of the Macbeath regions of the economic cap covering near a corner. This construction comes from \cite[Section 3]{CY3}. It aims at making completely explicit a family of subsets which was already crucial in the strategy developed in \cite{BR10b}.

Recalling the definition of Macbeath regions given at \eqref{def:macbeath}, let us write, for all $z \in K$, $M_K(z) := M_K(z, 1/2)$.
Given $\delta \in (0, 1/2)$ and integers $k_i \in \Z$ with $3^{k_i} \in (0,1/(3\delta))$ for $1 \leq i \leq d$  (so that $3^{k_i+1}\delta/2 \leq 1/2$), the dyadic rectangular solids $\prod_{i=1}^d [\frac{3^{k_i}\delta}{2}, \frac{3^{k_i + 1}\delta}{2}]$
coincide with the M-regions 
$M_K((3^{k_1}\delta, \ldots, 3^{k_d}\delta))$. When $\log_3(\frac{d! T}{d^d \delta^d})\in \Z$, an M-region $M_K(z)$ has center $z$ belonging to $K(\textsl{v}=t)$
as soon as $\sum_{i = 1}^d k_i = \log_3(\frac{d! T}{d^d \delta^d}) $, thanks to Lemma \ref{lem:floating body eq} and a direct computation.

For any $\delta \in (0, 1/2)$ with $\log_3(\frac{d! T}{d^d \delta^d}) \in \Z$ , we denote by 
$\cM(0,\delta)$ the set of all M-regions that can be written as
$$M_K((3^{k_1}\delta, \ldots, 3^{k_d}\delta)) = \prod_{i=1}^d \left[\frac{3^{k_i}\delta}{2}, \frac{3^{k_i + 1}\delta}{2}\right]$$ with $k_i \in \Z$, $3^{k_i} \in (0,1/(3\delta))$ for $1 \leq i \leq d$ and $\sum_{i = 1}^d k_i = \log_3(\frac{d! T}{d^d \delta^d}) $. The set $\cM(0,\delta)$ is maximal in the sense that it can not be enlarged to include another M-region with center on $K(\textsl{v}=T)\cap [0, 1/2]^d$, see  \cite[Lemme 3.1]{CY3}. Furthermore, 
the economic cap covering 
that we recalled in Theorem \ref{thm:economic cap} can be constructed such that the set of all the Macbeath regions of the saturated system covering $K(\textsl{v}=T)$ with center in $[0,1/2]^d$ is exactly $\cM(0,\delta)$. Henceforth, without loss of generality, we will assume that the number $\delta$, whenever it appears, satisfies the condition $\log_3(\frac{d! T}{d^d \delta^d})\in \Z.$

We can now introduce the supersets induced by the net of the M-regions in ${\mathcal M}_K(0,\delta)$. The goal is to partition $\cA(s, T^*,K)$ with these supersets.
This construction is the same as in \cite[Section 3.3]{CY3}.
We start with the supersets in $[0, 1/2]^d$.
We describe the construction of the cone sets and cylinder sets associated to an M-region in $\cM(0,\delta)$ as they intervene in the construction of the supersets.
For any M-region $M_j$ that meets $[0, (T^*)^{1/d}]^d$, we define its associated cone set $\text{Co}_j$ as the intersection of $K(\textsl{v} \leq T^*)$ with the smallest cone with apex at $((T^*)^{1/d}, \ldots, (T^*)^{1/d})$ that contains $M_j$.
For any M-region $M_j $ with center at $(3^{k_1}\delta, \ldots, 3^{k_d} \delta)$ that meets $([0, (T^*)^{1/d}]^d)^c$, we define its associated cylinder set as follows.
Here and in the sequel, we denote by $H_l, 1 \leq l \leq d$, the coordinate hyperplanes.
 We first define for every $1 \leq l \leq d$, the cylinder $$C_l(k_1, \ldots, k_d) := \prod_{i=1}^{l-1} [\frac{3^{k_i}\delta}{2}, \frac{3^{k_i + 1}\delta}{2}] \times \R \times \prod_{i=l+1}^d [\frac{3^{k_i}\delta}{2}, \frac{3^{k_i + 1}\delta}{2}] \cap ([0, (T^*)^{1/d}]^d)^c.$$
 It is the smallest cylinder containing $M_j$ and oriented in direction $n_{H_l}$ where $n_{H_l}$ is a unit normal vector of the hyperplane $H_l$.
 We define the regions $$\tilde{S}_j := \tilde{S}_j(k_1, \ldots, k_d) := \bigcup_{l : k_l = \min(k_1, \ldots, k_d)} C_l(k_1, \ldots, k_d) \cap K .$$
 When $k_l$ is the unique minimum, $\tilde{S}_j$ is a single cylinder $C_l$ and it simply extends $M_j \cap ([0, (T^*)^{1/d}]^d)^c$ in direction $n_{H_l}$.
 We wish the cylinder sets to cover $K(\textsl{v}\leq T^*) \cap [0, 1/2]^d \setminus [0, (T^*)^{1/d}]^d$ but the union of all the regions $\tilde{S}_j$ does not cover all of it. The uncovered parts are rectangular regions produced by exactly one M-region having a cubical face.  For this reason, we define the cylinder set $\text{Cyl}_j$ of $M_j$ as the union of $\tilde{S}_j$ and the rectangular regions produced by the ties in the minimum of $k_1, \ldots, k_d$.
 
 Let $M_j$ be an M-region in $\cM(0, \delta)$. The superset $S'_j$ associated with $M_j$ is defined as follows.
 \begin{itemize}
     \item If $M_j$ is entirely included in $[0, (T^*)^{1/d}]^d$,  the cone set associated with $M_j$ is also included in $[0, (T^*)^{1/d}]^d$ and the superset associated with $M_j$ is defined as $S'_j = \text{Co}_j \cap \cA(s,T^*, K)$. We say that $S'_j$ is a cone set.
     \item If $M_j$ meets $[0, (T^*)^{1/d}]^d$, but is not entirely included in $[0, (T^*)^{1/d}]^d$, the superset associated with $M_j$ is $S'_j = (\text{Co}_j \cup \text{Cyl}_j) \cap \cA(s,T^*,K)$. In this case we say that $S'_j$ is a cone-cylinder set.
     \item If $M_j$ is included in $([0, (T^*)^{1/d}]^d)^c$, we simply put $S'_j = \text{Cyl}_j \cap  \cA(s,T^*,K)$ and call $S'_j$ a cylinder set.
     \end{itemize}
    \begin{figure}
        \centering
        \includegraphics[trim= 0cm 15cm 0cm 0cm, clip, width=0.44\linewidth]{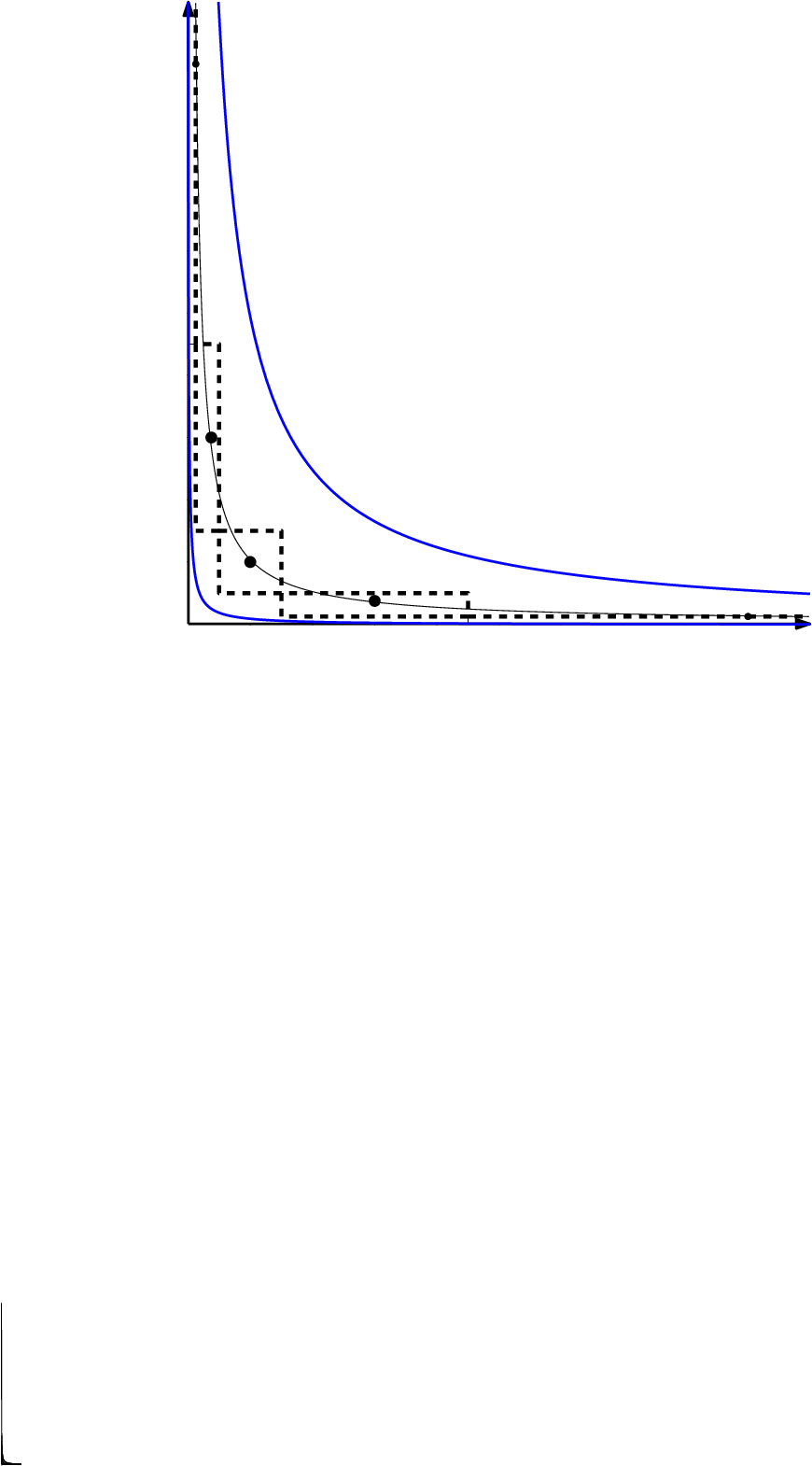}
        \includegraphics[trim= 0cm 15cm 0cm 0cm, clip, width=0.44\linewidth]{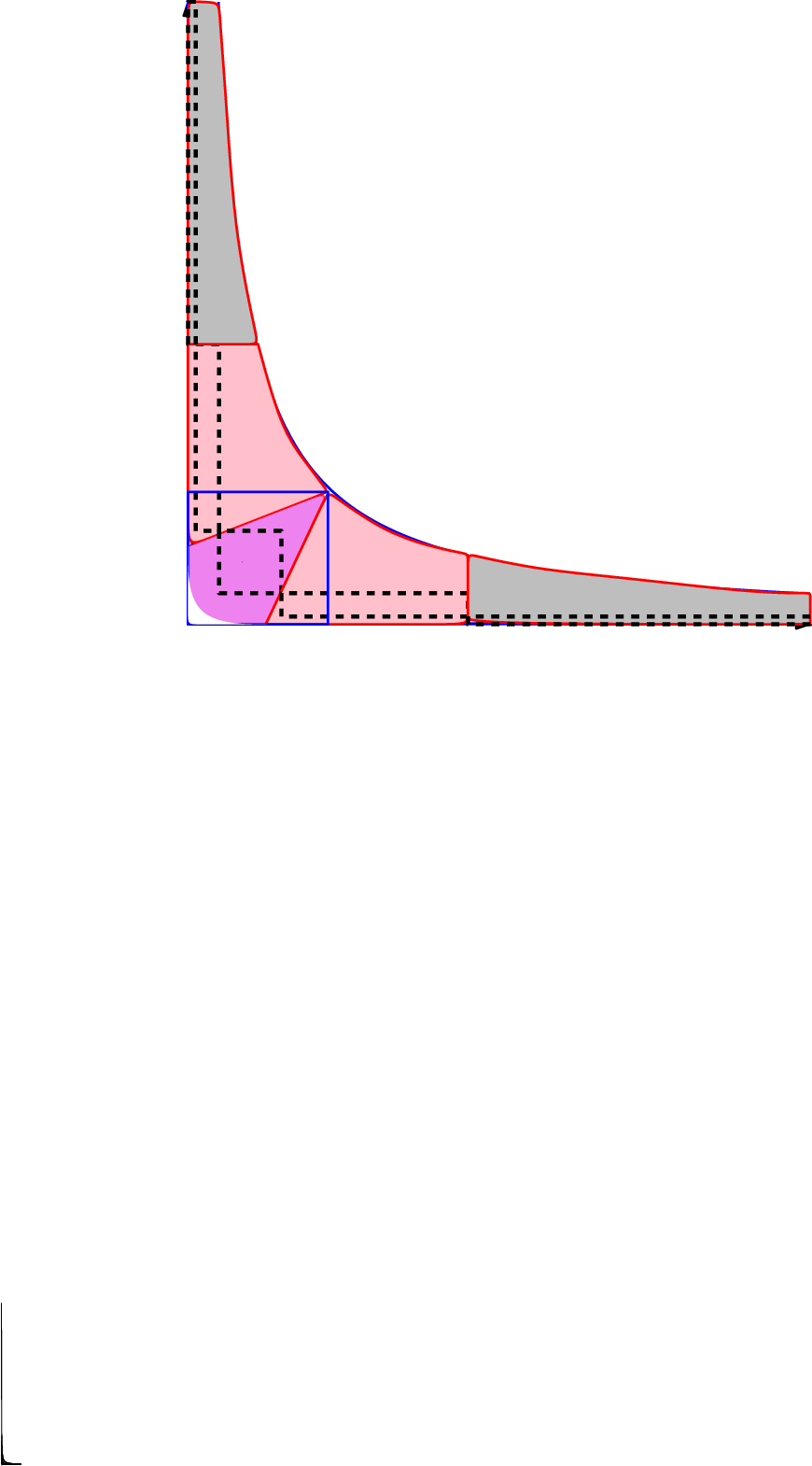}
        \caption{A saturated collection of Macbeath region and their associated supersets in pink, purple and grey, reproduced from \cite[Figures 3 and 4]{CY3}.}
        \label{fig:supersets}
    \end{figure}
    We reproduce here the Figure 3 and 4 from \cite{CY3}, to help visualize the shape of these supersets, see Figure \ref{fig:supersets}
    
     The supersets associated with an M-region in $\cM_K(0, \delta)$ are now well defined. For any vertex $\mathscr{V}_i \in \cV_K\setminus\{0\}$, we may likewise define a collection $\cM_K(\mathscr{V}_i, \delta)$ of dyadic M-regions. We can then generate an associated collection of supersets in a similar way. We embed the union of all the $\cM_K(\mathscr{V}_i, \delta)$ into a (not necessarily unique) larger collection of M-regions $\cM_K(m(T, \delta))$ with cardinality $m(T, \delta)$, that is maximal for the entirety of $K(\textsl{v} = T)$. This is always possible because among all the collections of M-regions containing the union of the $\cM_K(\mathscr{V}_i, \delta)$, there is always at least one that is maximal. To each additional M-region that is not in any of the $\cM_K(\mathscr{V}_i, \delta)$, we associate a superset $S'_j$ as in \cite{BR10b}, i.e. such that $M_j \subset S'_j$, the sets $S'_j$ are pairwise internally disjoint and such that their union cover the part of the sandwich that is not already covered by the supersets associated with the M-regions in one of the $\cM_K(\mathscr{V}_i, \delta)$.
     We also ask for an additional technical condition on these $S'_j$ which is that $S'_j \subset K_j^{(\gamma)}$ where $\gamma = 3d^3 6^d$. This is possible because $K(\textsl{v} \leq T^*)$ is covered by $\cup_i K_i^{(\gamma)}$, see \cite[Claim 2.6]{BR10b} where
     $T$ and $d6^d$ play the role of $s$ and $\la$ therein respectively.
     We summarize in Lemma \ref{lem:def supersets} all of the useful properties satisfied by every superset. They come directly from \cite[Section 5, 1516--1518]{BR10b}.
     \begin{lem}\label{lem:def supersets}
         The supersets $(S'_j)_j$ are pairwise interior disjoint, $\cup_j S'_j = \cA(s, T^*, K)$ and for every $1\leq j \leq m(T, \delta)$, we have
         $M_j \subset S'_j$ and $S'_j \subset K_j^{(\gamma)}$.
     \end{lem}
     Note that the M-regions $M_j$ also verify the properties of the economic cap covering result, i.e. Theorem \ref{thm:economic cap}.
     In particular, combining them with Lemma \ref{lem:def supersets}, we deduce the following volume estimates for the supersets, as in \cite[display (5.4)]{BR10b}. For all $1 \leq j \leq m(T, \delta)$, we have
     \begin{equation}\label{eq:vol regions superset}
        \frac{(6d)^{-d}\alpha \log \log(\la)}{\la} \leq \Vol_d(M_j) \leq \Vol_d(S'_j) \leq \Vol_d(K_j^{(\gamma)}) \leq \frac{(6\gamma)^d\alpha \log\log(\la)}{\la}.
     \end{equation}

     In the sequel, we denote by $\cS'(\delta) := \{S'_j\}_{j=1}^{m(T,\delta)}$ the set of all the supersets generated by M-regions in $\cM_K(m(T,\delta)).$
    Now that we have defined the supersets, we are able to define the event $A_\la$, which is a refinement of $\tilde{A}_\la$ defined at \eqref{eq:def tilde A la}, where we ask additionally for each superset to contain less than $c \log \log(\la)$ Poisson points, i.e.
 \begin{align}\label{eq:def A la}
     A_\la = A_\la(\delta) :&= \{\cup_{l=1}^n \partial \conv_l(\cP_\la) \subset  \cA(s,T^*,K),\ \forall j\ \text{card}(S'_j \cap \cP_\la) \leq 3(6 \gamma)^d \alpha \log \log(\la)\}\nonumber \\
     &=\tilde{A}_\la \cap \{\forall j\ \text{card}(S'_j \cap \cP_\la) \leq 3(6 \gamma)^d \alpha \log \log(\la)\}
     .
 \end{align}
The event $A_\la$ depends on $\delta$ through the explicit construction of the M-regions in each $\cM(\mathscr{V}_i, \delta)$. However, the next lemma shows that the probability of $A_\la$ can be estimated independently from $\delta$. As a result, we omit the dependency on $\delta$ in the sequel. A particular choice of $\delta$ will be made in the next subsection.

The following estimation of $\P(A_\la)$ follows directly from Theorem \ref{thm:sandwiching} and \cite[Claim 5.2.]{BR10b}. As a consequence, its proof is omitted.
\begin{lem}\label{thm:sandwiching et nb pts} There exists a constant $c > 0$ such that for all $\la \in [1, \infty)$
        $$\P(A_\la) \ge 1 - c \log^{-4d^2}(\la).$$
\end{lem}
    
Using the estimate of the number of supersets given in \cite[Theorem 2.7]{BR10b} and the definition of $A_\la$ yields the following lemma, that provides an upper bound of the number of points of $\cP_\la$ in the sandwich $ \cA(s,T^*,K) $ when on the event $A_\la$.
\begin{lem}\label{lem:nb points sandwich A la}
    There exists a constant $c > 0$ such that on the event $A_\la$,
    $$ \text{\normalfont card}(\cP_\la \cap \cA(s,T^*,K)) \leq c\log^{d-1}(\la) \log \log(\la) .$$
\end{lem}

\subsection{Dependency graph}\label{sec:depgraph}

In this subsection, we use the previous construction of the supersets to build a dependency graph as in \cite[Section 3.4]{CY3}. 

We define a graph $\cG :=  (\cV_{\cG}, \cE_{\cG})$, where $\cV_{\cG}$ is the set $\{1, \ldots, m(T, \delta)\}$. With a slight abuse of notation, we may sometimes identify the vertex $i$ of $\cG$ with the corresponding superset $S'_i$ and say for instance that there is an edge between $S'_i$ and $S'_j$ instead of $i$ and $j$.
We define the edges of the graph as follows. For any $1 \leq j \leq m(T, \delta)$, we define $L_j$ to be the union of all the supersets $S'_k \in \cS'(\delta)$ such that there exist $a \in S'_j $ and $b \in S'_k$ with the segment $[a, b]$ disjoint from $K(\textsl{v} \geq T^*)$.
In particular, $S'_j \subset L_j$ and for any $k\ne j$, $S'_j \subset L_k$ if and only if $S'_k \subset L_j$. We join the vertices $i$ and $j$ in $\{1, \ldots, m(T, \delta) \}$ by an edge if and only if $L_i$ and $L_j$ contain at least one superset $S'_k$ in common. 

To each $i\in \cV_\cG$, we associate the random variable $ \sum_{x\in \cP_\la \cap S'_i} \xi(x,\cP_\la).$ The subsection consists in showing that conditional on $A_\la$, the graph $\cG$ endowed with this particular collection of random variables is a dependency graph, which is a direct consequence of Lemma \ref{lem:no edge independent}. We then state useful properties of $\cG$ in view of Section \ref{sec:decompo}, namely a general bound on its degree in Lemma \ref{lem:max degree} and a geometric criterion in Lemma  \ref{lem:distance no edge} which guarantees that there is no edge between two specific vertices of $\cG$.

\begin{lem}\label{lem:no edge independent}
    Let $m \geq 2$ be an integer $\mathcal{W}_1, \ldots \mathcal{W}_m$ be disjoint subsets of $\mathcal{V}_{\mathcal{G}}$ having no edge between them and let $\xi \in \Xi$. Then conditional on $A_\la$, the random variables
    $$ \left(\sum_{x\in \cP_\la \cap(\cup_{j \in \mathcal{W}_l} S'_j)} \xi(x,\cP_\la)\right)_{l=1,\ldots,m}  $$
    are independent.
\end{lem}
\begin{proof}
    Let $j \in \{1, \ldots, m(T,\delta)\}$. At least one facet of each layer with number $\leq n$ must intersect $S'_j$. The vertices of these facets can only be points in $L_j$ because they are at one end of a segment crossing $S'_j$. We deduce from this observation that for any $p \le n$, the construction of $\conv_p(\cP_\la)$ in $S'_j$ only depends on the points in $L_j$.

    Let $L^{(l)} = \cup_{i \in \mathcal{W}_l} L_i$ for $l = 1, \ldots, m.$ The sets $L^{(l)}$  are unions of sets $S'_j$ and have disjoint interior because there is no edge between the $\mathcal{W}_l$. Since $\sum_{x\in \cP_\la \cap(\cup_{j \in \mathcal{W}_l} S'_j)} \xi(x,\cP_\la)$ only depends on $L^{(l)} \cap \cP_\la$  and the point sets $(L^{(l)}\cap \cP_\la)_{i=1, \ldots, m}$  are independent conditional of $A_\la$, the result follows.
\end{proof}
The next lemma is deterministic and provides
information on the maximal degree of the underlying dependency graph. Its proof is omitted as it is an exact rewriting of \cite[Theorem 6.2]{BR10b}.

\begin{lem}\label{lem:max degree}
    The maximal degree $D(\la)$ of the dependency graph $(\mathcal{V}_{\mathcal{G}}, \mathcal{E}_{\mathcal{G}})$ satisfies $$D(\la) = O((\log\log(\la))^{6(d-1)}).$$
\end{lem}

Lemma \ref{lem:distance no edge} below states that two regions far enough from each other have no edge of $\cG$ between them. This is a slight rephrasing of \cite[Lemma 3.5]{CY3}, whose proof builds on both (3.12) and Lemma 3.3 therein. We claim that such results still hold in our setting mainly because our values for $T$ and $T^*$ are the same as in \cite{CY3} up to a multiplicative constant. As a consequence, we omit the proof of Lemma \ref{lem:distance no edge}. In order to state this lemma, we introduce the notation
$ L(\la) := T(K)^3 (\log\log(\la))^{3(d-1)},$
that comes from \cite[p.17]{CY3}. We also introduce 
\begin{equation}\label{eq:def delta_1}
\delta_1:= r(\la, d) \delta_0
\end{equation}
where we recall the definition of $\delta_0$ at \eqref{eq:def delta_0} and $r(\la, d) \in [1, 3^{1/d})$ is chosen so that $\log_3(\frac{d! T}{d^d \delta_1^d}) \in \Z$. 
\begin{lem}\label{lem:distance no edge}
    There exists a constant $c' \in (0, \infty)$ such that if $S'_0 \in \cS'(\delta_1)$ is a subset of $[0, \frac{3}{2} \delta_1]^d$ and $S' \in \cS'(\delta_1)$ is at distance at least $c' \delta_1 3^{c^* L(\la)}$ from $[0, \frac{3}{2} \delta]^d$, then for $\la$ large enough, there is no edge between $S'_0$ and $S'$. 
\end{lem}


\subsection{Decomposition of expectations and variances}\label{sec:decompo}
The aim of this section is to show that both the expectation and variance of the considered geometric functionals of $\conv(\cP_\la)$ can be decomposed into the sum over the vertices of $K$ of the contributions of the scores near each vertex, up to asymptotically negligible terms. This is stated in Proposition \ref{prop:decompEVar} whose proof requires a series of intermediary steps. First, Lemma \ref{lem:pavage zero delta} shows that the supersets needed to construct the dependency graph are suitably partitioning the cube $[0,\frac32\delta]^d$ inside the sandwiching set $\cA(s,T^*,K)$. Then, using the dependency graph described in Section \ref{sec:depgraph}, we obtain in Lemma \ref{lem:conditional independence Zi} the asymptotic independence of the variables $Z_i(\frac32\delta)$ conditional on the high probability event $A_\la$ defined at \eqref{eq:def A la}. In Lemma \ref{lem:flat part negligible}, we prove that conditional on $A_\la$, the contribution of $Z_0(\frac32\delta)$ is negligible. This is a consequence of Lemma \ref{lem:vol flat}, which is a geometric result providing an upper bound for the volume of the \textit{flat} part $\mathcal{A}(s, T^*, K, \frac32 \delta)$ defined at \eqref{eq:defflatpart}. The conditioning is then removed thanks to Lemmas \ref{lem:approximation with A lambda} and \ref{lem:approximation conditional indicator}. We conclude with the proof of Proposition \ref{prop:decompEVar}.

Lemma \ref{lem:pavage zero delta} relates the supersets from the dependency graph to the cube at the vertex $0$ of $K$.
\begin{lem}\label{lem:pavage zero delta}
    For $\la$ large enough and any choice of $\delta = \delta(\la)$ such that $T= o (\delta)$ as $\la$ goes to infinity, $$\bigcup_{j : S'_j \subset [0, \frac{3}{2}\delta]^d}S'_j  = [0, \frac{3}{2}\delta]^d \cap \cA(s,T^*,K) .$$
\end{lem}
\begin{proof}
    For any $i$, we consider the M-regions $$M_{i,p} := \prod_{l=1}^{i-1}  [\frac{3^{k_{l,p}}}{2}\delta, \frac{3^{k_{l,p }+1}}{2}\delta] \times [\frac{1}{2} \delta, \frac{3}{2} \delta] \times \prod_{l = i+1}^d [\frac{3^{k_{l,p}}}{2}\delta, \frac{3^{k_{l,p }+1}}{2}\delta]   $$
    such that $M_{i,p} \in \mathcal{M}(0, \delta)$ and $k_{i,p} \leq 0$ for all $i$ and $p$.
    As $T = o(\delta)$ and $\sum_{i=1}^d k_i = \log_3(d!T/(d \delta)^d)$, for $\la$ large enough, these regions are included in $[0, \frac{3}{2}\delta]^d$. We claim that the cylinder sets $S'_j$ corresponding to these regions cover the part of $\partial[0, \frac{3}{2}\delta]^d$ that is contained in $\cA(s,T^*,K)$. Since the supersets $S'_j$ are pairwise interior disjoint and cover all of $\cA(s,T^*,K)$, we deduce that $[0, \frac{3}{2} \delta] \setminus \cup_{i,p} M_{i,p}$ is covered by supersets $S'_j$ that can only be contained in $[0, \frac{3}{2} \delta]^d.$ The result follows.
\end{proof}
An important consequence of Lemma \ref{lem:pavage zero delta} is the following decomposition of $Z_1(\frac32\delta)$, corresponding to the vertex $0$ of $K$:
\begin{equation}
    Z_1(\frac32\delta) = \sum_{j : S'_j \subset [0, \frac{3}{2}\delta]^d} \sum_{x \in S'_j} \xi_{n, k}(x, \cP_\la).
\end{equation}

We now fix $\delta=\delta_1$ introduced at \eqref{eq:def delta_1}. Our next result is the required asymptotic independence  of the variables $Z_i(\frac32\delta_1)$, up to conditioning on the high probability event $A_\la$ given at \eqref{eq:def A la}. This result is a direct consequence of Lemmas \ref{lem:distance no edge} and \ref{lem:no edge independent}.
\begin{lem}\label{lem:conditional independence Zi}
    For any $\xi \in \Xi$, conditional on $A_\la$, for $\la$ large enough the variables $Z_i(\frac32\delta_1)$, for $1 \leq i \leq f_0(K)$, are independent.
\end{lem}
\begin{proof}
    Thanks to Lemma \ref{lem:pavage zero delta}, the variable $Z_i(\frac32 \delta_1)$, $1\le i\le f_0(K)$, is the sum of the scores of all points in the supersets included in $p_d({\mathscr{V}}_i,\frac32\delta_1)$. 
    Let us consider two supersets $S'^{(i)}$ and $S'^{(j)}$, included in
    $p_d({\mathscr{V}}_i,\frac32\delta_1)$ and $p_d({\mathscr{V}}_j,\frac32\delta_1)$ respectively.
    As $d(S^{(i)},S^{(j)})  
    \geq d(\mathscr{V}_i, \mathscr{V}_j) - c \delta_1,$
    for $\la$ large enough,
    the Euclidean distance between $S^{(i)}$ and $S^{(j)}$  is larger than $c' \delta_1 3^{c^*L(\la)}$ since this quantity tends to 0. Thus, applying Lemma \ref{lem:distance no edge} after application of the respective affine transformations $a_i$ and $a_j$, we obtain that there is no edge between $S^{(i)}$ and $S^{(j)}$. We conclude the proof with an application of Lemma \ref{lem:no edge independent}. 
\end{proof}
Lemma \ref{lem:vol flat} below is a purely geometrical one and basically states that most of the volume of the sandwich $\cA(s,T^*,K)$ is located near the vertices. This will be an important ingredient for showing the negligibility of the flat parts of the sandwich in the calculation of the expectation and variance of $Z$ as it implies that most of the Poisson points and subsequently facets of the layers of the peeling should lie near the vertices of $K$.

The statement of Lemma \ref{lem:vol flat} is in fact a slight variation around \cite[display (7.13)]{CY3} which itself relies on \cite[display (4.1)]{BB93}. We omit the proof as the only change compared to these two references comes from the fact that our constants $T$ and $T^*$ differ by a multiplicative constant from their respective definitions in \cite{CY3} and \cite{BR10b}. This easily leaves the growth rate of the volume unchanged.

\begin{lem}\label{lem:vol flat}
    For $\delta \in (0, 1/2)$, let 
    \begin{equation}\label{eq:defflatpart}
    \mathcal{A}(s, T^*, K, \delta):=\mathcal{A}(s, T^*, K)\setminus \bigcup_{i=1}^{f_0(K)}p_d({\mathscr{V}}_i,\delta).
    \end{equation}
    When $\la\to\infty$, we get 
    $$ \text{\normalfont Vol}_d(\mathcal{A}(s, T^*, K, \frac32 \delta_1)) = O(\log\log(\la) \log^{d - 2 + 1/d}(\la)/\la).$$
\end{lem}

In the sequel we write for each $\delta \in (0, 1/2)$
$$Z_0(\delta) := \sum_{x \in  \cP_\la \cap \cA(s, T^*,K,\delta)} \xi_{n,k}(x, \cP_\la) .$$

Next, we derive the equivalent of \cite[Lemma 3.7]{CY3}.
Note that the statement is slightly different, i.e.  $o(\Var(Z))$ is replaced by $o(\log^{d-1}(\la))$, since in \cite{CY3} the growth rate of $\Var[Z]$ was already known and the challenge was to obtain a precise limit. Here, in the context of the subsequent layers of the convex hull peeling, we do not even know an order of magnitude for the variance of the number of $k$-faces (or defect volume). 
The proof does not differ much from \cite[Lemma 3.7]{CY3} but still, we write it in detail in order to justify that the issue mentioned above is in fact harmless. It also brings to light the importance of having the same $T$ and $T^*$ as in \cite{CY3} and \cite{BR10b} up to a constant factor.
\begin{lem}\label{lem:flat part negligible}
     For any $\xi \in \Xi$, we have $\E[Z_0(\frac32\delta_1) | A_\la] = o(\log^{d-1}(\la))$
        and $\Var[Z_0(\frac32\delta_1) | A_\la] = o(\log^{d-1}(\la))$.

\end{lem}
\begin{proof}
    We mostly focus on the case of the variance and briefly describe at the end how to extend the result to the expectation.
    For sake of simplicity, we write $Z_0$ for $Z_0(\frac32\delta_1)$ in this proof. 
    
     \noindent (1) First, we prove that $\Var[Z_0 \mathds{1}_{A_\la}] = o(\log^{d-1}(\la))$.  Let us write $\Tilde{\xi}(x, \cP_\la) := \xi(x, \cP_\la) \mathds{1}_{A_\la}$ and $\cP^x_\la := \cP_\la \cup \{ x \}$ for any $x \in K$. Using Mecke's formula, we obtain the  following decomposition
     $$ \Var[Z_0 \mathds{1}_{A_\la}] = \Var[ \sum_{x \in \cP_\la(s, T^*, K, \frac32\delta_1)} \tilde{\xi}(x, \cP_\la) ] = V_1 + V_2 $$
     where
     $$ V_1 := \la\int_{ \cA_\la(s, T^*, K, \frac32\delta_1)} \tilde{\xi}(x, \cP_\la)^2 \dd x $$
     and 
     $$ V_2 := \la^2\int_{\cA(s, T, K, \frac32\delta_1)^2} (\E[\tilde{\xi}(x, \cP_\la^y) \tilde{\xi}(y, \cP_\la^x)] - \E[\tilde{\xi}(x, \cP_\la)] \E[\tilde{\xi}(x, \cP_\la)]) \dd x \dd y .$$
     
     \noindent (a) In order to bound $V_1$ and subsequently $V_2$, we prove the following general bound for $\tilde{\xi}$.
     \begin{equation}
         \label{eq:big O flat}
         \sup_{x,y \in \cA(s, T^*, K, \frac32\delta_1)} |\tilde{\xi}(x, \cP_\la^y)| = O((\log \log(\la))^{(6(d-1)+1)d/2+1}). 
     \end{equation}
     Each $x \in \cA(s, T^*, K, \frac32\delta_1)$ is in some $S'_i \in (S'_j)_{j=1}^{m(T, \frac32\delta_1)}.$ Let $\cS_x$ be the union of all the $S'_j$ connected to $S'_i$ in the dependency graph.
     As the cardinal of $\cS_x$ is bounded by the maximal degree in the graph, Lemma \ref{lem:max degree} implies 
     $$ \text{card}(\{j: S'_j \subset \cS_x \}) = O((\log \log (\la))^{6(d-1)}).$$
     Since only the points in $\cS_x$ can be in a $k$-face containing $x$, see the definition of the edges in the graph, and since on $A_\la$, each $S'_j$ contains at most $c\log\log(\la)$ points of the process, only $O((\log \log (\la))^{6(d-1) + 1})$ points of $\cP_\la$ can contribute to a $k$-face containing $x$. McMullen's bound \cite{M70} states that the number of $k$-faces on the $n$-th layer of a set of $l$ points is at most $c l^{d/2}$ for $l$ large enough. Thus, when $\xi=\xi_{n,k}$, we obtain the following bound
     \begin{equation}\label{eq:big O flat k faces}
         \sup_{x,y \in \cA(s, T^*, K, \frac32\delta_1)} |\tilde{\xi}_{n,k}(x, \cP_\la^y)| = O((\log \log(\la))^{(6(d-1)+1)d/2}). 
     \end{equation}
     and a fortiori, $\tilde{\xi}_{n,k}$ satisfies \eqref{eq:big O flat}.
     
     When $\xi=\xi_{V,n}$, we follow the same method as in \cite[Proof of Lemma 3.7]{CY3}, i.e. we observe that for $x,y \in \cA(s, T^*, K, \frac32\delta_1)$,
     \begin{align}\label{eq:maj score vol}
     &\tilde{\xi}_{V,n}(x,\cP_{\la}^y)\le \la\;\mbox{card}(\mbox{facets of $\mbox{conv}_n(\cP_\la^y)$ containing $x$})\nonumber\\& \hspace*{2
     cm}\times \sup\{\mbox{Vol}_d(\mbox{Cap of $K$}): \mbox{Cap of $K$ tangent to $K(v=s')$ with $s\le s'\le T^*$}\}. 
         \end{align}
         Thanks to \eqref{eq:big O flat k faces} and \cite[Lemma 2.4]{BR10b}, this implies that $\tilde{\xi}_{V,n}$ also satisfies \eqref{eq:big O flat}.
         
     Combining  \eqref{eq:big O flat} and Lemma \ref{lem:vol flat}, we deduce that
     \begin{align*}
         V_1 &= \la \int_{\cA(s, T^*, K, \frac32\delta_1)} \E[\tilde{\xi}(x, \cP_\la)^2] \dd x =
         o(\log^{d-1}(\la)).
     \end{align*}
     
     \noindent (b) We now bound $V_2$. It goes into two parts, i.e. we first treat the sum over $x \in \cA(s,T^*,K, \frac32\delta_1)$, $y \not\in \cS_x$ and then the sum over $x \in \cA(s,T^*,K,\frac32\delta_1)$, $y \in \cS_x$ separately.
     \\~\\
     \noindent (b1) We start with the sum over $x \in \cA(s,T^*,K, \frac32\delta_1)$, $y \not\in \cS_x$. 
     In this case, as the scores of $x$ and $y$ only depend on points in $\cS_x$ and $\cS_y$ respectively, see the beginning of the proof of Lemma \ref{lem:no edge independent} for an explanation, and since $y \not\in \cS_x$, conditional on $A_\la$, the scores $\xi(x, \cP_\la^y)$ and $\xi(y, \cP_\la^x)$ are independent. Therefore
     $$ \E[\xi(x, \cP_\la^y)\xi(y, \cP_\la^x) | A_\la] - \E[\xi(x, \cP_\la) | A_\la] \E[\xi(x, \cP_\la) | A_\la] = 0 .$$
     Consequently, as $\E[X \mathds{1}_A] = \E[X  | A] \P(A) $ for a variable $X$ and event $A$ we get 
     \begin{align*}
         \E[\tilde{\xi}(x,\cP_\la^y)\tilde{\xi}(y,\cP_\la^x)] - \E[\tilde{\xi}(x,\cP_\la)]&\E[\tilde{\xi}(y,\cP_\la)] 
         \\&= \E[\xi(x, \cP_\la) | A_\la]\E[\xi(y, \cP_\la) | A_\la] \P(A_\la)\P(A_\la^c).
     \end{align*}
     Combining this last equality with \eqref{eq:big O flat}, slightly modified with $ \cP_\la$ instead of $\cP_\la^y$, Lemma \ref{thm:sandwiching et nb pts} and Lemma \ref{lem:vol flat}, we obtain
     \begin{equation}
         \la^2\int_{\cA(s, T, K, \frac32\delta_1)^2} (\E[\tilde{\xi}(x, \cP_\la^y) \tilde{\xi}(y, \cP_\la^x)] - \E[\tilde{\xi}(x, \cP_\la)] \E[\tilde{\xi}(x, \cP_\la)]) \mathds{1}_{y \not\in \cS_x} \dd x \dd y = o(\log^{d-1}(\la)).
     \end{equation}
     \noindent (b2) We now deal with the case $y \in \cS_x$. 
     The bound \eqref{eq:big O flat} implies that 
     \begin{equation}\label{eq:cov bound}
         \sup_{x,y} |\E[\tilde{\xi}(x,\cP_\la^y) \tilde{\xi}(y,\cP_\la^x)] - \E[\tilde{\xi}(x,\cP_\la)] \E[\tilde{\xi}(y,\cP_\la)]| = O((\log\log(\la))^{(6(d-1)+1)d + 2})
     \end{equation}
     Moreover, using \cite[equation (5.4)]{BR10b} yields
     \begin{align}
         \label{eq:bound vol Sx}
         \sup_{x \in \cA(s,T^*,K,\frac32\delta_1)} \text{\normalfont Vol}_d(\cS_x) &\leq \sup_{x \in \cA(s,T^*,K,\frac32\delta_1)} \text{card}(\{ j : S'_j \subset \cS_x \}) \sup_j \text{\normalfont Vol}_d(S'_j)\nonumber\\&= O\left(\frac{(\log\log(\la))^{(6(d-1)+1)d}}{\la}\right).
     \end{align}
     Combining \eqref{eq:cov bound}, \eqref{eq:bound vol Sx} and Lemma \ref{lem:vol flat}, we obtain
     \begin{align}
         \la^2\int_{\cA(s, T, K, \frac32\delta_1)^2} &(\E[\tilde{\xi}(x, \cP_\la^y) \tilde{\xi}(y, \cP_\la^x)] - \E[\tilde{\xi}(x, \cP_\la)] \E[\tilde{\xi}(x, \cP_\la)]) \mathds{1}_{y \in \cS_x} \dd x \dd y
         = o(\log^{d-1}(\la)).
     \end{align}
     To summarize, we have proved that $V_2 = o(\log^{d-1}(\la))$. We have shown the same bound in (a) for $V_1$ and we can thus conclude that $\Var[Z_0\mathds{1}_{A_\la}] = o(\log^{d-1}(\la))$. 
     \\~\\
     (2) Our last step is to show that $\Var[Z_0 | A_\la] = o(\log^{d-1}(\la))$. First, we
     notice that 
     \begin{equation}\label{eq:decomp var cond}
     \Var[Z_0 | A_\la] = \P(A_\la)^{-2}(\Var[Z_0\mathds{1}_{A_\la}] - \E[Z_0^2\mathds{1}_{A_\la}]\P(A_\la^c)).    
     \end{equation}
     As $\P(A_\la) \rightarrow 1$, see Lemma \ref{thm:sandwiching et nb pts}, and we already know that $\Var[Z_0\mathds{1}_{A_\la}] = o(\log^{d-1}(\la))$, we only need to prove that 
     $\E[Z_0^2\mathds{1}_{A_\la}]\P(A_\la^c) = o(\log^{d-1}(\la))$. In fact, this term is proved to go to zero by a method very similar to (1) and thanks to the estimates already obtained  in Lemma \ref{lem:vol flat}, \eqref{eq:big O flat}
     and Lemma \ref{thm:sandwiching et nb pts}.
     \\~\\
     To prove the result on the conditional expectation, one can adapt the proof of the upper bound for $V_1$ where we integrate $\tilde{\xi}(x,\cP_\la)$ instead of its square.
\end{proof}

For now, on the one hand, we have shown the additivity of $\Var[\sum_i Z_i(\frac32\delta_1) | A_\la]$ thanks to Lemma \ref{lem:conditional independence Zi} and the negligibility of $\E[Z_0(\frac32\delta_1) | A_\la]$ and $\Var[Z_0(\frac32\delta_1) | A_\la]$ in Lemma \ref{lem:flat part negligible}. On the other hand, we found the asymptotics of $\E[Z_i(\delta_0) \mathds{1}_{A_\la}]$ and  $\Var[Z_i(\delta_0) \mathds{1}_{A_\la}]$ in Proposition \ref{prop:asymptotics corner}. 

The goal in the sequel is to estimate the error made when the indicator function is replaced by the conditioning given $A_\la$ and the error made when we remove the conditioning. 


For the next step, we need an additional technical result. We put $U = \frac{\log(\la)}{\la}$ and $U^* = d6^d U.$. Moreover, we introduce the event
\begin{equation}\label{eq:defBla}
B_\la=\{K(\textsl{v} \geq U^*) \subset \Phi_n(\cP_\la)\}\cap\{\mbox{card}(\cP_{\la}\cap K(\textsl{v}\leq U^*))\le c\log^d(\la)\}.     
\end{equation}
The same method as in the previous sections yields the following estimate, which is an adaptation of \cite[Lemma 5.3]{BR10b}.
\begin{lem}\label{lem:estim B la}
    $$\P(B_\la^c) = O(\la^{-3d}).$$
\end{lem}

The following lemma proves that the conditional expectation and the conditional variance of $Z$ are close to the expectation and the variance of $Z$ respectively. The proof is essentially an adaptation of \cite[Section 8]{BR10b} to the layers of the peeling with a few extra details which were missing there.
\begin{lem}\label{lem:approximation with A lambda}
    For any $\xi \in \Xi$ and $\la$ large enough we have
    $$ \max\{  |\E[Z] - \E[Z | A_\la]|, |\Var[Z] - \Var[Z | A_\la]| \} = O(\log^{-2d^2}(\la)).$$
\end{lem}
\begin{proof}
We first recall that, by \cite[Claim 8.3]{BR10b} and Lemma \ref{thm:sandwiching et nb pts} 
    \begin{equation}\label{eq:diff esp et esp cond}
     |\E[\zeta] - \E[\zeta | A]| \leq (\E[\zeta | A] + \E[\zeta | A^c])\P(A^c)\le c\log^{-4d^2}(\la)(\E[\zeta | A] + \E[\zeta | A^c]). 
     \end{equation}
     Applying \eqref{eq:diff esp et esp cond} to $\zeta=Z$, we observe that it is enough to find a proper upper bound for $\E[Z|A_\la]$ and $\E[Z|A_\la^c]$.

     Similarly,     \begin{align} \label{eq:ineq diff var var cond}
    |\Var[Z] - \Var[Z | A_\la]| 
    &\leq
    |\E[Z^2] - \E[Z^2 | A_\la]| + |\E[Z]^2 - \E[Z | A_\la]^2|
    \nonumber
    \\&\leq |\E[Z^2] - \E[Z^2 | A_\la]|
    +  |\E[Z] - \E[Z | A_\la]|(|\E[Z] - \E[Z | A_\la]| + 2\E[Z|A_\la]).
    \end{align}
    Consequently, applying \eqref{eq:diff esp et esp cond} to $\zeta=Z$ and to $\zeta=Z^2$, we also need to bound $\E[Z^2|A_\la]$ and $\E[Z^2|A_\la^c]$.
    Thus, we fix $p=1,2$ and estimate $\E[Z^p | A_\la]$ and $\E[Z^p | A_\la^c]$.
    \\~\\
    \noindent (1) \textit{Estimation of $\E[Z^p | A_\la]$.} 

    We start by noticing that
    \begin{equation}\label{eq: calcul generique espcond}
\E[Z^p | A_\la] = \frac{1}{\P(A_\la)}\E[Z^p\mathds{1}_{A_\la}].
    \end{equation}
    On the event $A_\la$, the sandwich $\cA(s,T^*,K)$ contains at most $c\log^{d-1}(\la)\log\log(\la)$ points of $\cP_\la$,  see Lemma \ref{lem:nb points sandwich A la}, and any cap included in $\cA(s,T^*,K)$ has a volume of at most $c\log\log(\la)/\la$ by \cite[Lemma 2.4]{BR10b}. Consequently, combining \eqref{eq: calcul generique espcond} with Lemma \ref{thm:sandwiching et nb pts} and either MacMullen's bound \cite{M70} in the case $\xi=\xi_{n,k}$ or \eqref{eq:maj score vol} in the case $\xi=\xi_{V,n}$, we deduce that for $\la$ large enough
    \begin{equation}\label{eq:estimation Z sachant A}
    \E[Z^p | A_\la] \leq c\log^{d(d-1)}(\la)(\log\log(\la))^{d+2}.
    \end{equation}
    \\
    \noindent (2) \textit{Estimation of $\E[Z^p | A_\la^c]$}. 
    
    We use the event $B_\la$ introduced at \eqref{eq:defBla}. We get
    \begin{equation}\label{eq:decomposition expect Z cond Ac}
    \E[Z^p | A_\la^c] = \E[Z^p\mathds{1}_{B_\la} | A_\la^c] + \E[Z^p\mathds{1}_{B_\la^c} | A_\la^c].\end{equation}
    We consider the first term on the right-hand side of \eqref{eq:decomposition expect Z cond Ac}. On the event $B_\la$, there is at most $c \log^d(\la)$ points in $K(\textsl{v}\leq U^*)$. Consequently, we deduce from MacMullen's bound \cite{M70} that when $\xi=\xi_{n,k}$, 
    \begin{equation}\label{eq:finalement jen ai besoin}
    Z\le c\log^{\frac{d^2}{2}}(\la).    
    \end{equation}
    When $\xi=\xi_{V,n}$, we notice that on the event $B_\la$
    \begin{align}\label{eq:finalement jen ai besoin2}
    Z\le \la \mbox{Vol}_d(K\setminus \mbox{conv}_n(K_\la))
    \le c\la \mbox{Vol}_d(K(v\le U^*))
    \le c\log^d(\la)
    \end{align}
    where the last inequality is due to \cite[Theorem 2.7]{BR10b}. Consequently, in both cases, we obtain
    \begin{equation}\label{eq:estimation Z ind B sachant Ac}
    \E[Z^p \mathds{1}_{B_\la} | A_\la^c] \leq c\log^{d^2}(\la).
    \end{equation}
    The estimate of $\E[Z^p\mathds{1}_{B_\la^c} | A_\la^c]$ is more difficult as on the event $B_{\la}^c$, the variable $Z$ is not necessarily bounded. 
    
    We start by treating the case when $\xi=\xi_{n,k}$.
    Using the conditional total probability formula, writing $E_m$ for the event $\text{card}(\cP_\la) = m$ and MacMullen's bound we get
    \begin{align*}
        \E[Z^p\mathds{1}_{B_\la^c} | A_\la^c] &= \sum_{m=0}^\infty \E[Z^p\mathds{1}_{B_\la^c}|A_\la^c \cap E_m] \P(E_m | A_\la^c)\\
        &\leq \sum_{m=0}^{3\la} (3\la)^{\frac{pd}{2}} \E[\mathds{1}_{B_\la^c}|A_\la^c \cap E_m] \P(E_m | A_\la^c)
        + \sum_{m \geq 3\la} m^{\frac{pd}{2}}\E[\mathds{1}_{B_\la^c}|A_\la^c \cap E_m] \P(E_m | A_\la^c)\\
        &\leq c\left( \sum_{m=0}^{3\la}(3\la)^{\frac{pd}{2}} \P(B_\la^c|A_\la^c \cap E_m) \P(E_m | A_\la^c)
        + \sum_{m \geq 3\la}m^{\frac{pd}{2}}\P(E_m | A_\la^c)\right).
    \end{align*}
    At this point, a lower bound for $\P(A_\la^c)$ is required. Recalling the relation between $A_\la$ and $\tilde{A}_\la$, see \eqref{eq:def tilde A la} and \eqref{eq:def A la}, we claim that in fact, as in \cite{BR10b}, $\tilde{A}_\la$ may be replaced everywhere by the event $A_\la'\cap \{\cP_\la\cap K(v\le s)=\emptyset\}$ where $A_\la'$ has been defined at \eqref{eq:inclusion A'la}. This modification is harmless, as witnessed in the proof of Theorem \ref{thm:sandwiching}. Still, we have made the choice of keeping the event $\tilde{A}_\la$ as it is throughout the paper as it makes reading easier and because the current proof is the only place where we actually need to bound from below $\P(A_\la^c)$. 
    
    In the same way as in \cite[Claim 5.2]{BR10b}, we get 
    $$\P((A_\la')^c)\ge \P(\exists 1\le i\le m(T):\cP_\la\cap K_i'=\emptyset)\ge c\log^{-(3d)^{d+2}}(\la).$$  
    Consequently, we get, thanks to Lemma \ref{lem:estim B la}, that
    \begin{equation}\label{eq:estim B sachant Ac}
\P(B_\la^c | A_\la^c \cap E_m)\P(E_m | A_\la^c) \le        \frac{\P(B_\la^c)}{\P(A_\la^c)}\le c \la^{-3d +1}\mbox{ and } \P(E_m|A_\la^c)\le c\log^{(3d)^{d+2}}(\la)\P(E_m).
    \end{equation}
    Thus, we have
    $$ \E[Z^p\mathds{1}_{B_\la^c} | A_\la^c] \leq c \left(\la^{-2d+1} + \log^{(3d)^{d+2}}(\la)\sum_{m\geq 3_\la} m^{\frac{pd}{2}} \P(E_m)\right). $$
After an estimation of $\sum_{m \geq 3 \la} m^{\frac{pd}{2}} \P(E_m)$ that we leave to the reader, we deduce that for $\la$ large enough
    \begin{equation}\label{eq:estimation Z ind Bc sachant Ac}
        \E[Z^p\mathds{1}_{B_\la^c} | A_\la^c] \leq c \la^{-2d+1}. 
    \end{equation}
    Combining \eqref{eq:decomposition expect Z cond Ac}, \eqref{eq:estimation Z ind B sachant Ac} and \eqref{eq:estimation Z ind Bc sachant Ac}, we obtain
    \begin{equation}\label{eq:estimation Z sachant Ac}
        \E[Z^p | A_\la^c] \leq c\log^{d^2}(\la).
    \end{equation}
    When $\xi=\xi_{V,n}$, it is enough to bound $Z$ by $c\la$, which, thanks to \eqref{eq:estim B sachant Ac}, implies that
    $$\E[Z^p{\bf 1}_{B_\la^c}|A_\la^c]\le c\la^2\frac{\P(B_\la^c)}{\P(A_\la^c)}\le c \la^{-3(d-1)}.$$
    This obviously leads to \eqref{eq:estimation Z sachant Ac} as well.
    
 \noindent (3) \textit{Conclusion}. 
    
    We apply \eqref{eq:diff esp et esp cond} and \eqref{eq:ineq diff var var cond} to $Z$ and insert in them both \eqref{eq:estimation Z sachant A} and \eqref{eq:estimation Z sachant Ac} for $p=1$ and $p=2$. This completes the proof of Lemma \ref{lem:approximation with A lambda}.
\end{proof}
Next, we estimate the difference between the expectation (resp. variance) of the variable $Z_i(\delta_0)$ conditional on $A_\la$ and the expectation (resp. variance) of $Z_i(\delta_0) \mathds{1}_{A_\la}$.
\begin{lem}\label{lem:approximation conditional indicator}
    For every $1 \leq i \leq f_0(K)$ and any $\xi \in \Xi$,  we have
        $$|\E[Z_i(\delta_0) | A_\la] - \E[Z_i(\delta_0) \mathds{1}_{A_\la}]| = o(\log^{d-1}(\la))$$
        and
    $$|\Var[Z_i(\delta_0) | A_\la] - \Var[Z_i(\delta_0) \mathds{1}_{A_\la}]| = o(\log^{d-1}(\la)).$$
\end{lem}
\begin{proof}
    We only prove the result on the variance as the method goes along similar lines for the expectation. Let $1 \leq i \leq f_0(K)$. Thanks to \eqref{eq:decomp var cond}, we get
      $$ \Var[Z_i(\delta_0) | A_\la] - \Var[Z_i(\delta_0) \mathds{1}_{A_\la}] = \left( \P(A_\la)^{-2} - 1 \right) \Var[Z_i(\delta_0) \mathds{1}_{A_\la}] - \P(A_\la)^{-2}\P(A_\la^c) \E[Z_i(\delta_0)^2 \mathds{1}_{A_\la}],$$
      On the event $A_\la$, we can use the same method leading to \eqref{eq:finalement jen ai besoin} and \eqref{eq:finalement jen ai besoin2} when $T^*$ plays the role of $U^*$. 
      Applying Lemma 
      \ref{thm:sandwiching et nb pts}, we then get $\P(A_\la)^{-2}\P(A_\la^c) \E[Z_i(\delta_0)^2 \mathds{1}_{A_\la}] = o(\log^{d-1}(\lambda))$ and the result follows.
      
\end{proof}

We are now finally able to prove the main result of this section, which yields the decomposition of $\E[Z]$ and $\Var[Z]$ into the sum of the variables $\E[Z_i(\delta_0) \mathds{1}_{A_\la}]$ and the sum of the variables $\Var[Z_i(\delta_0) \mathds{1}_{A_\la}]$ respectively, up to a negligible term. Proposition \ref{prop:decompEVar} below is the analogue of \cite[Proposition 3.2]{CY3} for the layers of the peeling.
Notice that there is a small technical issue that needs to be addressed in the proof. Indeed, Lemma \ref{lem:flat part negligible} dealing with the negligibility of the contribution of the flat part concerns the variable $Z_0(\frac32\delta_1)$ where $\delta_1$ has been introduced at \eqref{eq:def delta_1}. As the transformation in the corner is applied to $Q_0 = [0, \delta_0]^d$ with $\delta_0$ given at \eqref{eq:def delta_0}, we need to replace $Z_i(\frac32\delta_1)$ by $Z_i(\delta_0)$ for $1\le i\le f_0(K)$ and estimate the error.


\begin{prop}\label{prop:decompEVar}
    When $\la\to\infty$ and any $\xi \in \Xi$, we get
    \begin{equation}
    \E[Z] = \sum_{\mathscr{V}_i \in \mathcal{V}_K} \E[Z_i(\delta_0) \mathds{1}_{A_\la}] + o( \log^{d-1}(\la)).    
    \end{equation}
    and
    \begin{equation}\label{eq:decompVarOK}
    \Var[Z] = \sum_{\mathscr{V}_i \in \mathcal{V}_K} \Var[Z_i(\delta_0) \mathds{1}_{A_\la}] + o( \log^{d-1}(\la)).    
    \end{equation}
    \end{prop}
\begin{proof}
    We focus on the variance, the proof for the expectation is similar and actually easier.
    Our first step is to assert that the  variances of $Z_i(\frac32\delta_1)$ and $Z_i(\delta_0)$,  conditional on $A_\la$, are close to each other.
    More precisely, for every $1\le i\le f_0(K)$ \begin{equation}\label{eq:delta0 delta1}
    \Var[Z_i(\frac32\delta_1) |A_\la] = \Var[Z_i(\delta_0) |A_\la] + o(\log^{d-1}(\la)).
    \end{equation}
    
    The idea is to write the difference $Z_i(\frac32\delta_1) - Z_i(\delta_0)$ as a sum of scores in the flat part for a $\delta'_1$ that is slightly smaller than $\delta_0$ and allows the construction of dyadic Macbeath region in the same spirit as what we have done for $\frac32\delta_1$. As for $\frac32\delta_1$, the scores in this flat part are negligible, see the proof of \cite[Proposition 3.2]{CY3} for more details.
Now, we use Lemma \ref{lem:approximation with A lambda} to replace $\mbox{Var}[Z]$ with its variance conditional on $A_\la$. We get
\begin{align}\label{eq:dec intermediaire}
\Var[Z] & = \Var[Z | A_\la] + o(\log^{d-1}(\la))\notag\\ 
        & =  \Var[ Z_0(\frac32\delta_1) + \sum_{i = 1}^{f_0(K)} Z_i(\frac32\delta_1) | A_\la] + o(\log^{d-1}(\la))\notag\\    
        &=\Var[Z_0(\frac32\delta_1) | A_\la] +  \Var[\sum_{i = 1}^{f_0(K)} Z_i(\frac32\delta_1) | A_\la] + 2\text{Cov}(\sum_{i = 1}^{f_0(K)} Z_i(\delta_1), Z_0(\delta_1) | A_\la)+ o(\log^{d-1}(\la)).
\end{align}
     By Lemma \ref{lem:flat part negligible}, we have 
     \begin{equation}\label{eq:reenonce flat part}
     \Var[Z_0(\frac32\delta_1) | A_\la] = o(\log^{d-1}(\la)).    
     \end{equation}
     Additionally, we use the Cauchy-Schwarz inequality, Lemma \ref{lem:conditional independence Zi}, \eqref{eq:delta0 delta1}, Lemma \ref{lem:approximation conditional indicator} and Proposition \ref{prop:asymptotics corner} to obtain
    \begin{align}\label{eq:longue decomposition}
        \big|\text{Cov}(\sum_{i = 1}^{f_0(K)} Z_i(\frac32\delta_1), Z_0(\delta_1)|A_\la)\big| &\leq
        \left(\Var[\sum_{i = 1}^{f_0(K)} Z_i(\frac32\delta_1) | A_\la]\right)^{1/2} \left(\Var[Z_0(\frac32\delta_1) | A_\la]\right)^{1/2} \nonumber\\
        &= \left(\sum_{i = 1}^{f_0(K)}\Var[ Z_i(\delta_0) | A_\la] + o(\log^{d-1}(\la)) \right)^{1/2} \left(o(\log^{d-1}(\la))\right)^{1/2} \nonumber \\
        &= \left(\sum_{i = 1}^{f_0(K)}\Var[ Z_i(\delta_0) \mathds{1}_{A_\la}] + o(\log^{d-1}(\la)) \right)^{1/2} \left(o(\log^{d-1}(\la))\right)^{1/2} \nonumber\\
        &= o(\log^{d-1}(\la))
    \end{align}
    Inserting  \eqref{eq:reenonce flat part} and \eqref{eq:longue decomposition} into \eqref{eq:dec intermediaire} yields
    $$\Var[Z] 
    = \Var[\sum_{i=1}^{f_0(K)} Z_i(\frac32\delta_1) | A_\la] + o(\log^{d-1}(\la)). $$
    Finally, we apply successively the additivity of the variance for $Z_i(\delta_1)$ implied by Lemma \ref{lem:conditional independence Zi}, then \eqref{eq:delta0 delta1} and Lemma \ref{lem:approximation conditional indicator} to derive
    \begin{align*}
        \Var[Z] 
        = \sum_{i=1}^{f_0(K)}\Var[ Z_i(\frac32\delta_1) | A_\la] + o(\log^{d-1}(\la))
        &= \sum_{i=1}^{f_0(K)}\Var[ Z_i(\delta_0) | A_\la] + o(\log^{d-1}(\la))\\
        &= \sum_{i=1}^{f_0(K)}\Var[ Z_i(\delta_0) \mathds{1}_{A_\la}] + o(\log^{d-1}(\la))
    \end{align*}
    which is the desired result.
\end{proof}



\section{Proof of Theorem \ref{thm:theoreme principal amelio}}\label{sec:main results}
In this section, we concentrate on the proof of Theorem \ref{thm:theoreme principal amelio}. We recall that Theorem  \ref{thm:theoreme principal amelio} implies  Theorems \ref{thm:theoremeprincipalpolytope} and \ref{thm:theoremeprincipalvolumepolytope}, thanks to \eqref{eq:decompavantrescaling} for the number of $k$-faces and thanks to \eqref{eq:decomp vol exp}, \eqref{eq:decomp vol var} and \eqref{eq:decomp vol P} for the defect volume.
\subsection{Proof of the limiting expectation and variance}  
We only prove the result on the variance as the proof of the expectation is analogous.
    From Proposition
    \ref{prop:decompEVar}, we deduce that
    $$ \Var[Z] = \sum_{\mathscr{V}_i \in \mathcal{V}_K} \Var[Z_i(\delta_0) \mathds{1}_{A_\la}] + o( \log^{d-1}(\la)).$$
    Combining this with the rewriting of the limit
    $$\lim\limits_{\la \rightarrow \infty}\frac{\Var[Z_i(\delta_0)\mathds{1}_{A_\la}]}{\log^{d-1}(\la)} = I_1(\infty) + I_2(\infty)$$
    that is obtained in Proposition \ref{prop:asymptotics corner}, we get
    $$\lim\limits_{\la \rightarrow \infty}\Var[Z] = f_0(K) (I_1(\infty) + I_2(\infty)).$$
\subsection{Proof of the positivity of the limiting expectation}
In view of the limiting expectation given in Theorem \ref{thm:theoreme principal amelio}, we need to prove that the expected score that is the integrand in the expression of the limiting constant is positive, both for the $k$-faces and for the volume. We observe that it is the case as soon as for any $w_0=(0,h_0) \in \R^{d-1}\times\R$, 
\begin{equation}\label{eq:toprovepositiveexp}
\P(\ell^{(\infty)}((0,h_0), \cP) = n) \neq 0.
\end{equation}
Our task then consists in proving \eqref{eq:toprovepositiveexp}.
        The strategy follows closely \cite[Lemma 5.1]{CQ1} but the geometric arguments need to be adapted to the current geometry of the rescaled layers, as the cone-like grains play the role of the paraboloids from \cite{CQ1}.
        We start with a particular configuration of points that ensures that $(0,h_0)$ is on the $n$-th layer. Next, we introduce a small random perturbation of these points that keeps $(0, h_0)$ on the $n$-th layer and provides us with random configurations that occur with positive probability.

    Let $c = 1/2$ if $h_0 > 0$ and $c = 2$ if $h_0 < 0$.
    For $i = 0, \ldots, (n-1),$
    write $w_i := (0, h_i) = (0, c^i h_0)$. The choice of $c$ was made solely to ensure that $h_j < h_i$ if $j > i$ and consequently, $\Pi^{\downarrow}(w_j)\subset \Pi^{\downarrow}(w_i)$.
    Let us assume for the time being that 
    $$ \cP \cap \bigcup_{i=0}^{n-1} \Pi^{\downarrow}(w_i) = \{w_1, \ldots, w_{n-1} \}.$$
    In that case, we have in particular $\cP\cap\Pi^{\downarrow}(w_{n-1})  = \varnothing$ and thus $w_{n-1}$ is on the first layer of the peeling of $\cP$. Thanks to Lemma \ref{lem:incl cone}, any downward cone-like grain $\Pi^\downarrow(v,h)$ going through $w_{n-2}$ contains a downward cone with apex $w_{n-2}$ and thus contains $w_{n-1}$, see Figures \ref{fig:cone_esp} and \ref{fig:cone_esp2}.
    \begin{figure}
    \centering
    \begin{overpic}[width=0.95\textwidth,trim={0cm 7cm 0cm 4cm},clip]{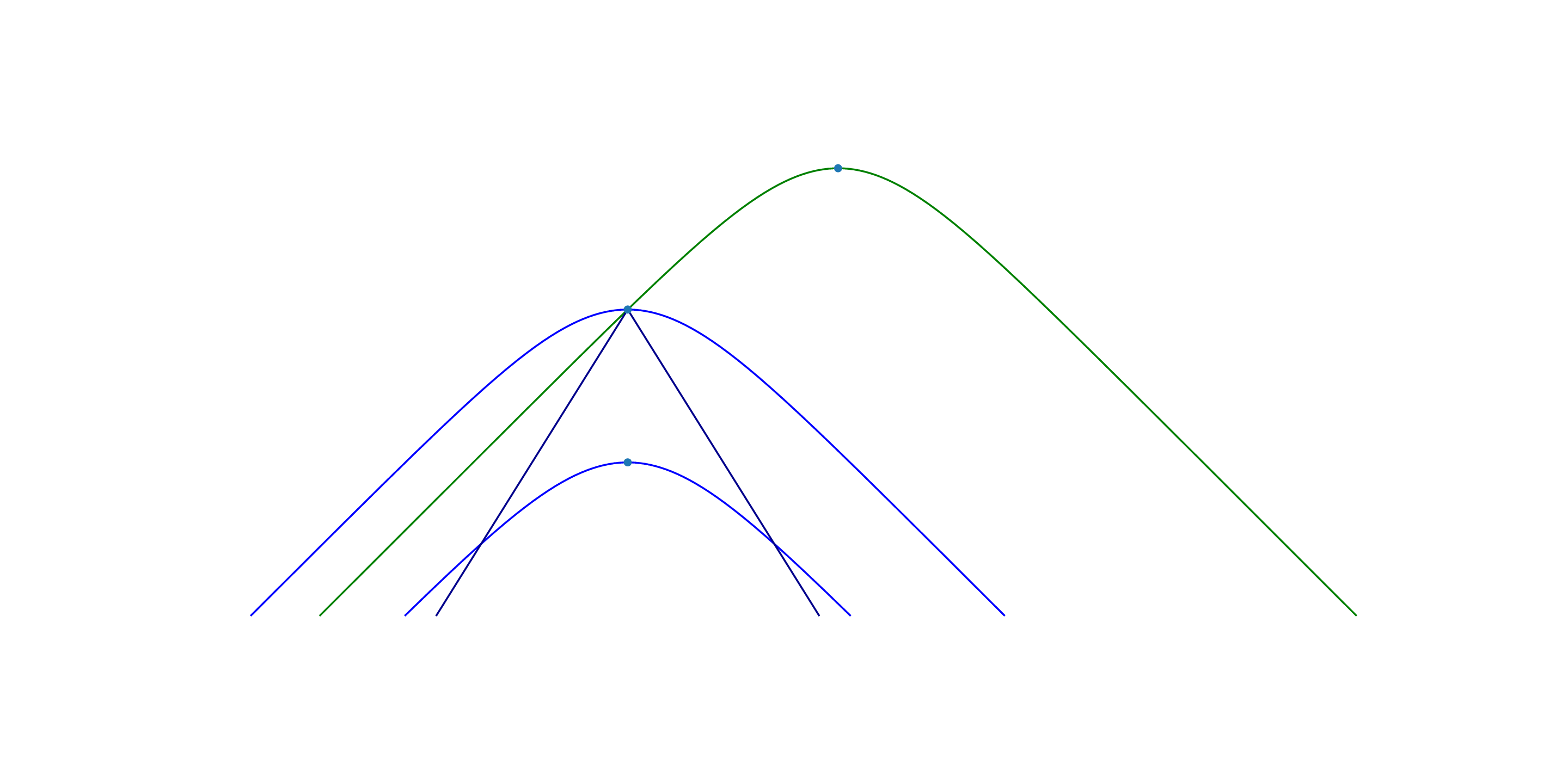}
    \put(36,16.5){$w_0$}
    \put(38,7){$w_1$}
    \put(53,25.5){$(v,h)$}
\end{overpic}
    \caption{Illustration in the case $n=2$ of the fact that any $\Pi^\downarrow(v,h)$ going through $w_0$ contains $w_1$.}
    \label{fig:cone_esp}
\end{figure}

\begin{figure}
    \centering
    \begin{overpic}[width=0.95\textwidth,trim={0cm 7cm 0cm 4cm},clip]{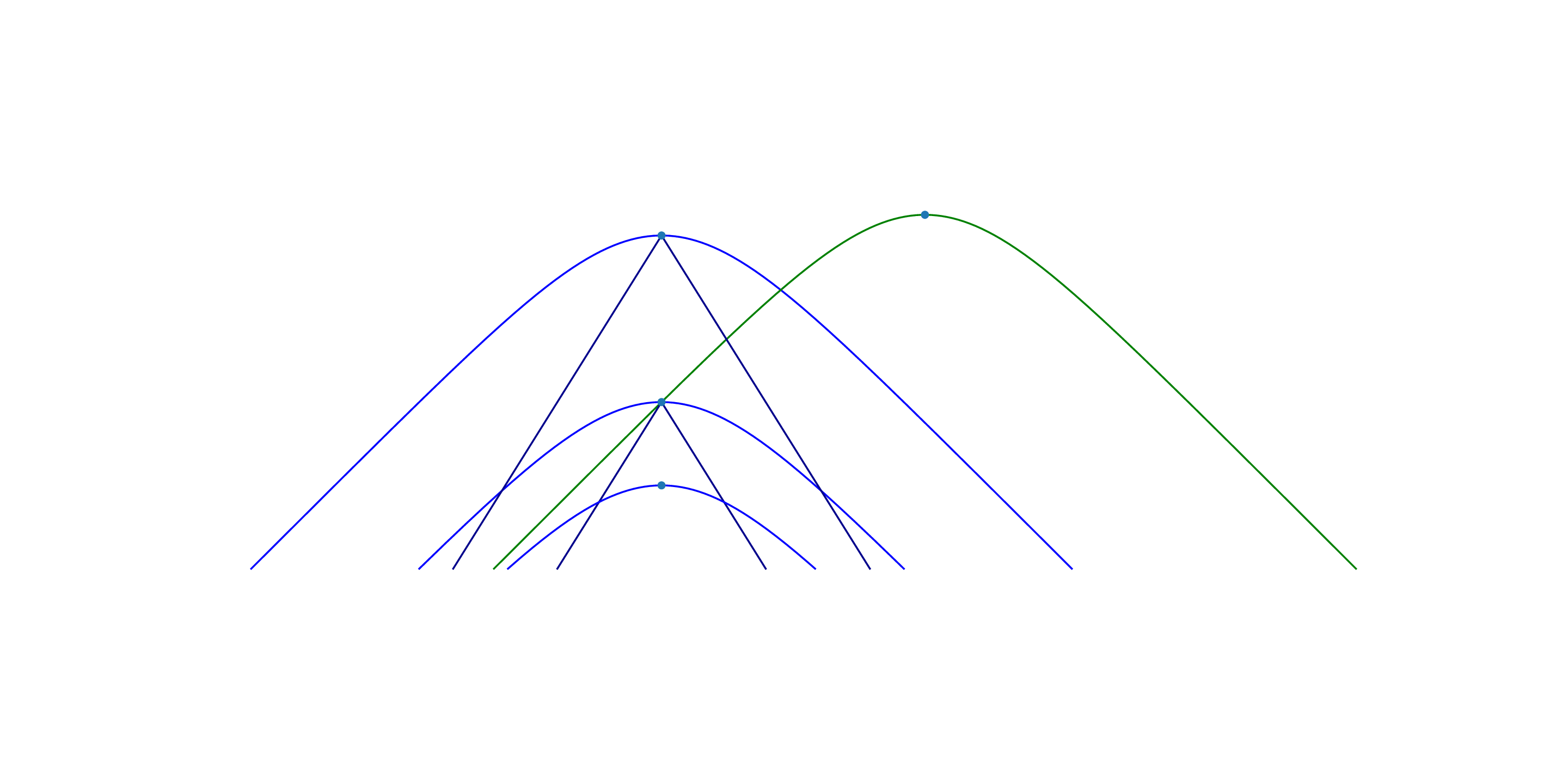}
    \put(38.5,21){$w_0$}
    \put(38.5, 10.5){$w_1$}
    \put(40.5, 2){$w_2$}
    \put(56,22.5){$(v,h)$}
\end{overpic}
    \caption{Configuration for $n = 3$ and illustration of the fact that any $\Pi^\downarrow(v,h)$ going through $w_1$ contains $w_2$.}
    \label{fig:cone_esp2}
\end{figure}
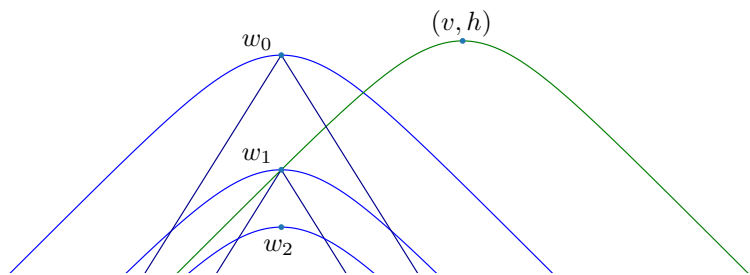

    In particular, $w_{n-2}$ is on layer at least $2$ because of Lemma \ref{lem:be on layer n}. As $\cP\cap\Pi^{\downarrow}(w_{n-2})  = \{w_{n-1}\}$, it is empty after the removal of the first layer so $w_{n-2}$ is on layer $2$. The same reasoning applied by induction shows that $w_{n-i}$ is on layer $i$ for any $i = 1, \ldots, n$. Thus, $w_0$ is on layer $n$ of the peeling of $\cP\cup\{w_0\}$.

    The event described above happens with probability $0$ as the positions of $w_1, \ldots, w_{n-1}$ are fixed. Let us perturb them slightly to obtain an event with positive probability. We choose $\varepsilon > 0$ small enough so that for any $i$ and any $w \in B(w_i, \varepsilon)$, $B(w_{i+1}, \varepsilon)$ is included in the cone contained in $\Pi^\downarrow(w)$ given by Lemma \ref{lem:incl cone}  and additionally $\Pi^\downarrow(w) \subset \Pi^\downarrow(w_0)$.
    We write $E$ for the event
    $$E := 
    \{\text{for all } 1\le i\le (n-1), \text{card}  (\cP \cap B(w_i, \varepsilon)) = 1  \} \cap \{ \cP \cap  \Pi^\downarrow(w_0) \setminus  (\cup_{i=1}^{n-1} B(w_i, \varepsilon))  = \varnothing \}.$$
    If we write $w'_i$ for the single point in $B(w_i, \varepsilon)$ on the event $E$, we show, as in the unperturbed case, that $w'_i$ is on layer $(n-i)$ for the peeling of $\cP$ and thus $w_0$ is on layer $n$ for the peeling of $\cP\cup\{w_0\}$, see Figure \ref{fig:cone_esp3} for an example of a perturbed configuration.
\begin{figure}
    \centering
    \begin{overpic}[width=0.95\textwidth,trim={0cm 3cm 0cm 3cm},clip]{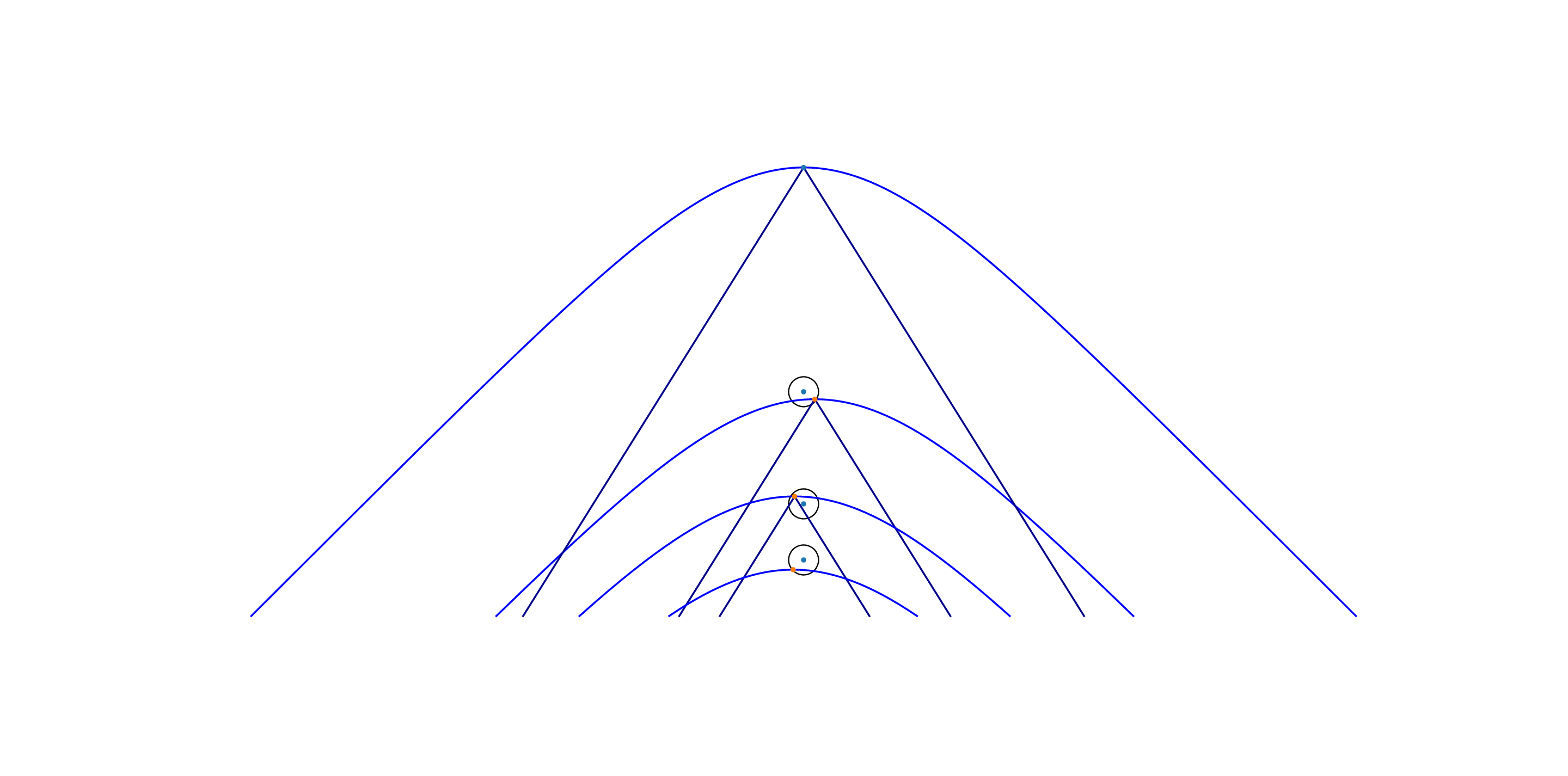}
    \put(50, 33.7){$w_0$}
    \put(45.5, 20.5){$B(w_1, \varepsilon)$}
    \put(69, 7.5){$\color{blue}\partial\Pi^\downarrow(w'_1)$}
\end{overpic}
    \caption{Example of perturbed configuration for $n=4$.}
    \label{fig:cone_esp3}
\end{figure}
    As the event $E$ occurs with positive probability, the result \eqref{eq:toprovepositiveexp} follows and the positivity of the limiting expectation in Theorem \ref{thm:theoreme principal amelio} as well.

\begin{rem}
    To prove the positivity of the limiting constant in the case of the defect volume, we can also use the same argument as in \cite{CQ1}. The defect volumes are increasing with the number of the layer. Thus, as we know that the limiting constant for the first layer is positive, it follows that this remains true for all of the subsequent defect volumes.
\end{rem}
\subsection{Proof of the positivity of the limiting variance}
Our method is again an adaptation of the proof of the limiting variance when $K$ is the unit ball \cite{CQ1} to our context. 
Let us recall this strategy here: contrary to the previous proof for the positivity of the expectation, we do not work on the explicit integral formula for the limiting variance. Instead, we show that the limit in Proposition \ref{prop:asymptotics corner} is positive, i.e. for any $1\le i\le f_0(K)$ 
\begin{equation}\label{eq:toprovepositivevar}
\lim\limits_{\la \rightarrow \infty}\frac{\Var[Z_i(\delta_0)\mathds{1}_{A_\la}]}{\log^{d-1}(\la)} > 0.    
\end{equation}
To do so, we use \eqref{eq:rewritingZ_iinrescaledworld} and concentrate on the variance of $\sum_{w\in \cP^{(\la)}\cap W_\la}\xi^{(\la)}(w,\cP^{(\la)})\mathds{1}_{A_\la}.$
We then discretize $W_\la$ given at \eqref{def:Wla} as a disjoint union of $O(\log^{d-1}(\la))$ parallelepipeds and construct in each parallelepiped two different configurations, said to be good, which have a positive probability to occur and which give birth to two
different values for the sum of scores inside the parallelepiped. This induces that the contribution of each parallelepiped has a positive variance. Next, we check that this contribution
is not affected by the external configuration of the point process and finally, we find a lower bound for the total variance
conditional on the intersection of the Poisson point process with the outside of the parallelepipeds. We proceed below with the case when $\xi=\xi_{n,k}$ and explain at the end of the section how to extend the method to the volume score $\xi_{V,n}$.
\\~\\
\noindent \textit{Step 1 : Construction of a good configuration in a thin parallepiped}. 
First, we consider a cube
$Q \subseteq\R^{d-1}$ and we take $\rho \in (0, \infty)$ smaller than the diameter of $Q$. We take $\delta > 0$ sufficiently small
such that the cone-like grains, restricted to the upper half-space $\Pi^\downarrow(w) \cap (\R^{d-1}\times\R_+)$ are pairwise disjoint for $w$ belonging to the grid $(\rho\Z^{d-1} \cap Q) \times \{\delta\}$.
For each $w \in (\rho\Z^{d-1} \cap Q)\times \lbrace\delta\rbrace$, we construct the exact same sequence with $c=1/2$ as in the proof of the limiting expectation inside of $\Pi^\downarrow(w)$. We write $\cT_{n,\rho}$ the set of all the points that intervene in this construction once applied to every $w$. Since any point of the grid $(\rho\Z^{d-1} \cap Q) \times \{\delta\}$ has its cone-like grain disjoint from the cone-like grains associated with the other points of the grid in the upper half-space, 
we get that every $w \in (\rho\Z^{d-1} \cap Q)\times \lbrace\delta\rbrace$ is on the $n$-th layer of the peeling of $\cT_{n,\rho}$.

Let us write
	$F_{n,k}(Q, \rho, \delta)$ for the number of k-faces of the $n$-th layer of $(\cP^{(\la)}\setminus (Q\times [0,\infty)))\cup \cT_{n,\rho}$ going through any 
	$w \in (\rho \mathbb{Z}^{d-1} \cap Q) \times \{ \delta \}$. If we ignore the points of $\mathcal{P}^{(\la)}\setminus (Q\times [0,\infty))$, as the diameter of $Q$
	gets large compared to $\rho$, boundary effects become negligible and we get
	$$ F_{n,k}(Q, \rho/2, \delta) \sim 2^{d-1} F_{n,k}(Q, \rho, \delta).$$
	Then we consider $\varepsilon > 0$ such that for any 
$w=(v,\delta+\varepsilon) \mbox{ and }w'=(v',\delta+\varepsilon)\mbox{ with } \|v-v'\|\ge \rho - 2 \varepsilon$, we have $\Pi^{\downarrow}(w)\cap \Pi^{\downarrow}(w') \cap \R^{d-1}\times\R_+ =\varnothing$. This is always possible, up to a possible slight reduction of $\delta$ and $\varepsilon$. Indeed, thanks to Lemma \ref{lem:sandwich cone}, it is enough to take $\varepsilon$ and $\delta$ such that
\begin{align}\label{eq:conddeltaeps}
\delta+(1+\underline{c})\varepsilon& \le \frac12\underline{c}\rho-\log(d)
\end{align}
where we recall that $\underline{c} >0$ is the constant appearing in the aforementioned lemma. For this condition to make sense, we need $\rho$ to be larger than $\frac{2}{\underline{c}}\log(d)$, which we will assume to be true in what follows.

	As in the proof of the positivity of the expectation, we consider a random perturbation of the points of $\cT_{n,\rho}$ by at most $\varepsilon<\rho/2$ for the Euclidean distance, i.e. we assume that $\cP^{(\la)}\cap B(w,\varepsilon)$ consists of one point for each $w\in \cT_{n,\rho}$. We denote by $\cT_{n,\rho}'$ the set of these perturbed points from the Poisson point process. We notice that the points from $\cT_{n,\rho}'$ have a height at most equal to $\delta+\varepsilon$ and are distant by at most $2\rho$ laterally. Moreover, the set $(Q \times [0, \infty) ) \setminus \bigcup_{w'\in \cT_{n,\rho}'} \Pi^\uparrow(w')$ contains the first $n$ layers of the convex hull peeling of $\cT_{n,\rho}'$. Let us denote by $\alpha$ the maximal height of that set. We observe that for $\varepsilon$ small enough, $\alpha$ is less than the maximal height of a point in $(Q \times [0, \infty) ) \setminus \bigcup_{w \in (Q \cap 2 \rho \mathbb{Z}^{d-1}) \times \{ \delta +\varepsilon\} } \Pi^\uparrow(w)$. In particular, $\alpha$ is bounded from above by the height of the intersection point of two cone-like grains with apices distant laterally by $2\rho$ and with height $(\delta+\varepsilon)$, which, thanks to Lemma \ref{lem:sandwich cone} and \eqref{eq:conddeltaeps}, satisfies
 \begin{align}\label{eq:calpha}
    \alpha\le (\delta+\varepsilon)+\overline{c}\rho\le (\overline{c}+\frac12\underline{c})\rho-\log(d). 
 \end{align}
 Figure \ref{fig:alpha} shows what $\alpha$ is when the points are unperturbed for sake of simplicity.
 \begin{figure}
     \centering
         \begin{overpic}[width=0.95\textwidth,trim={0cm 0.5cm 0cm 0.5cm},clip]{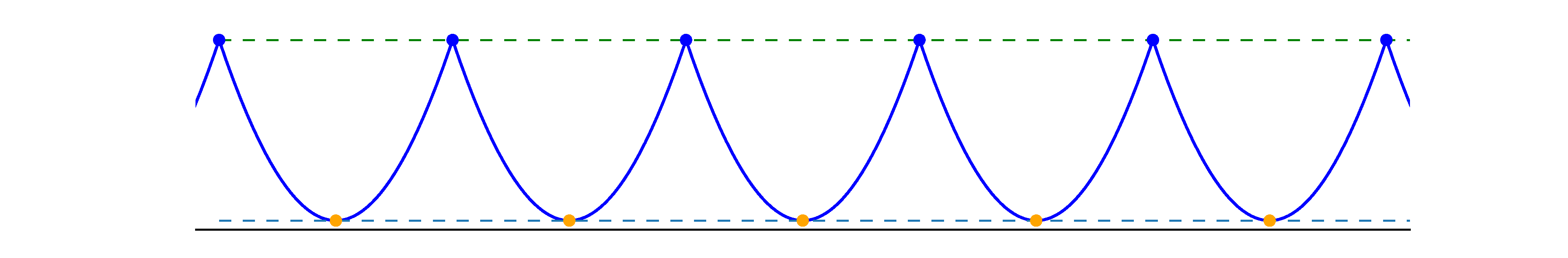}
    \put(90.5,-1){$0$}
    \put(90.5,1){$\delta$}
    \put(90.5,11.8){$\alpha$}
\end{overpic}
     \caption{Illustration of the maximal height $\alpha$ when the points are unperturbed and at distance $\rho$ of each other.}
     \label{fig:alpha}
 \end{figure}
Consequently, let us consider the event 
	\begin{align*}
	A_{n,\rho}=\cap_{w\in {\cT_{n,\rho}}}\{{\mbox{card}(\mathcal P^{(\la)}}\cap B(w,\varepsilon)) =1\}
\cap \lbrace \cP^{(\la)}\cap [\cup_{w\in Q\times [0,\alpha]} \Pi^{\downarrow}(w)\setminus \cup_{w\in \cT_{n,\rho}}B(w,\varepsilon)]=\varnothing\rbrace
 .\end{align*}
 This event has positive probability. To prove it, it is enough to show that the union of cone-like grains $\cup_{w\in Q\times [0,\alpha]} \Pi^{\downarrow}(w)$ has finite measure. This union of cone-like grains can be included in $\cup_{w \in Q \times [0, \alpha]}\cC_w$ where $\cC_w$ is a circular cone with apex at $w$ given by Lemma \ref{lem:sandwich cone}. The set $\cup_{w \in Q \times [0, \alpha]}\cC_w$ can in turn be covered entirely with a finite union of circular cones with apex on $Q \times [0, \alpha']$, for any choice of $\alpha' > \alpha$. This  finite union of circular cones has finite measure for the intensity measure of $\cP_\la$ and thus, $\cup_{w\in Q\times [0,\alpha]} \Pi^{\downarrow}(w)$ has finite measure as well.
Next, let us write $F_{n,k}(Q, \rho, \delta, \varepsilon)$
	for the total number of $k$-faces of the $n$-th layer of $\cP^{(\la)}$ going through at least one point in $\cP^{(\la)}\cap \cup_{w\in (\rho \mathbb{Z}^{d-1} \cap Q) \times \{ \delta \}}B(w,\varepsilon)$. Conditional on $A_{n,\rho}$ and when the points of $\cP^{(\la)}\setminus (Q\times [0,\alpha])$ are ignored, for $\varepsilon$ small enough, this quantity is in fact equal to $F_{n,k}(Q, \rho, \delta)$ and we keep the relation
	\begin{equation}\label{var pos scaling} F_{n,k}(Q, \rho/2, \delta, \varepsilon) \sim 2^{d-1} F_{n,k}(Q, \rho, \delta, \varepsilon).\end{equation}
\\~\\
\noindent\textit{Step 2. Influence of the points outside of the thin parallelepiped $Q\times [0,\alpha]$.}
		For any closed set 
	$C \subseteq \mathbb{R}^{d-1}$, we define for any $\gamma > 0$, $C^{(\gamma)} := \{ x \in C : d(x, \partial C) > \gamma \}$. We claim that on $A_{n,\rho}$, for any $w \in \cP^{(\la)}\cap (Q^{(\rho)}\times [0,\alpha])$, the status of $w$ does not depend on points outside $Q\times[0,\alpha]$, i.e.
	$$\ell^{(\la)}(w,\cP^{(\la)}\cap (Q\times[0,\alpha]))=\ell^{(\la)}(w,\cP^{(\la)}).$$
	This is due to the fact that the condition \eqref{eq:conddeltaeps}
 guarantees that the intersection with the upper half-space of the downward cone-like grains with apices at points from $\cP^{(\la)}\cap (Q^{(\rho)}\times [0,\alpha])$ are included in $Q\times[0,\alpha]$. 
	Moreover, for $\textbf{c}=(\frac{\overline{c}}{\underline{c}}+\frac12)\rho + 2 \rho
 $, we assert that the facial structure around any point inside $\cP^{(\la)}\cap (Q^{(\textbf{c})}\times [0,\alpha])$ which belongs to the $n$-th layer of the peeling of $\cP^{(\la)}$ does not depend on points outside $Q\times[0,\alpha]$. Indeed, let us consider $w\in \cP^{(\la)}\cap (Q^{(\textbf{c})}\times [0,\alpha])$. We choose $(d-1)$ points from $\cP^{(\la)}\cap(Q^{(\textbf{c} - 2\rho)}\times [0,\alpha])$ which share with $w$ a common facet of the $n$-th layer of the peeling of $\cP^{(\la)}\cap (Q\times[0,\alpha])$. The cone-like grain which contains this facet has an apex in $Q^{(\textbf{c} - 2\rho)}\times[0,\alpha]$. Consequently, on the event $A_{n,\rho}$, thanks to \eqref{eq:calpha} and the lower bound in Lemma \ref{lem:sandwich cone}, the intersection of that cone-like grain with the upper half-space is included in $Q\times[0,\alpha]$, which implies that the facet containing $w$ and the $(d-1)$ other points is a facet of the $n$-th layer of the peeling of $\cP^{(\la)}$. 
\\~\\
\textit{Step 3. Discretization of $W_\la$ and lower bound for the variance.} Recall the definition of $W_\la$ and $\delta_0$ at \eqref{def:Wla} and \eqref{eq:def delta_0} respectively.
Identifying $V$ and $\R^{d-1}$, we notice that $W_\la$ has a pyramidal shape with height $\frac{\log(\la)}{d} - \log^{1/d}(\la)$. In particular, for any $\delta > 0$, we can find a constant $C_{\text{disc}} >0$ such that $C_{\text{disc}}\left[-\log(\la) , \log(\la)\right]^{d-1} \times [0, \delta] \subset W_\la$ for $\la$ large enough. 

We are now ready to discretize the set $C_{\text{disc}}\left[-\log(\la) , \log(\la)\right]^{d-1}$ and isolate the \textit{good} parallelepipeds from the discretization, according to the two previous steps. 
	We choose $\delta$  and $\varepsilon$ which satisfy \eqref{eq:conddeltaeps}
 with the choice $\rho= \frac{8}{\underline{c}} \log(d)$.  
	We take a large positive number $M$ and we partition $C_{\text{discr}}\left[ -\log(\la), \log(\la)\right]^{d-1}$ into $L:= \left[\frac{C_{\text{discr}}\log(\la)}{M}\right]^{d-1}$ cubes 	$Q_1, \ldots, Q_L$. We consider the cubes $Q_i$ satisfying the following properties:
	
	(a) For each $z \in (\rho\mathbb{Z}^{d-1} \cap (Q_i \setminus Q_i^{(\textbf{c})})) \times  \{ \delta \}$,  $\cP^{(\la)} \cap B(z, \varepsilon) $ is a singleton and is put on the $n$-th layer using the  tree construction associated with $\cT_{n, \rho}$ and the perturbation of each point by at most $\varepsilon$ as in Step $1$.
	
	(b) One of these two conditions holds:
	\begin{enumerate}
		\item For each $z \in (\rho\mathbb{Z}^{d-1} \cap Q_i^{(\textbf{c})}) \times  \{ \delta \}$, $\cP^{(\la)} \cap B(z, \varepsilon)$ is a singleton and this point is put on the $n$-th layer as in property (a).
		\item For each $z \in (\frac{\rho}{2}\mathbb{Z}^{d-1} \cap Q_i^{(\textbf{c})}) \times  \{ \delta \}$, $\cP^{(\la)} \cap B(z, \varepsilon)$ is a singleton and this point is put on the $n$-th layer as before.
	\end{enumerate}

	(c) Aside from the points described above, $\mathcal{P}^{(\la)}$ has no point in 
	$Q_i \times [0, \alpha]$, and also no other point in any downward cone-like grain $\Pi^\downarrow(w')$ for any $w' \in Q_i\times [0,\alpha]$.
	
	After a possible relabeling, we denote by $I := \{1, \ldots K \} $  the indices of the cubes in the discretization of $C_{\text{disc}}\left[ -\log(\la), \log(\la)\right]^{d-1}$ that verify properties (a) to (c).
	Since any cube has positive probability to verify these properties, we deduce that 
	\begin{equation} \label{var pos number Qi}
		\mathbb{E}[K] \geq c \log^{d-1}(\la) .
	\end{equation}
	Let $\mathcal{F}_\lambda$ be the $\sigma$-algebra
	generated by $I$, the positions of points in $W_\lambda \setminus (\bigcup_{i \in I} Q_i^{(\textbf{c})}  \times [0, \alpha])$ and the scores $\xi_{n,k}^{(\infty)}(x,\cP^{(\la)})$ at these points.
	For each $i \in I$, we claim that
	\begin{equation}\label{var pos Qi}
		\Var\bigg[ \sum_{x \in \mathcal{P}^{(\la)} \cap (Q_i \times [0, \alpha])} \xi^{(\infty)}_{n,k}(x, \mathcal{P}^{(\la)}) \Big| \mathcal{F}_\lambda \bigg]=\Var\bigg[ \sum_{x \in \mathcal{P}^{(\la)} \cap (Q_i^{(\textbf{c})} \times [0, \alpha])} \xi^{(\infty)}_{n,k}(x, \mathcal{P}^{(\la)}) \Big| \mathcal{F}_\lambda \bigg] \geq c_0>0.
	\end{equation}
	We justify the last inequality as follows. Either condition (b1) or condition (b2) occurs in $Q_i^{(\textbf{c})} \times [0, \alpha]$ and each with positive probability. Moreover, we notice that $\sum_{x \in \mathcal{P}^{(\la)} \cap (Q_i \times [0, \alpha])} \xi^{(\infty)}_{n,k}(x, \mathcal{P}^{(\la)})$ is larger when (b2) is satisfied. Indeed, this comes from the scaling result \eqref{var pos scaling} which implies that the contribution of points inside $Q_i^{(\textbf{c})}\times [0, \alpha]$ provides a quantity almost $2^{d-1}$ times larger as soon as (b2) is satisfied. 
 
	We recall the equality at \eqref{eq:rewritingZ_iinrescaledworld}. Considering that conditional on $\mathcal{F}_\lambda$, only scores in $\cup_{i \in I} Q_i \times [0, \alpha]$ have any variability, we get that
	\begin{align*}
		\Var[Z_i(\delta_0)\mathds{1}_{A_\la}] & =\Var[\E[Z_i(\delta_0)\mathds{1}_{A_\la}|\cF_\la]+\E[\Var[Z_i(\delta_0)\mathds{1}_{A_\la}|\cF_\la]]\\&\geq \mathbb{E}\left[ \Var\left[Z_i(\delta_0)\mathds{1}_{A_\la} | \mathcal{F}_\lambda  \right] \right]\\
			&= \mathbb{E}\bigg[ \Var\bigg[\sum_{i \in I} \sum_{x \in \mathcal{P}^{(\la)} \cap (Q_i^{(\textbf{c})} \times [0, \alpha])} \xi^{(\infty)}_{n,k}(x, \mathcal{P}^{(\la)}) \Big| \mathcal{F}_\lambda \bigg] \bigg].
	\end{align*}
	Finally, we use the fact that the sums of scores in $Q_i^{(\textbf{c})} \times [0, \alpha]$ and $Q_j^{(\textbf{c})} \times [0, \alpha]$ for $i \neq j$ are independent conditional on $\mathcal{F}_\lambda$ since the scores in $Q_i^{(\textbf{c})} \times [0, \alpha]$ only depend on points in $Q_i \times [0, \alpha]$.
	Thus, we can write
	\begin{align*}
	\Var[Z_i(\delta_0)\mathds{1}_{A_\la}] &\geq \mathbb{E}\bigg[ \sum_{i \in I} \Var\bigg[ \sum_{x \in \mathcal{P}^{(\la)} \cap (Q_i \times [0, \alpha])} \xi^{(\infty)}_{n,k}(x, \mathcal{P}^{(\la)}) \Big| \mathcal{F}_\lambda \bigg] \bigg]\\
	&\geq c_0 \mathbb{E}[K]\\
	&\geq c \log^{d-1}(\la).
	\end{align*}
	where the second inequality comes from \eqref{var pos Qi} applied to each $Q_i$ and the last inequality comes from
	\eqref{var pos number Qi}. This proves the positivity of the limiting variance.

The key in the preceding proof was to construct two configurations of points in each small cube, that occur with a positive probability and lead to two different sums of scores. When counting $k$-faces, we changed the spacing between our points in the $(d-1)$ first coordinate, i.e. by dividing $\rho$ by $2$, to get configurations with different scores. 
In the case of the volume we can get two different configuration by dividing $\delta$ by 2, i.e., by changing the height at which we put our points. This will lead to different scores for each configuration while keeping configurations that occur with positive probability. The proof is otherwise entirely identical and is thus omitted.
%
\subsection{Proof of the central limit theorem}
    To prove the CLT for $Z = Z_\la$, stated in Theorem \ref{thm:theoremeprincipalpolytope}, we follow the idea used in \cite{BR10} where the case $n = 1$ is treated. We first prove the CLT conditional on $A_\la$ thanks to the dependency graph that we built in Section \ref{sec:depgraph}. We then remove the conditioning thanks to a lemma that guarantees that the CLT remains valid for the variables with no conditioning as long as they are close enough in distribution to the conditioned variables.

    We define the random variables $Z'_i $ and $Z'$ in the same way as the variables $Z_i$ and $Z$ respectively, save for the fact that the Poisson point process $\cP_\la$ is replaced by a point process $\cP'_\la$ that has the distribution of $\cP_\la$ conditional on the event $A_\la$, i.e. for $\xi \in \Xi$
     \begin{equation}\label{eq:defZ'etZ'i}
 Z':= \sum_{x \in \cP'_\la} \xi(x, \cP'_\la).
 \end{equation}
Again, for sake of simplicity, the dependency on $\la$ of $Z'$ is made invisible, even though the aim is to derive a limit theorem for $Z'$ and ultimately $Z$ when $\la\to\infty$.

\noindent\textit{Step 1 : application of the central limit theorem for dependency graphs.}
Let us recall the general central limit theorem stated by Rinott in \cite{R94}.
\begin{thm}[Rinott]\label{thm:rinott}
Let $\zeta_i$, $i \in \cV_{\cG}$, be random variables having a dependency graph $ \cG :=  (\cV_{\cG}, \cE_{\cG})$. Set $\zeta = \sum_{i \in \cV_{\cG}} \zeta_i $ and $\sigma^2 = \Var[\zeta]$. Denote the maximal
degree of $\cG$ by $D$ and suppose that $|\zeta_i - \E[\zeta_i]| \leq M $ almost surely for all $i$. Then, for every $x \in \R$,
\begin{equation}\label{eq:ineg rinott}
\left|\P\left( \frac{\zeta - \E[\zeta]}{\sqrt{\Var(\zeta)}} \leq x \right) - \phi(x)\right| \leq \frac{1}{\sqrt{2\pi}} \frac{DM}{\sigma} + 16 \frac{|\cV_{\cG}|^{1/2} D^{3/2}M^2}{\sigma^2} + 10 \frac{|\cV_{\cG}|D^2 M^3}{\sigma^3}
\end{equation}
where $\phi$ is the cumulative distribution function of the $\cN(0,1)$ distribution.
\end{thm}
The underlying dependency graph structure is due to the conditioning on $A_\la$ and leads us to apply Theorem \ref{thm:rinott} to get a CLT for $Z'$.

\begin{lem}\label{prop:TCLZ'}
    There exists $c>0$ such that
    $$\left|\P\left( \frac{Z' - \E[Z']}{\sqrt{\Var(Z')}} \leq x \right) - \phi(x)\right| \leq c \frac{(\log\log(\la))^{9d^2+12(d-1)}}{\log^{(d-1)/2}(\la)}.$$
\end{lem}
\begin{proof}
Exactly as in Section \ref{sec:depgraph}, the dependency graph ${\mathcal G}$ is built over the set ${\mathcal V}_{\mathcal G}=\{1,\ldots,m(T,\delta_0)\}$ when taking $\zeta_i=\sum_{x \in \cP_\la\cap S'_i} \xi(x, \cP'_\la)$ for $\xi\in \Xi$. 
Recall that Lemma \ref{lem:no edge independent} ensures that we have indeed a dependency graph. We also notice that $\zeta = \sum_{i = 1}^{m(T,\delta_0)} \zeta_i = Z'$.
        As intended, we apply Theorem \ref{thm:rinott} to that particular dependency graph. To do so, we need bounds for the number of vertices and the maximal degree of the dependency graph.

        Thanks to Lemma \ref{lem:max degree}, we have for some $c_1 > 0$ and $\la$ large enough $$D \leq c_1(\log\log(\la))^{6(d-1)}.$$
        The same reasoning as for \eqref{eq:big O flat} yields
        $$M\leq c_2 (\log\log(\la))^{3d^2} $$
        for some $c_2 > 0$.
        As $|\cV_{\mathcal G}| = m(T,\delta_0)$, see \cite[Theorem 2.7]{BR10} tells us that
        $$|\cV_{\mathcal G}| \leq c_3 \log^{d-1}(\la) $$
        for $c_3 >0$.
        Note that $\Var[Z'] = \Var[Z | A_\la]$, so $\Var[Z']$ is lower bounded for $\la$ large enough by $c_4 \log^{d-1}(\la)$ with $c_4 > 0$ thanks to
        Lemma \ref{lem:approximation with A lambda} and Theorem \ref{thm:theoreme principal amelio}.

        Using Theorem \ref{thm:rinott} and computing the upper bounds gives us that the dominant term in the right hand side of \eqref{eq:ineg rinott} is the third one, that is $c \frac{(\log\log(\la))^{9d^2 + 12(d-1)}}{\log^{(d-1)/2}(\la)}$ for some $c > 0$.
\end{proof}
\noindent\textit{Step 2: deconditioning.} We now need to extend the CLT from Lemma \ref{prop:TCLZ'} to $Z$, using the fact that $Z$ and $Z'$ are close in distribution and a general result due to Bárány and Vu \cite[Lemma 4.1]{BV07}
 and   stated in Lemma \ref{lem:baranyvu} below, which   asserts that if two sequences of random variables are sufficiently close and if one of the two satisfies a CLT, then the second one does as well.
    \begin{lem}[Bárány, Vu]
        \label{lem:baranyvu}
        Let $(\zeta_\la)_{\la >0}$ and $(\zeta'_\la)_{\la>0}$ be two families of random variables with means
$\mu_\la$ and $\mu'_\la$ , variances $\sigma^2_\la$ and $\sigma'^2_\la$ , respectively. Assume that there are functions
$\varepsilon_1(\la)$, $\varepsilon_2(\la)$, $\varepsilon_3(\la)$, $\varepsilon_4(\la)$, all tending to zero as $\la$ tends to infinity, such that
\begin{enumerate}
    \item $|\mu_\la - \mu'_\la| \leq \varepsilon_1(\la)\sigma_\la$
    \item $|\sigma^2_\la - \sigma'^2_\la| \leq \varepsilon_2(\la)\sigma^2_\la$
    \item for every $x \in \R$, $|\P(\zeta_\la \leq x) - \P(\zeta'_\la \leq x)| \leq \varepsilon_3(\la)$
    \item for every $x \in \R$,
    $$ \left|\P\left(\frac{\zeta'_\la - \mu'_\la}{\sigma'_\la} \leq x \right) - \phi(x) \right| \leq \varepsilon_4(\la).$$
    \noindent Then there exists a constant $c >0$ such that
    \begin{equation}\label{eq:conc baranyvu}
    \left|\P\left(\frac{\zeta_\la - \mu_\la}{\sigma_\la} \leq x \right) - \phi(x) \right| \leq c\sum_{i=1}^4\varepsilon_i(\la) .        
    \end{equation}
\end{enumerate}
\end{lem}
We are now ready to prove that $Z$ satisfies a central limit theorem.
   We need to check the conditions of Lemma \ref{lem:baranyvu} for $\zeta_\la = Z_\la = Z$ and $\zeta'_\la = Z'_\la = Z'$. Condition 4 is verified thanks to Lemma \ref{prop:TCLZ'} with $\varepsilon_4(\la) = c \frac{(\log\log(\la))^{9d^2+12(d-1)}}{\log^{(d-1)/2}(\la)}$. Because $\E[Z'] = \E[Z | A_\la]$ and $\Var[Z'] = \Var[Z | A_\la]$, the conditions 1 and 2 are covered thanks to Lemma \ref{lem:approximation with A lambda} with the choice $\varepsilon_1(\la)=c\log^{-2d^2-\frac{d-1}2}(\la)$ and $\varepsilon_2(\la)=c\log^{-2d^2-(d-1)}(\la)$. Only the third condition remains to be proved.
   For any $x\in \R$, we get
   $$|\P(Z' \leq x) - \P(Z \leq x)| = |\P(Z \leq x) - \P(Z \leq x | A_\la)| \leq \big(\P(Z\leq x | A_\la^c) + \P(Z\le  x |A_\la\big) \P(A_\la^c) \leq 2 \P(A_\la^c).$$
   Lemma \ref{thm:sandwiching et nb pts} then implies that the condition 3 holds with the choice $\varepsilon_3(\la)=c\log^{-4d^2}(\la)$. Applying then Lemma \ref{lem:baranyvu}, we notice that the leading term in the right hand side of \eqref{eq:conc baranyvu} is $\varepsilon_4(\la)$. This proves both central limit theorems in Theorems \ref{thm:theoremeprincipalpolytope} and \ref{thm:theoremeprincipalvolumepolytope}.

\section{Concluding remarks}\label{sec:concrempolytope}
To conclude, we discuss some possible extensions of our results and related open problems.\\
\begin{itemize}

    \item \textit{Intrinsic volumes}. The functionals under investigation are currently limited to the number of $k$-faces and the defect volume of the $n$-th layer. 
    In the case of the unit ball, we obtained in \cite{CQ1} the asymptotics for the defect intrinsic volume of $\conv_n(\cP_\la)$ for any fixed $n$. An interesting perspective would then be to do the same in the case of the convex hull peeling in a (simple) polytope. This would require to decompose the defect intrinsic volume of order $k\in \{1,\ldots,(d-1)\}$ corresponding to the $n$-th layer as a sum of scores and apply stabilization techniques in the vicinity of each vertex of $K$. 
    The method seems definitely more intricate, especially because in the case of the convex hull, the problem remains open for both the expectation and the variance. To the best of our knowledge, only the expectation for the defect mean width, i.e. for $k=1$, has been derived in \cite{S87}.
    ~\\
    \item \textit{Monotonicity.}
    Theorem \ref{thm:theoremeprincipalpolytope} induces a natural monotonicity problem, i.e. how do the constants $c_{n,k,d}$ evolve with $n$ ? We may start with $c_{n,0,d}$, i.e. the expected renormalized number of extreme points on the $n$-th layer.
    We recall that this particular issue was already raised in \cite[Section 6]{CQ1} in the context of the convex hull peeling in the unit ball. It turns out that our simulations show a significant difference between these two models regarding the behaviour of $c_{n,0,d}$.
    While $c_{n,0,d}$ seems to be decreasing with $n$ in the case of the unit ball, it looks like it is increasing up to a certain point and then decreasing  in the case of a polytope. We assert that the two behaviors are in fact consistent with our current knowledge of the convex hull peeling. Indeed, the expected total number of layers estimated by Dalal in \cite{D04} and then more precisely by Calder and Smart in \cite{CS20} is proved to behave like $\Theta(\la^{2/(d+1)})$ in any case, i.e. whether $K$ is smooth or is a polytope. In the particular case of the polytope, we know that the first layers are occupied by (at most) $\Theta(\log^{d-1}(\la))$ Poisson points while the expected total number of Poisson points is $\la \Vol_d(K)$. Consequently, that regime of $\Theta(\log^{d-1}(\la))$ could not be shared by all the $\Theta(\la^{2/(d+1)})$ layers up to the last one. 
    We know that $N_{n,k,\la}$ has to increase and reach, for some $n$ depending on $\la$, a phase transition where it starts to grow polynomially fast. That phase transition should probably occur when $n$ is proportional to $\la^{2/(d+1)}$, i.e. in the regime investigated in \cite{CS20}. This heuristic reasoning would imply the increase of the constant $c_{n,0,d}$ up to the phase transition. After the point when $N_{n,0,\la}$ reaches the polynomial regime, we think that the Markov property of the peeling construction implies that the subsequent layers should behave in the same way as for the convex hull peeling in the unit ball, i.e. we expect a decrease of $N_{n,0,\la}$. This is indeed what we observe in our simulations.
    ~\\
    \item \textit{Other regimes.} In view of the previous point, it appears that  a change in the growth rate of $N_{n,0,\la}$ should occur. We expect the mean of $N_{n,0,\la}$ to grow at most like $\Theta(\log^{d-1}(\la))$ when $n$ is fixed from the start and does not vary with $\la$, because of Theorem \ref{thm:theoremeprincipalpolytope}, then to grow like $\Theta(\la^{\frac{d-1}{d+1}})$ for the last layers, i.e. when $n$ is proportional to $\la^{2/(d+1)}$. This last prediction comes from both the conjecture \eqref{eq:conjCS polytope} stated by Calder and Smart \cite{CS20} and our own simulations. We would then wish to determine, if possible, the specific `time' $n$ when the phase transition occurs and the behavior of the expectation of $N_{n,0,\la}$ when $n=f(\la)$ reaches all the intermediate regimes between $f$ being constant and $f(\la)=c\la^{2/(d+1)}$. In fact, this would mean being able to describe $N_{n,0,\la}$ as a random process depending on time $n=f(\la)$.
    ~\\
    \item \textit{Non simple polytopes.} Our strategy is based on the crucial assumption that the polytope $K$ is simple. This makes it possible to use the scaling transformation introduced in \cite{CY3} and the stabilization methods. In the case of the first layer, the limits for the expectation were proved in 
    \cite{R05a} and \cite{BB93} while the CLTs were proved in \cite{BR10b}, for any convex polytope $K$. The assumption of $K$ being simple was only needed to prove an exact limit for the variance, see \cite{CY3}.
    It remains to be determined whether the techniques from \cite{R05a} and \cite{BB93} can be extended to the subsequent layers. Since the CLTs do not require a precise computation of limit expectations and variances and since we extensively use the construction of a dependency graph relying on Macbeath regions, i.e. the key idea from \cite{BR10b},
    we can reasonably hope that such CLTs would hold for a general polytope $K$. The main task would be to lower bound the variance without the help of Section \ref{sec:rescaling}.
    ~\\
    \item \textit{Depoissonization.} In this paper, the input set has been assumed to be a Poisson point process, i.e. the total number of i.i.d. uniform points in $K$ is Poisson distributed, which involves notably nice independence properties and Mecke's formula. When the input set is a binomial point process, i.e. the total number of points is deterministic, the formula for the limiting expectation of the number of $k$-faces of the convex hull is known to depoissonize \cite{BB93}. We expect that Theorem \ref{thm:theoremeprincipalpolytope} should depoissonize as well, even though there is no mention in \cite{CY3} of a binomial version of the limiting variance in the case of the convex hull.
    Methods of depoissonization
    in the spirit of \cite[Theorem 1.2]{CY1} and \cite[Theorem 1.1]{CY4} could be relevant here.
    ~\\
    \item \textit{Position of the subsequent layers.} Lemma \ref{lemme hauteur max} provides a rough estimate of the distribution tail of the maximal height of the $n$-th layer in the rescaled picture. We did not try to discuss how this height should depend on $n$ even though we expect it to be asymptotically proportional to $n$, in the spirit of the results of Calder and Smart for the parabolic hull peeling \cite[Section 2]{CS20}. This would first require to improve Lemma \ref{lemme hauteur max} by taking $t$ and the constant $c$ therein depending on $n$.
\end{itemize}
    ~\\
    \noindent{\em Acknowledgments}. This work was partially supported by the French ANR grant GrHyDy (ANR-20-CE40-0002) and the Institut Universitaire de France. The authors would also like to thank J.~E.~Yukich for fruitful discussions and helpful comments.
\bibliography{biblio.bib}
~\\
{Pierre Calka, Univ Rouen Normandie, CNRS, Normandie Univ, LMRS UMR 6085, F-76000 Rouen, France\\
pierre.calka@univ-rouen.fr}
~\\
{Gauthier Quilan,
Univ Bretagne Sud, CNRS UMR 6205, LMBA, F-56000 Vannes, France
\\gauthier.quilan@gmail.com}
\end{document}